\documentclass[11pt]{amsart}
\usepackage{amsmath,amsthm,amssymb,mathrsfs}
\usepackage{microtype,xspace}
\usepackage{graphicx}
\usepackage{booktabs}
\usepackage[utf8]{inputenc}
\usepackage{enumitem}
\usepackage{xcolor}
\usepackage{longtable} 
\usepackage{algorithm}
\usepackage[noend]{algpseudocode}
\usepackage{float} 
\usepackage{tablefootnote}
\usepackage{subfigure}
\usepackage{multirow}
\usepackage{stmaryrd}
\usepackage{appendix}
\usepackage{arydshln}
\usepackage{tikz-cd}
\usepackage[normalem]{ulem}
\usepackage{hyperref}
\usepackage{dsfont}
\usepackage{cleveref,cite,url}
\usepackage{hologo} 
\usepackage{comment}
\usepackage{xcolor}
\usepackage{xspace}

\hypersetup{
	colorlinks=true,
	linkcolor = cyan,
	urlcolor  = cyan,
	citecolor = cyan,
	anchorcolor = cyan
}

\newcommand{\ourset}{spartrahedron\xspace}
\newcommand{\Szero}[2]{\ensuremath{\mathcal{S}_{#1,#2}}}
\newcommand{\Sone}[2]{\ensuremath{\mathcal{S}_{#1,#2}^{\ell_1}}}
\newcommand{\Sz}[2]{\ensuremath{\mathcal{S}_{#1,#2}^{\textit{\tiny{Z}}}}}
\newcommand{\Sbs}[2]{\ensuremath{\mathcal{S}_{#1,#2}^{\textit{\tiny{BCP}}}}}
\newcommand{\Sint}[2]{\ensuremath{\mathcal{S}_{#1,#2}^{\textit{\tiny{z}}}}}

\theoremstyle{plain} 
\newtheorem{thm}{Theorem}[section]

\newtheorem{lem}{Lemma}[section]

\newtheorem{cor}{Corollary}[section]

\theoremstyle{definition}
\newtheorem{dfn}[thm]{Definition} 
\newtheorem{ex}[thm]
{Example}

\theoremstyle{remark}

\newtheorem*{exe*}{Exercise}

\newcommand{\vX}{{\mathbf{X}}}

\newcommand{\cA}{{\mathcal{A}}}

\newcommand{\cD}{{\mathcal{D}}}

\newcommand{\RR}{\mathbb{R}}
\renewcommand{\SS}{{\mathbb{S}}}

\newcommand{\NN}{{\mathbb{N}}}
\newcommand{\vzero}{\mathbf{0}}

\newcommand{\st}{\text{\, s.t.\, }}

\newcommand{\X}{\boldsymbol{X}}

\DeclareMathOperator{\rank}{rank}
\DeclareMathOperator{\corank}{corank}
\DeclareMathOperator{\codim}{codim}

\DeclareMathOperator{\val}{val}
\DeclareMathOperator{\conv}{conv}
\DeclareMathOperator{\diag}{diag}
\DeclareMathOperator{\tr}{Tr}
\DeclareMathOperator{\extr}{extr}
\DeclareMathOperator{\ACQ}{ACQ}
\DeclareMathOperator{\supp}{supp}

\begin{document}
	
	\title[Sparse PCA, Regression, and QCQP via the \MakeUppercase{\ourset}]
	{Solving Sparsity Constrained PCA, Regression, and QCQP via the \MakeUppercase{\ourset}}
	\author{Diego Cifuentes}
	\address{}
	\email{diegcif@gmail.com}
	
	\author{Zhuorui Li}
	\address{Georgia Institute of Technology \ Atlanta, GA, USA}
	\email{zli966@gatech.edu}
	
	\begin{abstract}
		Sparsity is a fundamental modeling principle in statistics, signal processing, and data science. 
		However, optimization with sparsity constraints is notoriously difficult. 
		We introduce a new convex relaxation framework for {sparse quadratically constrained quadratic programs} (QCQPs), a class that subsumes sparse regression, sparse principal component analysis (PCA), and related problems. 
		Our approach is based on a novel convex cone, the {\ourset}, which exactly characterizes sparsity at the matrix level. 
		This leads to a semidefinite programming (SDP) relaxation that is tight whenever its solution is rank-one, providing a simple certificate of global optimality. 
		We establish theoretical guarantees, including approximation bounds and exactness regions for sparse PCA and sparse ridge regression, as well as a general stability result under perturbations. 
		Numerical experiments on sparse PCA, sparse regression, RIP constant estimation, and sparse canonical correlation analysis (CCA) demonstrate the practical success of our methods. 
	\end{abstract}
	
	\maketitle
	
	\section{Introduction}
	
	Sparsity plays a central role in modern statistics, signal processing, and data science.
	By concentrating modeling capacity on a small set of salient variables, sparse models improve interpretability and generalization in high dimensions.
	In inverse problems, compressed sensing theory demonstrates that sparse signals can be reconstructed exactly or stably from highly undersampled measurements, enabling accelerated imaging and other acquisition-limited applications.
	In regression analysis, sparsity is widely promoted by techniques such as the lasso.
	Sparse principal component analysis (PCA) produces components with few nonzero loadings, retaining much of the explanatory power of classical PCA while dramatically improving interpretability.
	Across domains, sparsity acts as a powerful inductive bias that enables statistical guarantees and efficient computation.
	
	Despite extensive research, enforcing sparsity remains a challenging task.
	Practical methods often rely on heuristic surrogates such as $\ell_1$ penalization, which offer limited guarantees and are highly sensitive to regularization choices.
	This paper introduces a new framework for handling sparsity in a broad class of optimization problems---including sparse regression and sparse PCA---using convex optimization tools that provide rigorous guarantees on solution quality.
	
	Our starting point is the general family of sparsity-constrained {quadratically constrained quadratic programs} (QCQPs):
	\begin{equation}\label{sparse_qcqp}\tag{\small{P}}
		\min_{x\in\RR^n}\quad x^T C x
		\quad\st\quad \|x\|_0\leq k,\quad
		x^T A_i x = b_i,\quad i\in [m],
	\end{equation}
	where $C,A_i\in \SS^n$ are symmetric matrices, $b_i \in \RR$ are scalars, $k \in \NN$ is an integer, and $\|\cdot\|_0$ denotes the number of nonzero entries.
	When the sparsity constraint is removed, the problem reduces to a standard QCQP, a classical yet challenging optimization class with numerous applications \cite{nemirovski1999maximization, ben2001lectures, bao2011semidefinite, aholt2012qcqp, luo2010semidefinite, cifuentes2021convex, cifuentes2022local}.
	While QCQPs are already NP-hard in general, the addition of a sparsity constraint makes the problem significantly more difficult: sparse QCQPs are NP-hard even when $m=0$ \cite{bienstock1996computational, natarajan1995sparse, pardalos1991quadratic}, whereas QCQPs remain tractable for $m \leq 1$ \cite{sturm2003cones, pong2014generalized, ben2001lectures}.
	
	Two important instances of \eqref{sparse_qcqp} illustrate its broad relevance.
	The first is {sparse PCA}, which seeks a sparse unit vector $x \in \RR^n$ that maximizes correlation with a covariance matrix $\Sigma \in \SS^n$:
	\begin{equation}\label{sparse_pca}\tag{\small{P-spca}}
		\max_{x\in\RR^n}\quad
		x^T \Sigma x\quad\st\quad \|x\|_0\leq k,\quad \|x\|_2^2 = 1 
	\end{equation}
	More generally, the $\ell$-sparse PCA problem seeks $\ell$ orthonormal sparse vectors,
	and can be formulated as a sparse QCQP as follows:
	\begin{equation*}
		\max_{X\in\RR^{n\times \ell}}\quad
		\Sigma \bullet X X^T \quad\st\quad \|X\|_0\leq k,\quad X^T X = I_\ell 
	\end{equation*}
	where $A \bullet B = \tr(B^T A)$ denotes the trace inner product.
	
	The second is {sparse regression}.
	Given $A \in \RR^{m\times n}$ and $y \in \RR^m$, the goal is to estimate $x \in \RR^n$ minimizing $\|Ax-y\|_2^2$ subject to sparsity.
	We focus on a sparse variant of ridge regression, which adds an $\ell_2$ penalty:
	\begin{equation}\label{sparse_lr}\tag{\small{P-sridge}}
		\min_{x\in \RR^n} \quad \tfrac 1 m \|A x - y\|^2_2 + \alpha \|x\|^2_2
		\quad\st\quad \|x\|_0 \leq k,
	\end{equation}
	for a fixed regularization parameter $\alpha \geq 0$.
	Sparse regression is a cornerstone in statistics and data analysis,
	often approached through lasso.
	
	\subsection*{Our approach.}
	We propose a new semidefinite relaxation for sparse QCQPs.
	Convex relaxations based on semidefinite programming (SDP) have been extensively studied for QCQPs, but far less so for their sparse counterparts.
	Although $\ell_1$ regularization offers a proxy for sparsity, such methods typically lack theoretical guarantees.
	By contrast, our relaxation provides provable approximation bounds---for both sparse PCA and sparse regression---and is grounded in a new convex cone, the \emph{\ourset}.
	
	The $k$-\ourset in $\SS^n$ is defined as
	\begin{equation}
		\label{eq:ourset}
		\Szero{n}{k} := \{X \in \mathbb{S}^n:
		k\diag X \succeq X \succeq 0\}
	\end{equation}
	where $\diag X$ retains only the diagonal entries of $X$, and $A \succeq B$ denotes positive semidefiniteness of $A-B$.
	This set captures sparsity exactly:
	\begin{equation}\label{eq:property}
		x x^T \in \Szero{n}{k} \quad \iff \quad \|x\|_0 \leq k.
	\end{equation}
	This leads to the following SDP relaxation for \eqref{sparse_qcqp}:
	\begin{equation}\label{sparse_sdp}\tag{\small{Q}}
		\begin{aligned}
			\min_{X \in \mathbb{S}^n}\quad C \bullet X
			\quad\text{s.t.}\quad A_i \bullet X = b_i,\;  i \in [m],\quad
			X\in \Szero{n}{k}
		\end{aligned}
	\end{equation}
	The relaxation is \emph{exact} or \emph{tight} when any optimal solution of \eqref{sparse_sdp} is a rank one matrix $X = x x^T$, in which case $x$ is optimal for \eqref{sparse_qcqp}. 
	The SDP \eqref{sparse_sdp} can be viewed as a sparsity-aware analogue of the classical Shor relaxation for QCQPs.
    The induced SDP relaxation \eqref{sparse_sdp} exhibits robustness to small perturbations when the relaxation is tight; see \Cref{thm:sparse_local_stability} for details.
	We also explore other conic relaxations that might be tighter than~\Szero{n}{k}. In particular, we introduce the SDP~\eqref{sparse_sdp_plus}, which yields stronger guarantees than~\eqref{sparse_sdp} at the expense of increased computational cost.
	
	\subsection*{Contributions and Organization}
	
	The main contribution of this paper is the introduction of the \emph{\ourset} and its use in convex relaxations for sparse QCQPs.
	Our framework applies broadly, encompassing problems such as sparse PCA, sparse regression, estimation of restricted isometry property (RIP) constants, and sparse canonical correlation analysis (CCA).
	
	Our specific contributions are as follows:
	\begin{itemize}
		\item \emph{Geometric foundations (\Cref{section:ourset}).}
		We study the geometry of the \ourset, compare it to other convex relaxations, and prove its fundamental sparsity property~\eqref{eq:property}.
		We also show that the relaxation~\eqref{sparse_sdp} is always exact when $n=3$ and $k=2$.
		
		\item \emph{Stability under perturbations (\Cref{section:stability}).}
		We prove that our relaxation is robust: if it is exact for a given instance, it remains exact under small data perturbations.
		This extends the local exactness results of~\cite{cifuentes2022local} to the sparse setting.
		
		\item \emph{Sparse PCA (\Cref{section:spca}).}
		We establish that the approximation ratio of our relaxation is bounded by $\min\{k,\, n/k,\, r\}$, where $r$ is the rank of the objective matrix.
		We further identify an exactness region for the SDP and apply these results to spiked Wigner and Wishart models.
		
		\item \emph{Sparse ridge regression (\Cref{section:ridge}).}
		We prove an $O(k)$ bound on the approximation ratio and identify an exactness region in the overdetermined regime.
		
		\item \emph{Numerical experiments (\Cref{section:experiments}).}
		We provide experiments on sparse PCA, sparse regression, RIP estimation, and sparse CCA.
		Across all tasks, our SDP consistently outperforms both heuristic methods and existing SDP-based approaches.
	\end{itemize}
	
	\subsection*{Related Works}
	
	Given their broad practical reach, sparse QCQPs have motivated a wide range of algorithms. There are three main classes: (i) approximation and heuristic methods that scale well but may lack guarantees; (ii) mixed-integer formulations (MIP/MINLP) which may provide global optimality certificates at higher computational cost, and (iii) convex relaxations that offer tractable bounds and, under suitable conditions, exact recovery.
	
	\textit{Approximation (heuristic) methods.}
	A variety of first-order methods have been developed for special cases of sparse QCQP. Inspired by compressed sensing and sparse regression, thresholding-based approaches \cite{journee2010generalized, hein2010inverse,yuan2013truncated,luss2013conditional,yuan2018gradient} have been proposed; beyond simple truncation rules, greedy schemes \cite{tropp2007signal,bahmani2013greedy,tropp2004greed,xie2020scalable} iteratively expand and prune the support. These methods are computationally efficient, and many enjoy guarantees of convergence to a first-order stationary point; however, most existing algorithms only consider simple constraint structures such as a single sparsity constraint coupled with an $\ell_2$-norm bound, so extending to general sparse QCQPs with multiple quadratic constraints is substantially more challenging. Moreover, they typically do not attain the global optimum and do not furnish optimality certificates. The following approaches are developed to find and certify 
	exact solutions.
	
	\textit{Mixed-integer programming.}
	With advances in mixed-integer programming (MIP), numerous methods have been proposed for special cases of \eqref{sparse_qcqp}, notably portfolio selection, sparse regression and sparse PCA. MIP and mixed-integer SDP (MISDP) formulations \cite{gally2018framework,li2025exact,bertsimas2022solving,dey2022using} and Branch-and-bound algorithms \cite{moghaddam2005spectral,berk2019certifiably}  have been employed to solve sparse PCA. In the absence of quadratic constraints, mixed-integer formulations \cite{zheng2014improving} and branch-and-bound algorithms \cite{bertsimas2009algorithm} have been developed for sparse quadratic programs which have applications in portfolio and subset selection. Sparse regression is another important special case of \eqref{sparse_qcqp} where MIP techniques can be applied \cite{bertsimas2020sparse,bertsimas2021sparse,atamturk2025rank}. For the general formulation \eqref{sparse_qcqp}, to the best of our knowledge there is no broadly applicable mixed-integer programming framework; existing MIP techniques are tailored to special cases with additional structure (e.g., a single quadratic constraint, convexity, or application-specific modeling). Developing a unified, scalable MIP approach for \eqref{sparse_qcqp} therefore remains an open challenge. While some of these approaches can certify global optimality for specific applications, their worst-case computational complexity is exponential, which limits scalability in practice. To obtain more tractable algorithms, the aforementioned mixed-integer formulations are often relaxed to convex optimization problems; we discuss these convex relaxations in the next section.
	
	\textit{Convex Relaxation.}
	Numerous convex relaxations have been proposed for special cases of \eqref{sparse_qcqp}. We first consider semidefinite relaxations: for sparse PCA, d’Aspremont et al. \cite{d2004direct} introduced an SDP using an $\ell_1$-based surrogate for cardinality, with subsequent work establishing optimality conditions and approximation guarantees \cite{d2008optimal,d2014approximation}. Although devised for sparse PCA, this framework extends to broader instances of \eqref{sparse_qcqp}, such as SVM formulations \cite{chan2007direct}. In this paper, we investigate a different SDP relaxation tailored to \eqref{sparse_qcqp}. A second line of convexification arises by relaxing mixed-integer models which yield tractable formulations that have been applied to sparse PCA \cite{dey2022using,li2025exact,bertsimas2022solving}, sparse quadratic programs \cite{zheng2014improving}, and sparse regression \cite{bertsimas2020sparse,bertsimas2021sparse,atamturk2025rank}. However, even for special structures, prior works neither systematically characterize when these relaxations are exact, nor have concise way to certify the global optimality. In this paper, our SDP relaxation gives the global optimal solution of \eqref{sparse_qcqp} if and only if the solution is of rank-one, which is easy to check. Furthermore, we establish sufficient conditions ensuring exactness for two concrete applications—sparse PCA and sparse regression—and further derive a general stability result for \eqref{sparse_qcqp} that quantifies robustness under perturbations.
	
	\subsection*{Notation}
	Let $\RR^n$ be the space of real vectors of length~$n$, $\mathbb{S}^k$ the space of $k \times k$ real symmetric matrices, and $\mathbb{S}^k_+$ the set of $k \times k$ PSD matrices.  Given a vector $x\in\RR^n$, $\|x\|_0$ denotes the the number of nonzero entries; $\|x\|_1$ denotes the $\ell_1$-norm, and $\|x\|$ denotes the $\ell_2$-norm; Let \(S \subseteq [n]\). Define the restricted norm \(\|x\|_{S} := \big(\sum_{i \in S} x_i^{2}\big)^{1/2}\). For a matrix \(X\), let \(X_{S}\) denote the principal minor of \(X\) indexed by \(S\), and for a vector \(x\), let \(x_{S}\) denote the subvector consisting of the entries of \(x\) indexed by \(S\). Given $X \in \mathbb{S}^k$, the notation $X \succeq 0$ means that $X$ is PSD. Given a matrix $X \in \mathbb{S}^k$, we let $\diag(X) \in \mathbb{S}^k$ be a matrix whose diagonal is the same as $X$ and other entries are all zeros; let $\|X\|_1 = \sum_{1\leq i,j\leq k} |X_{ij}|$; let $\|X\|_\infty = \max_{1\leq i,j\leq k} |X_{ij}|$; and $\|X\|$ denotes the spectral norm.

	\section{The \ourset and other convex relaxations}\label{section:ourset}

	For the study of sparse QCQPs,
	we need to work with quadratic functions restricted to sparse vectors~$x$.
	This motivates the following nonconvex set:
	\[
	Q_{n,k} = \{xx^T\in\SS^n: \|x\|_0\leq k\}
	\]
	We are interested in obtaining tractable convex relaxations of $Q_{n,k}$.
	In this section we study three different such relaxations.
	The first convex relaxation is the \ourset $\Szero{n}{k}$.
	The second relaxation relies on the $\ell_1$-norm:
	\[
	\Sone{n}{k} =
	\{X\in \mathbb{S}^n : \|X\|_1 \leq k \tr X, \; X \succeq 0\}
	\]
	The third one is a variant of the method from \cite{li2025exact}:
	\begin{equation*}
		\begin{split}
			\Sz{n}{k} = \{X\in\SS^n:\; 
			&\exists z\in \RR^n,\st \|X_{i,:}\|^2\leq X_{ii}z_i,\;\|X_{i,:}\|_1^2\leq kX_{ii}z_i,\\
			&0\leq z_i\leq \tr(X), i\in[n], \sum_{i=1}^n z_i = k \tr(X),X\succeq 0\}
		\end{split}
	\end{equation*}
    The above relaxation motivates the following SDP formulation, 
    \begin{equation}\label{sparse_sdp_plus}\tag{\small{\(\text{Q}^+\)}}
		\min_{X \in \mathbb{S}^n}\quad C \bullet X
		\quad\text{s.t.}\quad A_i \bullet X = b_i,\;  i \in [m],\quad
		X\in \Szero{n}{k}\cap\Sz{n}{k},
	\end{equation}
    and we show later that \eqref{sparse_sdp_plus} yields a strictly tighter relaxation than \eqref{sparse_sdp}.
    
    This section explains geometric properties of these relaxations. We also discuss how such cones can be used for sparse QCQPs.
	
	\subsection{Comparison of the cones}
	
	We start by showing property \eqref{eq:property}, which implies that $\Szero{n}{k}$ is a convex relaxation of $Q_{n,k}$.
	
	\begin{lem} \label{lem:cardinality}
		For a vector $x\in \RR^n$, then $x x^T \in \Szero{n}{k}$ if and only if $\|x\|_0\leq k$. 
		In particular,
		\(
		\conv(Q_{n,k}) \subseteq \Szero{n}{k}
		\) 
	\end{lem}
	
	\begin{proof}
		Let $X:=xx^T$.
		Let us show that $k \diag X\succeq X$ if and only if $\|x\|_0 \leq k$.
		
		Assume first that $\|x\|_0 > k$.
		Let $y \in \RR^n$ with entries
		$y_i = \frac{1}{x_i}$ for $i \in \supp(x)$
		and $y_i = 0$ for $i \notin \supp(x)$.
		Then
		\[
		y^T (k \diag X) y =
		k\sum_{i=1}^n (x_iy_i)^2
		= k \|x\|_0
		< \|x\|_0^2 = (y^Tx)^2
		= y^T X y
		\]
		and hence \(k \diag X \not\succeq X\).
		
		Assume next that $\|x\|_0 \leq k$.
		Let $y \in \RR^n$ be arbitrary.
		Let \(z = x \circ y\) be their Hadamard product.
		Let \(\mathds{1}_x\) be the indicator vector of $\supp(x)$.
		Then
		\[
		y^T (k \diag X) y
		= k \|z\|_2^2
		\geq \|\mathds{1}_x\|_2^2 \cdot \|z\|_2^2 \geq (\mathds{1}_x^T z)^2 = (y^Tx)^2
		= y^T X y
		\]
		and hence \(k \diag X\succeq X\).

		It follows that $Q_{n,k} \subseteq \Szero{n}{k}$,
		so by convexity
		$\conv(Q_{n,k}) \subseteq \Szero{n}{k}$.
	\end{proof}

	The next theorem characterizes the relationship between $\Szero{n}{k}$ is and \(\conv Q_{n,k}\) when~$k=2$.
	
	\begin{thm}\label{thm:special_case_n3_k2}
		We have \(\conv(Q_{3,2}) = \Szero{3}{2}\) and
		\( \conv(Q_{n,2}) \subsetneq \Szero{n}{2}\) for $n \geq 4$.
	\end{thm}
    
	\begin{proof}
		See \Cref{append:thm_special_case_k2}
	\end{proof}
	
	The following theorem indicates that \(\conv(Q_{n,k}) \subseteq \Sz{n}{k}\)
	\begin{lem}\label{lem:Sz_rank_one}
		For a vector \(x \in \RR^n\), then \(xx^T\in \Sz{n}{k}\) if and only if \(\|x\|_0 \leq k\).
	\end{lem}
    
	\begin{proof}
		See \Cref{append:lem_Sz_rank_one}.
	\end{proof}
	
	We next show that \(\Sz{n}{k} \subseteq \Sone{n}{k}\). Consequently, $\Szero{n}{k}^{\ell_1}$ is also a relaxation of~$Q_{n,k}$.
	
	\begin{lem}\label{lem:Q_subset_Sone}
		We have that
		\(
		\Sz{n}{k} \subseteq \Sone{n}{k} 
		\) 
	\end{lem}
	\begin{proof}
		Suppose \(X \in \Sz{n}{k}\), it is enough to show that \(\|X\|_1\leq k\tr(X)\). Observe that
		\begin{equation*}
			\|X\|_1 \leq \sum_{i=1}^n\sqrt{kX_{ii} z_i}\leq \sqrt{k} \sqrt{\textstyle \sum_{i=1}^n X_{ii}} \sqrt{\textstyle\sum_{i=1}^n z_i} \leq k \tr(X)
		\end{equation*}
		Hence, it is true that $\Sz{n}{k} \subseteq \Sone{n}{k}$.
	\end{proof}
	
	We proceed to compare $\Szero{n}{k}$ with $\Sone{n}{k}$ and $\Sz{n}{k}$.
	The next example and theorem indicate that $\Sone{n}{k}$ and $\Sz{n}{k}$ are not better than $\Szero{n}{k}$.
	
	\begin{ex}[$n=3,k=2$]
		In order to visualize $\Szero{3}{2}$, we will restrict ourself to the 3-dimensional hyperplane
		\[
		H = \Big\{\Big(\begin{smallmatrix}
			1 & a & b\\
			a & 1 & c\\
			b & c & 1
		\end{smallmatrix}\Big): a,b,c \Big\}\in \SS^n
		\]
		\Cref{fig:comparison_of_convex_sets} shows that $\Szero{3}{2}\cap H \subseteq \Sz{3}{2}\cap H\subseteq \Sone{3}{2}\cap H$.
		\begin{figure}[t]
			\centering
			
			\subfigure[$\Szero{3}{2} \cap H$]{
				\label{fig:k_sparse}
				\includegraphics[width=0.25\linewidth]{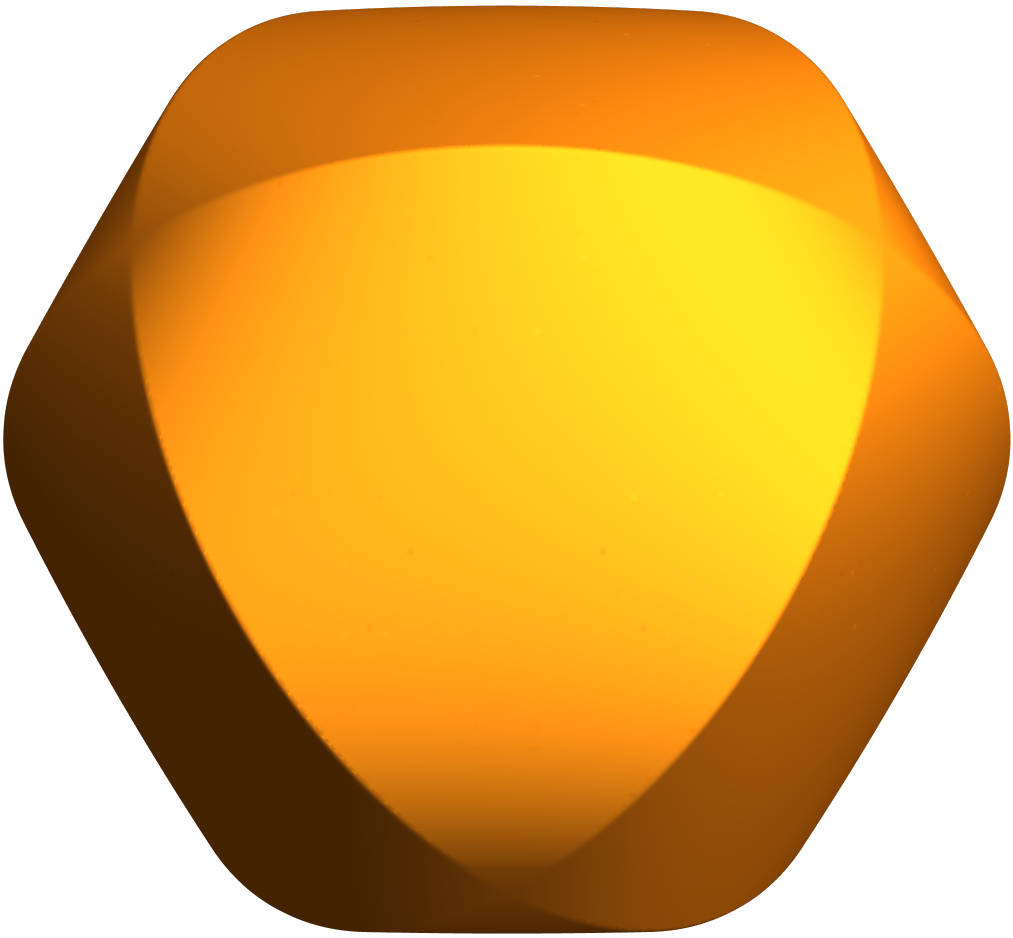}}
			\hspace{0.5em}
			\begin{minipage}[c]{0.02\linewidth}
				\centering \raisebox{15ex}{\Large$\subseteq$}
			\end{minipage}
			\hspace{0.5em}
			\subfigure[$\Sz{3}{2} \cap H$]{
				\includegraphics[width=0.25\linewidth]{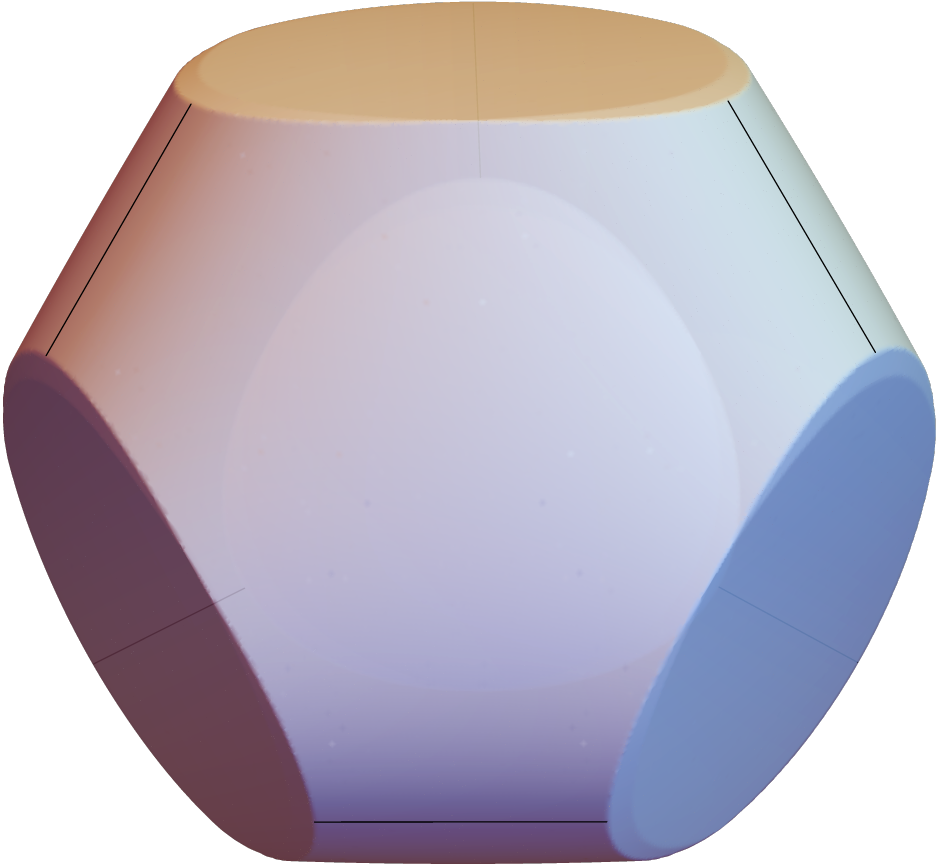}}
			\hspace{0.5em}
			\begin{minipage}[c]{0.02\linewidth}
				\centering \raisebox{15ex}{\Large$\subseteq$}
			\end{minipage}
			\hspace{0.5em}
			\subfigure[$\Sone{3}{2} \cap H$]{
				\includegraphics[width=0.25\linewidth]{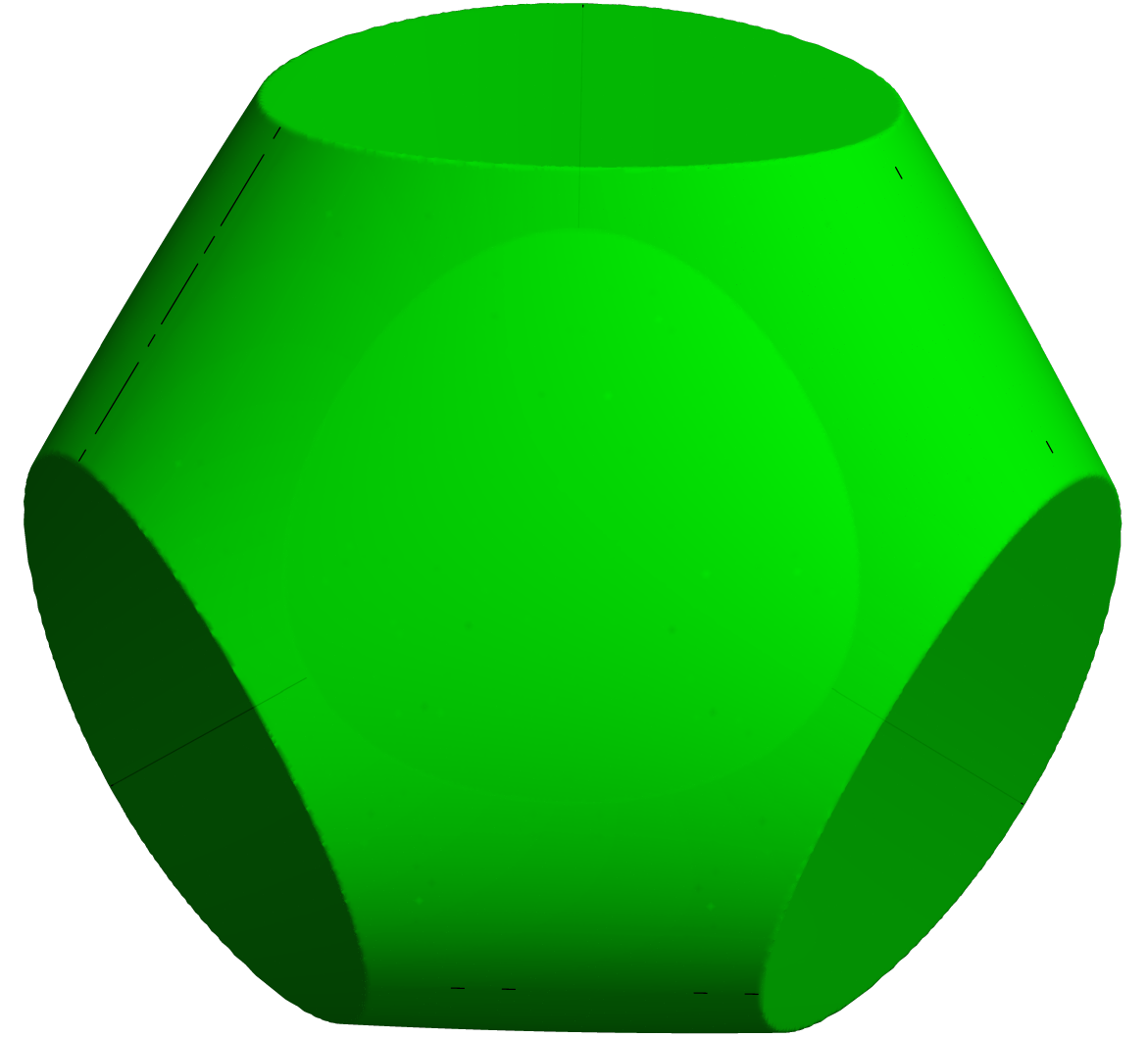}}
			
			\caption{Comparison of $\Szero{3}{2}$, $\Sone{3}{2}$ and $\Sz{n}{k}$}
			\label{fig:comparison_of_convex_sets}
		\end{figure}
	\end{ex}
	
	\begin{thm}\label{thm:Sone_notsubset_Szero}
		For $2 \leq k < n$,
		there exists rank-one matrices $x x^T \in \Sone{n}{k}$ with $\|x\|_0 > k$.
		In particular,
		\(\Sone{n}{k} \not\subset \Szero{n} {k}\)
		and
		\(\conv(Q_{n,k})\subsetneq \Sone{n}{k}\).
	\end{thm}
	\begin{proof}
		Let \(x = (\epsilon,1,2,\dots,k,0\dots,0)^T\in\RR^n\), with $\epsilon = 1/k^2$.
		By \Cref{lem:cardinality},
		$x x^T$ does not lie in \(xx^T\notin \Szero{n}{k}\) nor in \(\conv Q_{n,k}\).
		Let us show that \(xx^T\in \Sone{n}{k}\). Notice that
		\[
		\|xx^T\|_1 = \tfrac{(k+1)^2k^2}{4} + \epsilon k(k+1) + \epsilon^2
		\]
		and
		\[
		k\tr(xx^T) = \tfrac{k^2(k+1)(2k+1)}{6} + k\epsilon^2.
		\]
		Since \(\epsilon = 1/k^2\) and $k\geq 2$, we have
		\[
		k\tr(xx^T) - \|xx^T\|_1 = \tfrac{k^2(k^2-1)}{12} - \tfrac{1}{k}\Big(1 + \tfrac{1}{k^3}\Big)\geq 0
		\]
		and hence \(xx^T \in \Sone{n}{k}\).
	\end{proof}

    \Cref{lem:cardinality} and \Cref{lem:Sz_rank_one} show that $\Szero{n}{k}$ and $\Sz{n}{k}$ perfectly agree with $Q_{n,k}$ when restricted to rank-one matrices.
    However, as shown next, $\Sz{n}{k}$ is not robust against noise:
    it does not differentiate the rank one matrix $x_1 x_1^T$, where \(\|x_1\|_0=k\), 
    from the rank-two matrix \(x_1x_1^T + \epsilon  x_2x_2^T\).
    In contrast, $\Szero{n}{k}$ is robust against small perturbations.
    
	\begin{thm}\label{thm:Sz_perturbation}
		Let \(2k<n\), and let $x \in \RR^n$ with $\|x\|=1$, \(\|x\|_0=k\) and not all its non-zero coordinates identical.
        Then
		\begin{enumerate}
			\item \label{eq:compare_ourset}
            \(x x^T + \epsilon \, w w^T \notin \Szero{n}{k}\)
            for any $w \!\in\! \RR^n$, $\|w\|=1$, $\|w\|_0\!=\!n$ and any $\epsilon \!>\! 0$.
			\item\label{eq:compare_Sz}
            Suppose that \(w\) with $\|w\|=1$, \(\|w\|_0=n\), satisfies \(\sum_{i\in\supp(x)}(1 - |\tfrac{w_i}{x_i}|)^2 > n - k\), then we have
            \(x x^T + \epsilon \, w w^T \in \Sz{n}{k}\)
            $\|w\|_0 \!=\! n$ for small enough $\epsilon \!>\! 0$.
		\end{enumerate}
	\end{thm}

	\begin{proof}
		See \Cref{append:thm_Sz_perturbation}.
	\end{proof}
	
	The next example and theorem indicate that $\Szero{n}{k}$ is not better than $\Sone{n}{k}$ and $\Sz{n}{k}$ either.
	
	\begin{ex}[$\Szero{4}{2}\not\subset \Sone{4}{2}$]
		The following matrix lies in $\Szero{4}{2}$ but not in~$\Sone{4}{2}$ and $\Sz{n}{k}$:
		\[
		\Bigg(\begin{smallmatrix}
			3 & 0 & 2 & -2\\
			0 & 3 & 2 & 2\\
			2 & 2 & 3 & 0\\
			-2 & 2 & 0  & 3
		\end{smallmatrix}\Bigg)
		\]
	\end{ex}
	
	\begin{thm}\label{thm:Szero_notsubset_Sone}
		When \( n = q + 1\) where \(q\) is a prime power such that \(q \equiv 1 (\bmod 4)\) and \(2\leq k < \sqrt{n-1}+1\), we have \(\Szero{n}{k} \not\subset \Sone{n}{k}\).
	\end{thm}
	\begin{proof}
		See \Cref{append:thm_Szero_notsubset_Sone}.
	\end{proof}

	\subsection{Extreme rays and dual cones}
	We proceed to investigate the extreme rays of different convex relaxations of $Q_{n,k}$.
	
	\begin{thm}\label{thm:extreme_ray_of_Q}
		Denoting $\extr K$ the set of extreme rays of~$K$, then
		\begin{align*}
			\extr(\conv Q_{n,k})
			&= Q_{n,k} \setminus 0,\\
			\{X \in \extr \Szero{n}{k} : \rank X = 1\}
			&= Q_{n,k} \setminus 0,\\
			\{X \in \extr \Sz{n}{k} : \rank X = 1\}
			&= Q_{n,k} \setminus 0,\\
			\{X \in \extr \Sone{n}{k} : \rank X = 1\}
			&\supsetneq Q_{n,k} \setminus 0, \text{ for } 2 \leq k < n.
		\end{align*}
	\end{thm}
	\begin{proof}
		For any convex relaxation $K \supseteq Q_{n,k}$ we have that
		$\extr K \supseteq Q_{n,k} \setminus 0$,
		since $Q_{n,k}\setminus 0 \subseteq \extr \SS^n_+ = \{ x x^T : x \neq 0 \}$ .
		It is clear that $\conv Q_{n,k}$ does not have more extreme rays.
		$\Szero{n}{k}$ and $\Sz{n}{k}$ may have more extreme rays,
		but only of higher rank due to \Cref{lem:cardinality} and \Cref{lem:Sz_rank_one}.
		$\Sone{n}{k}$ does have other rank-one extreme rays due to \Cref{thm:Sone_notsubset_Szero}.
	\end{proof}
	
	The following theorem shows that \Szero{n}{k} may have extreme rays outside of \(Q_{n,2}\), of rank two. 
	
	\begin{thm}\label{thm:extreme_ray_special}
		Let \(n\geq4\) and \(2\leq k< n\).
		Let \(X = u_1u_1^T + u_2u_2^T\) where \(\|u_1\|=\|u_2\|=1\), \(\langle u_1,u_2\rangle=0\), and \(u_1,u_2\) satisfy
		\begin{equation}\label{eq:extreme_ray_special}
            \begin{gathered}
				\sum_{i=1}^n(u_1)_i(u_2)_i\big((u_1)^2_i+(u_2)^2_i\big) = 0,\quad
				\sum_{i=1}^n(u_1)_i^4 = \sum_{i=1}^n(u_2)_i^4 < \frac{1}{k},\quad \\
				\sum_{i=1}^n\big((u_1)_i^2 + (u_2)_i^2\big)^2 = \frac{2}{k},\quad (u_1)_i^2 + (u_2)_i^2 = (u_1)_j^2 + (u_2)_j^2,\forall i,j\in [n]
            \end{gathered}
		\end{equation}
		If \(X\notin \conv(Q_{n,k})\), then \(X\) is an extreme ray of \Szero{n}{k}.
	\end{thm}
    \begin{proof}
        See \Cref{append:thm_extreme_ray_special}.
    \end{proof}
	
	Then we show an example of rank-two extreme rays with \(n>k=2\).
	
	\begin{ex}
		The following matrix is an extreme ray in \Szero{4}{2},
		\[
		X:= \left(\begin{smallmatrix}
			\sqrt{5} & 1 & 2& 0\\
			1 & \sqrt{5} & 0 & 2\\
			2 & 0 & \sqrt{5} & -1\\
			0 & 2 & -1 & \sqrt{5}
		\end{smallmatrix}\right)
        = 2\sqrt{5}(u_1u_1^T + u_2u_2^T)
		\]
		where 
        \small{
		\begin{equation}
				u_1 = \tfrac{1}{2\sqrt{5-\sqrt{5}}}(2, 2, \sqrt{5}-1, \sqrt{5}-1)^T,
				u_2 = \tfrac{1}{2\sqrt{5+\sqrt{5}}}(2, -2, \sqrt{5}+1, -(\sqrt{5}+1))^T
		\end{equation}
        }
		In \Cref{lem:sdd_factorization} we check that \(X\notin \conv(Q_{4,2})\).
		It can be checked that \eqref{eq:extreme_ray_special} holds,
		so \(X\) is an extreme ray of \Szero{4}{2}.
	\end{ex}
	
	We proceed to investigate the dual cones. 
	The following is shown in \Cref{append_dual_cones}:
	\begin{align}\label{eq:dualcone}
		(\Szero{n}{k})^* &= \{Y + k\diag(Z) - Z: Y \succeq 0, Z \succeq 0\}\\
		(\Sone{n}{k})^* &= \{Y + \alpha kI_n - Z: \alpha\geq 0, Y \succeq0, Z\in\SS^n, \|Z\|_\infty \leq \alpha\}\\
		\conv(Q_{n,k})^* &= \{Y: Y_{S}\succeq 0, |S|\leq k, Y\in\SS^n\}
	\end{align}
	Some properties on the distance between  $\conv(Q_{n,k})^*$ and PSD cone are analyzed in \cite{blekherman2022sparse,blekherman2022hyperbolic}.
	We next find some of the extreme rays of these dual cones.
	\begin{thm}\label{thm:extreme_ray_of_ourset_dual}
		Let \(X = xx^T\) where \(\|x\|_0 \leq k\).
		Then \(X\) is an extreme ray in \( \conv(Q_{n,k})^*\),
		\((\Szero{n}{k})^*\) and \((\Sone{n}{k})^*\)
	\end{thm}
	
	\begin{proof}
		See \Cref{append:thm_extreme_ray_dual} for the case of \( \conv(Q_{n,k})^*\).
		The statement for the other cones follow from the containment
		\[
		(\Szero{n}{k})^* \cup  (\Sone{n}{k})^* \subseteq \conv(Q_{n,k})^*.\qedhere
		\]
	\end{proof}

	\subsection{Relaxations of sparse QCQPs}
	
	Problem \eqref{sparse_qcqp} can be equivalently written as
	\begin{equation}
		\min_{X\in\SS^n} \quad C \bullet X
		\quad \st\quad
		A_i \bullet X  = b_i,\; i \in [m], \quad X \in Q_{n,k} 
	\end{equation}
	Hence, for any convex cone $K \supseteq Q_{n,k}$, the following is a convex relaxation of problem~\eqref{sparse_qcqp}:
	\begin{equation}\label{eq:sdpK}\tag{\small{Q${}_K$}}
		\begin{aligned}
			\min_{X \in \mathbb{S}^n}\quad C \bullet X
			\quad\st\quad
			A_i \bullet X = b_i,\,  i \in [m],\quad
			X\in K
		\end{aligned}
	\end{equation}
	The SDP in \eqref{sparse_sdp} corresponds to $K = \Szero{n}{k}$.
	The dual problem is
	\begin{equation}\label{eq:sdpKdual}\tag{\small{D${}_K$}}
		\begin{aligned}
			\max_{\lambda\in \RR^m}\quad \lambda \cdot b
			\quad\st\quad
			C -  \sum_{i=1}^m \lambda_i A_i  \in K^*
		\end{aligned}
	\end{equation}
	where $K^*$ is the dual cone.
	In particular, the dual problem of \eqref{sparse_sdp} is
	\begin{equation}\label{sparse_sdp_dual}\tag{\small{D}}
		\begin{aligned}
			\max_{\lambda\in \RR^m, Z\in \SS^n}\quad & \lambda \cdot b  \\
			\st\quad &  Q:=C -  \sum_{i=1}^m \lambda_i A_i - (k \diag Z - Z) \succeq 0\\
			& Z \succeq 0
		\end{aligned}
	\end{equation}
	The optimal values of
	\eqref{eq:sdpKdual},
	\eqref{eq:sdpK},
	and
	\eqref{sparse_qcqp} always satisfy
	\[
	\val(D_K) \leq \val(Q_K) \leq \val(P).
	\]
	The relaxation \eqref{eq:sdpK} is \emph{exact} if its optimal solution lies in $Q_{n,k}$.
	
	The following lemma gives a simple optimality certificate for both $\Szero{n}{k}$ and $\Sz{n}{k}$:
	if the minimizer of the relaxation is rank one, then the relaxation is exact.
	
	\begin{lem}
		Let $K$ be a convex cone such that $Q_{n,k} \subseteq K\subseteq \Szero{n}{k}$.
		If $x x^T$ is optimal for \eqref{eq:sdpK} then the relaxation is exact and $x$ is optimal for \eqref{sparse_qcqp}.
	\end{lem}
	
	\begin{proof}
		Since $x x^T \in \Szero{n}{k}$, then $\|x\|_0 \leq k$ by Lemma~\ref{lem:cardinality}.
		Hence, $x$ is feasible for \eqref{sparse_qcqp},
		so it leads to an upper bound on the optimal value.
		But it is also a lower bound, since $x x^T$ is a solution of the relaxation.
		It follows that $x$ is optimal for \eqref{sparse_qcqp}.
	\end{proof}
	
	\section{Local stability analysis}\label{section:stability}
	
	In this section we study how the SDP relaxation \eqref{sparse_sdp} behaves when the data defining the problem is perturbed.
	We focus our attention to perturbations in the objective.
	This is motivated by applications such as sparse PCA and sparse ridge regression, where the problem data only appears in the the objective.
	The stability of general QCQPs was first studied in \cite{cifuentes2022local};
	we consider here the extension to the sparse setting.
	Our stability results will be applied to sparse PCA and sparse ridge regression in later sections.
	
	The general setup we study is as follows.
	Let $\bar C, \{A_i,b_i\}_{i \in [m]}$ be fixed, such that the SDP relaxation \eqref{sparse_sdp} is exact with those specific parameters.
	Consider perturbing the objective matrix $C = \bar C + \Delta C$, where the perturbation (or noise) $\Delta C$ is small.
	The constraint parameters $A_i, b_i$ remain fixed.
	Our goal is to establish conditions under which the relaxation \eqref{sparse_sdp} remains exact for the perturbed objective $C$,
	and to understand how much noise the relaxation tolerates while still being exact.
	
	We say that \eqref{sparse_sdp} is \textit{locally stable} near $\bar C$ if there exists a neighborhood of \(\bar C\), denoted as \(\delta(\bar C)\), such that
	\eqref{sparse_sdp} is exact (and hence \(\val(Q) = \val(P)\))
	for any objective $C \in \delta(\bar C)$.
	Let $\bar x$ be an optimal solution of \eqref{sparse_qcqp},
	and $\bar\lambda, \bar Z$ be an optimal solution of \eqref{sparse_sdp_dual}, using $\bar C$ for the objective.
	These solutions will certainly change when we consider other objectives $C$.
	Nonetheless, our local stability conditions only depend on properties of~$\bar x, \bar\lambda, \bar Z$.
	
	\subsection{Conditions guaranteeing local stability}
	
	Our stability analysis relies on a well-known constraint qualification, which we proceed to introduce.
	Given a feasible solution $\bar x$ of \eqref{sparse_qcqp}, let
	\begin{gather*}
		\vX_{\supp (\bar x)} = \{x\in\RR^n: g(x) = 0, \quad x_i = 0\;  \forall i \notin \supp(\bar x)\}
		\\
		\text{ where }
		g(x) := (x^T A_i x - b_i : i \in [m]) \in \RR^m
	\end{gather*}
	The \textit{Abadie constraint qualification} (ACQ) holds at $\bar x$ for $\vX_{\supp (\bar x)}$ if the set $\vX_{\supp (\bar x)}$ is a smooth manifold near $\bar x$ and $\rank\big(\nabla g (\bar x)\big) = \codim_{\bar x}(\vX_{\supp (\bar x)})$. Given a feasible solution \(\bar x\in\RR^n\), \(\lambda\in \RR^m\) is a \textit{Lagrange multiplier} at \(\bar x\) with respect to \(\vX_{\supp (\bar x)}\) if
	\[
	(C \bar x)_j - \sum_{i=1}^m\lambda_i (A_i\bar x)_j = 0,\; \forall j\in \supp(\bar x)
	\]
	
	We are ready to state our main theorem of this section.
	
	\begin{thm}\label{thm:sparse_local_stability}
		Let \(\bar C\) and \(\{A_i, b_i\}_{i\in [m]}\) be such that \(\val(D) = \val(Q) = \val(P)\).
		Let $\bar x$ be optimal for \eqref{sparse_qcqp}
		and $\bar\lambda,\bar Z$ be optimal for \eqref{sparse_sdp_dual}.
		Suppose that:
		\begin{enumerate}[label=(\alph*)]
			\item $\bar x$ has exactly $k$ nonzero entries.
			\item $\ACQ$ holds at $\bar x$ for $\vX_{\supp (\bar x)}$.
			\item $\bar Z = \vzero$. 
			\item $\bar Q:= \bar C - \cA(\bar\lambda)$ has corank one, where \(\cA(\lambda):= \sum_{i=1}^m \lambda_i A_i\).
		\end{enumerate}
		Then \eqref{sparse_sdp} is locally stable near $\bar C$.
	\end{thm}

    When noise appears, denote $( x, \lambda)$ as the optimal primal and dual solution of \eqref{sparse_qcqp} on \(\vX_{\supp (\bar x)}\), which is defined as 
	\begin{equation}\label{sparse_qcqp_supp}
        \tag{P-supp}
		\min_{x\in\RR^n}\quad x^T C x
		\quad\st\quad x \in \vX_{\supp (\bar x)}
	\end{equation}
	with respect to  \(C\) near \(\bar C\).
    Before proceeding to the proof of \Cref{thm:sparse_local_stability}, we first show that \(x \rightarrow \bar x\) as \(C\rightarrow \bar C\).

    \begin{lem}\label{lem:primal_convergence}
        For each \(C\), let \(x\) be the optimal solution of \eqref{sparse_qcqp_supp}. Then \(x\) converges to \(\bar x\), up to sign.
    \end{lem}
    \begin{proof}
        See \Cref{append:lem_primal_convergence}.
    \end{proof}

    Then we show that there exists \(\lambda\) such that \(\lambda\rightarrow\bar \lambda\) as \(C\rightarrow \bar C\).
	
	\begin{lem}\label{lem:dual_stability}
		Given a feasible solution \(\bar x\in\RR^n\) of \eqref{sparse_qcqp},
		suppose that ACQ holds at $\bar x$ for
		\(\vX_{\supp(\bar x)}\).
        Let \(S:=\supp(\bar x)\), and \(g_s := g|_{S} : \RR^S \to \RR^m \) be the restriction of $g$ to the variables on $S$.
        Let \(\bar J:=\nabla g_s(\bar x_S) \in \RR^{m \times k}\) be the Jacobian. 
		Given a Lagrange multiplier \(\bar \lambda \in \RR^m\) at \(\bar x\) with respect to \(\vX_{\supp (\bar x)}\),
		then there exists \(\lambda \in \RR^m\) such that \(C_S\bar x_S - \bar J^T\lambda =0\) and
		\begin{equation}
			\|\bar\lambda -  \lambda\| \leq \frac{1}{\sigma_s} \|C_S\bar x_S -\bar J^T\bar\lambda\|\leq \frac{1}{\sigma_s}\|\Delta C_S\|\|\bar x_S \|.
		\end{equation}
	\end{lem}
	\begin{proof}
		Since \(\bar\lambda\) is a Lagrange multiplier at \(\bar x\) with respect to \(\vX_{\supp (\bar x)}\), then
		\begin{equation}
			\begin{split}
				\bar J^T(\lambda- \bar\lambda)= C_S\bar x_S - \bar J^T \bar \lambda
			\end{split}
		\end{equation}
		which implies that 
        \begin{equation}
            \begin{split}
                \|\lambda- \bar\lambda\| = \|\bar J^\dagger(C_S\bar x_S - \bar J^T\bar\lambda)\bar\| \leq &  \frac{1}{\sigma_s} \|C_S\bar x_S -\bar J^T\bar\lambda\|\\
                = & \frac{1}{\sigma_s}\|C_S\bar x_S - \bar C_S\bar x_S + \bar C_S \bar x_S - \bar J^T\bar\lambda\|\\
                \leq & \frac{1}{\sigma_s}\|\Delta C_S\|\|\bar x_S \|
            \end{split}
        \end{equation}
		where \(\bar J^\dagger\) is the pseudo-inverse of \(\bar J\), which satisfies \(\|\bar J^\dagger\|_{op} = \frac{1}{\sigma_s}\).
	\end{proof}

    We are ready to prove \Cref{thm:sparse_local_stability}.
	\begin{proof}[Proof of \Cref{thm:sparse_local_stability}]
		Let $x$ be the optimal primal solution of \eqref{sparse_qcqp_supp}. From \Cref{lem:primal_convergence}, it is true that $x$ is close to $\bar x$. It follows that \(\supp(x)=\supp(\bar x)\), as $\bar x$ has exactly $k$ non-zeros.  Since ACQ holds at \(\bar x\) for $\vX_{\supp (\bar x)}$, then from \Cref{lem:dual_stability}, there exists dual multiplier $\lambda$ of \eqref{sparse_qcqp_supp} which is close to $\bar\lambda$. Let $P = C - \sum_{i=1}^m\lambda_iA_i$
	    and $ p =  P x$. Then we have
		\begin{equation}\label{eq:stability_first_order_condition}
			p_j=(Cx)_j - \sum_{i=1}^m\lambda_i(A_ix)_j=0,\;\forall j\in \supp(\bar x).
		\end{equation} 
        
		Without loss of generality, we assume that \(\supp(\bar x):=\{1,2,\dots,k\}\). Let \(x^{-1}:=(x^{-1}_1, x^{-1}_2,\dots, x^{-1}_k, \vzero)\).
        Consider the matrix 
		\begin{equation}
			Z := \zeta\zeta^T \in \SS^n,
			\quad\text{ where }\quad
			\zeta := \alpha x^{-1}  - (k\alpha)^{-1} p.
		\end{equation}
		Let $\cD(Z) = k\diag(Z) - Z$ and the dual slack matrix in \eqref{sparse_sdp_dual} is $Q = P  - \cD(Z)$.
		We will show that
		\begin{gather}\label{eq:complementarity}
			Q \bullet xx^T = 0, \quad
			Q \succeq 0,\quad\text{and}\quad\corank(Q)=1.
		\end{gather}
		By plugging the expression of $\zeta$, it is true that the complementary slackness holds.
		We proceed to prove that $ Q \succeq 0$,
		or equivalently,
		that $w^T Q w \geq 0$ for arbitrary~$w$. Consider a projection $\varphi: \RR^n \mapsto \langle  x\rangle$ where $\langle x\rangle$ denotes the linear span of~$x$, and $\varphi^\perp: \RR^n \mapsto \langle x\rangle^\perp$. Decomposing $w = \varphi(w) + \varphi^\perp(w)$, then we observe that
		\begin{equation}
			\begin{split}
				\varphi(w)^T Q \varphi^\perp(w) &= \varphi(w)^T (P - \cD(Z)) \varphi^\perp(w)\\
				&= \tfrac{\|x\|}{\|\varphi(w)\|}\Big(p^T \varphi^\perp(w)- \big(k\alpha^2 x\circ x^{-2} -(x^T\zeta) \zeta^T\big)\varphi^\perp(w)\Big)\\
				&= \tfrac{\|x\|}{\|\varphi(w)\|}\Big(p^T \varphi^\perp(w)-k \alpha\big(\alpha x^{-1} - \zeta\big)^T\varphi^\perp(w)\Big)\\
				&=0
			\end{split}
		\end{equation}
		From \Cref{lem:dual_stability}, it is true that \(P \rightarrow \bar Q\) as \(C\rightarrow\bar C\) which implies that
		\begin{equation}\label{eq:stable_P}
			\begin{split}
				\varphi^\perp(w)^T P \varphi^\perp(w)\rightarrow \varphi^\perp(w)^T \bar Q \varphi^\perp(w) \geq  \|\varphi^\perp(w)\|^2 \nu_2(\bar Q)>0,\; \text{as}\;\|C - \bar C\|\rightarrow 0
			\end{split}
		\end{equation}
		Then we have \(\|p\|=\|Px\|\leq\|P\bar x\| + \|P(x-\bar x)\| \rightarrow \|\bar Q \bar x\|=0\), as \(C\rightarrow\bar C\).
        By letting \(\alpha = \sqrt{\|p\|}\), it follows that
		\begin{equation}\label{eq:stable_D}
			\begin{split}
				&\varphi^\perp(w)^T \cD(Z) \varphi^\perp(w) \\
				=& k\alpha^2\sum_{i=1}^k \big(\tfrac{\varphi_i^\perp(w)}{ x_i}\big)^2 + \frac{1}{k\alpha^2}\sum_{i=k+1}^n \big(\varphi_i^\perp(w) p_i\big)^2  -  \frac{1}{(k\alpha)^2}\langle\varphi^\perp(w), p\rangle^2\\
				\leq& O(\|\varphi^\perp(w)\|^2\|p\|) \rightarrow 0\; \text{as}\;\|C - \bar C\|\rightarrow 0
			\end{split}
		\end{equation}
		Then it follows that
		\begin{equation}
			\begin{split}
				w^T Q w
				&= \varphi(w)^T Q \varphi(w) + 2 \varphi(w)^T Q \varphi^\perp(w) + \varphi^\perp(w)^T Q \varphi^\perp(w)\\
				&= \varphi^\perp(w)^T Q \varphi^\perp(w)\\
				&= \varphi^\perp(w)^T (P -\cD(Z)) \varphi^\perp(w) \\
				&\geq 0, \; \text{as}\;\|C - \bar C\|\rightarrow 0
			\end{split}
		\end{equation}
		We have shown that \eqref{eq:complementarity} holds. Thus, \eqref{sparse_sdp} is stable near \(\bar C\).
	\end{proof}
	
	\subsection{Sufficient conditions for exactness}
	
	The following theorem provides some sufficient conditions for \eqref{sparse_sdp} to remain exact for a specific objective  matrix~$C$. 
	
	\begin{thm}\label{thm:general_exact_region}
		Let $C = \bar C + \Delta C$ and let $x$ be the optimal solution of \eqref{sparse_qcqp} for such~$C$. Let  $s=\codim_{\bar x}(\vX_{\supp(\bar x)})$, and $\sigma_s > 0$ be the $s$-th smallest singular value of \(\nabla g_s(\bar x_S)\). The relaxation \eqref{sparse_sdp} is exact if the following condition holds:
		\begin{equation}\label{eq:exact_region}
			\begin{gathered}
				\nu_2(\bar Q) \;>\;
				\eta \Big(
				\big(1
				+\frac{1}{\sigma_s}\|\cA\|\|\bar x \|\big)\big(1 + \frac{1}{c_{ x}}\| x\|\big)\|\Delta C\|
				+ \frac{1}{c_{ x}}\|\bar Q\|\| x- \bar x\|
				\Big),
				\\
				\eta := \Big(1 + \frac{\|\bar x\|^2\langle  x,z_2\rangle^2}{\langle x, \bar x\rangle^2}\Big)
			\end{gathered}
		\end{equation}
		where $\nu_2(\bar Q)$ is the second smallest eigenvalue of $\bar Q$, $z_2$ is the corresponding eigenvector of $\nu_2(\bar Q)$, \(\cA\) is the linear map \( \lambda \mapsto \sum_{i=1}^m\lambda_iA_i\), and $c_{x}:= \min\{|x_i|: i \in \supp(x)\}$ is a constant. 
	\end{thm}
	\begin{proof}[Proof of Theorem~\ref{thm:general_exact_region}]
		Following the same notation as in \Cref{thm:sparse_local_stability}, denote $(x, \lambda)$ as the optimal primal and dual solution of \eqref{sparse_qcqp} on \(\vX_{\supp (\bar x)}\) with respect to  \(C\) near \(\bar C\). Let $P = C - \sum_{i=1}^m\lambda_iA_i$
		and let $ p =  P x$.
		Consider a projection $\varphi: \RR^n \mapsto \langle  x\rangle$ where $\langle x\rangle$ denotes the linear span of~$x$, and $\varphi^\perp: \RR^n \mapsto \langle x\rangle^\perp$.
		Define 
		\[v(Y) := \underset{y \in \RR^n}{\min} \;\varphi^\perp(y)^T Y \varphi^\perp(y) \;\st\; \|\varphi^\perp(y)\|=1 .
		\]
		Since we have that $P = \bar Q + (C - \bar C) + (\lambda - \bar \lambda)^T \cA$, then it is true that
		\begin{equation}\label{eq:exact_region_ineq1}
			\begin{split}
				v(P) &\geq v(\bar Q) + v\big(( C - \bar C) + \cA(\lambda - \bar \lambda)\big)\\
				&\geq v(\bar Q) - \|\Delta C\| - \frac{1}{\sigma_s}\|\cA\|\|\lambda - \bar\lambda\|\\
				&\geq v(\bar Q) - \|\Delta C\| - \frac{1}{\sigma_s}\|\cA\|\| C_S\bar x_S - \bar\lambda^T\bar J\| \\
				&\geq v(\bar Q) - \|\Delta C\| - \frac{1}{\sigma_s}\|\cA\|\| \Delta C\|\|\bar x \|=:  g
			\end{split}
		\end{equation}
		where we used that \(\| C_S\bar x_S - \bar\lambda^T\bar J\|\leq \|\Delta C\|\|\bar x \|\) in the last inequality.
		Consider the same matrix as we constructed in \Cref{thm:sparse_local_stability},
		\begin{equation}
			Z := \zeta\zeta^T \in \SS^n,
			\quad\text{ where }\quad
			\zeta := \alpha x^{-1}  - (k\alpha)^{-1} p.
		\end{equation}
		Let $\cD(Z) = k\diag(Z) - Z$, and let $\alpha = \sqrt{\frac{c_{ x} \|p\|_\infty}{k}}$.
		Then we have
		\begin{equation}\label{eq:exact_region_ineq2}
			v(\cD (Z)) \leq \max\{k\diag(Z)\} = \max\{\tfrac{k\alpha^2}{c^2_{x}}, \tfrac{1}{k\alpha^2} \| p\|^2_\infty\} = \tfrac{\|p\|_\infty}{c_x}
		\end{equation}
		Furthermore, we have
		\begin{equation}\label{eq:exact_region_ineq3}
			\begin{split}
				\| p\|_\infty =  \|P x\|_\infty
				\leq & \Big\|\big(\bar Q + (C - \bar C) + (\lambda - \bar\lambda)^T \cA\big)x\Big\|\\
				\leq & \big(1 +\frac{1}{\sigma_s}\|\cA\|\|\bar x\|\big)\| x\|\|\Delta C\|
				+\|\bar Q\|\| x- \bar x\| =:  w
			\end{split}
		\end{equation}
        
        Let \(\nu_i(\bar Q)\) be the second \(i\)-th smallest eigenvalue of \(\bar Q\), and \(z_i\) be the associated eigenvector. We claim that 
        \begin{equation}\label{eq:claim}
        v(\bar Q) \geq \frac{\langle x, \bar x\rangle^2}{\|\bar x\|^2\|x'\|^2} \nu_2(\bar Q).
        \end{equation}
        Consider \(y = \sum_{i=1}^n \langle y,z_i\rangle z_i\) which satisfies \(y^T x = 0\), and \(\|y\|=1\). Let \(x = \sum_{i=1}^n \langle x,z_i\rangle z_i\). It is true that
        \[
        y^T \bar Q y = \sum_{i=2}^n \nu_i(\bar Q) \langle y,z_i\rangle^2 \geq \nu_2(\bar Q) (1 - \langle y,\tfrac{\bar x}{\|\bar x\|}\rangle^2)
        \]
        Since we have \(y^Tx=0\) which is equivalent to 
        \[
         \langle y,\tfrac{\bar x}{\|\bar x\|} \rangle \langle x,\tfrac{\bar x}{\|\bar x\|} \rangle + \sum_{i=2}^n \langle x,z_i \rangle  \langle y,z_i \rangle = 0
        \]
        it follows that
        \[
         \langle y,\tfrac{\bar x}{\|\bar x\|} \rangle \langle x,\tfrac{\bar x}{\|\bar x\|} \rangle \leq \sqrt{\textstyle\sum_{i=2}^n  \langle x,z_i \rangle^2} \sqrt{1 - \langle y,\tfrac{\bar x}{\|\bar x\|} \rangle^2}
        \]
        which implies that
        \[
        \langle y,\tfrac{\bar x}{\|\bar x\|} \rangle^2\leq 1 - \frac{\langle x,\bar x \rangle^2}{\|x\|^2\|\bar x\|^2}
        \]
        Hence, the claim in \eqref{eq:claim} is true. 
        So if \eqref{eq:exact_region} holds, then $\frac{1}{c_{x}}  w <  g$. It follows that
		\[
		v(P-\cD(Z)) \geq v(P) + v(-\cD(Z))\geq g -  \frac{1}{c_{x}}  w> 0
		\]
		and consequently \eqref{sparse_sdp} is exact. 
	\end{proof}

	\section{Sparse PCA}\label{section:spca}
	
	For the specific case of \eqref{sparse_pca},
	our primal SDP becomes
	\begin{equation}\label{sdp_spca}
		\tag{\small{Q-spca}}
		\max_{X\in\mathbb{S}^n} \quad  \Sigma \bullet X
		\quad\text{ s.t. }\quad \tr(X)= 1, \quad
		X \in \Szero{n}{k}
	\end{equation}
	and its dual SDP is
	\begin{equation}\label{sdp_spca_dual}
		\tag{D-spca}
		\begin{aligned}
			\min \quad \rho
			\quad\text{ s.t. }\quad
			\rho I_n \succeq  k\diag(Z) - Z + \Sigma,\quad
			Z \succeq 0
		\end{aligned}
	\end{equation}
	In this section, we investigate the theoretical guarantees of these SDPs.
	First, we establishes an upper bound on the optimality ratio, that is, the ratio of the optimal value of \eqref{sdp_spca} and that of \eqref{sparse_pca}.
	Second, we provide explicit conditions under which \eqref{sdp_spca} is exact.
	Next, we prove high-probability exactness guarantees for random Wigner and spiked Wishart models. Finally, we apply our SDP to give theoretical bound of restricted isometry constant.
	
	\subsection{Relaxation gap}
	
	The following theorem provides a worst-case approximation ratio for the case of PSD objectives.
	
	\begin{thm}\label{thm:opt_ratio_spca}
		Suppose $\Sigma \succeq 0$.
		Let $v^*$ be the optimal value of \eqref{sparse_pca}, and $\bar v^*$ be the optimal value of \eqref{sdp_spca}.
		Then
		\[
		\frac{\bar v^*}{v^*} \leq  \min\{k,n/k,r\}
		\]
		where $r = \rank(\Sigma)$.
	\end{thm}
	
	In order to prove the theorem, we first consider the special case $\Sigma = \sigma \sigma^T$.
	
	\begin{lem}\label{lem:rank_one_case}
		If $\Sigma =\sigma\sigma^T$ for some $\sigma\in\mathbb{R}^n$, then \eqref{sdp_spca} is exact .
	\end{lem}
	\begin{proof}
		See \Cref{append:lem_rank_one_case}.
	\end{proof}
	
	We proceed to prove our optimality ratio.
	
	\begin{proof}[Proof of \Cref{thm:opt_ratio_spca}]
		The pair $Z_0:=\Sigma$, $\rho_0 := k \max\limits_{1\leq i\leq n} \Sigma_{ii}$ is feasible for the dual problem \eqref{sdp_spca_dual}, so
		\[
		\bar v^* \leq \rho_0 \leq k v^*.
		\]
		Similarly, $Z_1:=0$, $\rho_1:= \lambda_{\max}(\Sigma)$ is also dual feasible, so
		\[
		\bar v^* \leq \rho_1 \leq (n/k)v^*.
		\]
		Hence, the optimality ratio is bounded by $\min\{k,n/k\}$.
		It remains to bound it with the rank of $\Sigma$. 
		
		Consider the spectral decomposition of $\Sigma$, that is, $\Sigma = \sum_{i=1}^r \sigma^{(i)} (\sigma^{(i)})^T$ where $r = \rank(\Sigma)$. Let $\Sigma^{(i)}:=\sigma^{(i)}(\sigma^{(i)})^T$, and denote $v^*_i$ as the optimal value of \eqref{sparse_pca} with respect to objective matrix $\Sigma^{(i)}$. From \Cref{lem:rank_one_case}, for each $\Sigma^{(i)}$ there exists dual variables $Z^{(i)} = u^{(i)}(u^{(i)})^T$ and $\rho^{(i)}$ such that $\rho^{(i)} = v^*_i$. 
		Consider the pair $\tilde Z = \sum_{i=1}^r Z^{(i)}$, $\tilde\rho = \big\|k\diag(\tilde Z) - \tilde Z  +\Sigma\big\|$.
		Then
		\begin{equation}
			\begin{split}
				\bar v^*\leq
                \tilde\rho = \big\|k\diag(\tilde Z) - \tilde Z  +\Sigma\big\|
                \leq & \sum_{i=1}^r \big\|k\diag(Z^{(i)}) - Z^{(i)}  +\Sigma^{(i)}\big\|\\
				= &\sum_{i=1}^r \rho^{(i)}=\sum_{i=1}^r v^*_i \leq rv^*
			\end{split}
		\end{equation}
		Hence, the optimality ratio is bounded by $\rank \Sigma$, completing the proof.
	\end{proof}
	
	The next result applies to arbitrary~\(\Sigma\),
	even if \( \Sigma \not\succeq 0 \).
	
	\begin{cor}\label{cor:lower_bound_spca}
		Let $v^*$ and $\bar v^*$ be the optimal values of \eqref{sparse_pca} and \eqref{sdp_spca}, respectively.
		Then
		\[
		\bar v^* \leq q v^* - (q-1) \lambda_{\min}(\Sigma)
		\]
		where $q = \min\{k, n/k, r\}$ and $r = n - \mathrm{mult}(\lambda_{\min}(\Sigma))$ is the co-multiplicity of the smallest eigenvalue of~$\Sigma$.
	\end{cor}
	\begin{proof}
		Let \(\tau = \lambda_{\min}(\Sigma)\) and \(\tilde \Sigma:=\Sigma - \tau I_n\).
		Note that \(\tilde\Sigma\succeq 0\), $\rank \tilde \Sigma = r$.
		Applying \Cref{thm:opt_ratio_spca} to $\tilde\Sigma$, we have that \(\bar v^* - \tau \leq q(v^* - \tau)\) which is equivalent to \(\bar v^*\leq qv^* - (q-1)\tau \). 
	\end{proof}
	
	\subsection{Exact recovery} 
	We explore conditions under which \eqref{sdp_spca} is exact. Consider a projection \(\pi_{L^\perp}:\SS^n \mapsto L^\perp\;\), where \(
	L^\perp\) is the complement of \(L:= \{ t I : t \in \RR\}\). We begin with the following theorem.
	
	\begin{thm}\label{thm:spca_exact_case}
		Given \(\Sigma \in \SS^n\), if there exists \(\Sigma_0\in\SS^n\) such that \eqref{sdp_spca} is exact with respect to \(\Sigma_0\) and \(\pi_{L^\perp}(\Sigma) = \pi_{L^\perp}(\Sigma_0)\), then  \eqref{sdp_spca} is exact with respect to \(\Sigma\). 
	\end{thm}
	\begin{proof}
		Since we have \(\pi_{L^\perp}(\Sigma) = \pi_{L^\perp}(\Sigma_0)\), there exists \(t\in \RR^n\) such that \(\Sigma = \Sigma_0 + tI_n\). For any \(S \subseteq [n]\) with \(|S| = k\), let \(\Sigma_{SS}\) and \((\Sigma_0)_{SS}\) be $k\times k$ principal minors of \(\Sigma\) and \(\Sigma_0\). Note that the eigenvectors of  \(\Sigma_{SS}\) are the same as \((\Sigma_0)_{SS}\). Since \eqref{sdp_spca} is exact for \(\Sigma_0\), the optimal solution of \eqref{sdp_spca} for \(\Sigma\) remains the same, which implies that \eqref{sdp_spca} is exact with respect to \(\Sigma\).
	\end{proof}
	
	Based on \Cref{thm:general_exact_region}, the following theorem establishes an explicit bound which implies the exact region of our SDP.  
	
	\begin{thm}\label{thm:stable_sdp}
		Let $\bar \Sigma = \beta \bar x \bar x^T+tI_n,\beta\geq0$, where $\|\bar x\|_0=k\geq 2$, and $\|\bar x\|=1$. Suppose that $\Sigma$ lies in a small neighborhood of $\bar \Sigma$. Let $c = \min\{| \bar x|_i: i \in \supp(\tilde x)\}$. If the following assumption on the objective matrix $\Sigma$ is satisfied
		\begin{equation}\label{eq:spca_threshold}
			\|\pi_{L^\perp}(\Sigma - \bar \Sigma)\| < \nu\beta  
		\end{equation}
		where \(\nu < \min\big\{ \tfrac{1}{2}(\frac{3}{2} + \frac{3 + 4\sqrt{2}}{c})^{-1},\tfrac{c^2}{4}\big\}\), then \eqref{sdp_spca} is exact. 
	\end{thm}
	
	\begin{proof}
		We start with certifying the assumption in our stability result.  
		For \eqref{sparse_pca}, the only constraint is $g(x) = \|x\|^2-1 = 0$. ACQ holds at $x = \bar x$, because of $\rank(\nabla g(\bar x))  = \codim_{\bar x} (\vX_{\supp(\bar x)}) = 1$.Let \(\Delta \Sigma = \Sigma - \bar\Sigma \).  Decompose \(\Delta \Sigma =  \pi_{L}(\Delta\Sigma)+  \pi_{L^\perp}(\Delta\Sigma) \), and suppose that \( \pi_{L}(\Delta\Sigma) = t'I_n\). For \eqref{sdp_spca_dual}, let the dual multiplier $\bar \rho = \beta+t+t'$ and $\bar Z=0$. Then for objective matrix $\bar \Sigma + \pi_{L}(\Delta\Sigma)$, the primal solution $\bar X = \bar x\bar x^T$ and dual matrix $\bar Q:= \bar \rho I_n - (\bar\Sigma + \pi_{L}(\Delta\Sigma))$ satisfy 
		\[
		\bar X \bullet \bar Q = 0 \quad \text{and}\quad \corank(\bar Q)=1
		\]
		Let \(x^*\) be the optimal solution of \eqref{sparse_pca} with \(\Sigma\). We claim that if \(\|\pi_{L^\perp}(\Delta\Sigma)\|\leq \tfrac{c^2}{2}\beta\), then we have that \(\supp(x^*)\) remains the same as \(\supp(\bar x)\). Let \(S := \supp(\bar x)\), and let \(v^*_S\) be the objective value of \eqref{sparse_pca} on the support \(S\).  Notice that
		\[
		v^* + t + t'  \geq \beta - \|\pi_{L^\perp}(\Delta\Sigma)\|
		\]
		Also let \(S'\) be another support, that is, \(S' \neq S\), and let \(v^*_{S'}\) be the objective value of \eqref{sparse_pca} on the support \(S'\). We have
		\[
		v^*_{S'} + t  +t' \leq \beta(1-c^2) + \|\pi_{L^\perp}(\Delta\Sigma)\|
		\]
		Thus,  \(\|\pi_{L^\perp}(\Delta\Sigma)\|\leq \tfrac{c^2}{2}\beta\) implies that \(\supp(x^*)=\supp(\bar x)\). 
		
		From Davis–Kahan theorem (\cite{yu2015useful}), we derive  
		\[
		\|x^* - \bar x\| \leq \sqrt{2}\sin(\angle(x^*,\bar x))\leq\tfrac{2 \sqrt{2}}{\beta}\|\pi_{L^\perp}( \Sigma - \bar \Sigma)\| \leq 2\sqrt{2}\nu.
		\]
		Notice that if \(\nu\leq\tfrac{c^2}{4}\), then we have that
		\[
		\max_{i\in [n]} |x^*_i - \bar x_i| \leq \|x^* - \bar x\| \leq\frac{\sqrt{2}c^2}{2}\leq \frac{c}{2}
		\]
		Then let \(c^*:= \min\{|x^*_i|,i\in\supp(x^*)\}\), we have \(c^*\geq \tfrac{c}{2}\). Furthermore, if we have \(\nu\leq\tfrac{c^2}{4}\) which implies that 
		\(\|x^* - \bar x\|^2 \leq \tfrac{1}{8}\),
		that is, \(\langle\bar x, x^*\rangle \geq \tfrac{15}{16}\). Since we have that
		\[
		\langle z_2, x^*\rangle^2 +  \langle\bar x, x^*\rangle^2 \leq 1
		\]
		where \(z_2\) is the eigenvector corresponding to the second smallest eigenvalue of \(\bar Q\), it follows that \(|\langle z_2, x^*\rangle|\leq |\langle\bar x, x^*\rangle|\), and as defined in \eqref{eq:exact_region} \(\eta = \big(1 + \frac{\|\bar x\|^2\langle  x,z_2\rangle^2}{\langle x, \bar x\rangle^2}\big) \leq 2\).
		
		Since $\nabla g(\bar x) = 2\bar x$, we have $\sigma_s = 2$. Note that $\|\cA\| = 1$ where \(\cA(\lambda) = \lambda I_n\), and \(\|\bar Q\| = \lambda\). Then from all of the above discussion,  \eqref{eq:exact_region} is equivalent to
		\begin{equation}\label{eq:spca_1}
			\begin{split}
				\tfrac{1}{2}\beta >&\tfrac{3}{2}\|\pi_{L^\perp}(\Delta\Sigma)\| (1+\tfrac{1}{c^*}) + \tfrac{\beta}{c^*}\|x^*-\bar x\|
			\end{split}        
		\end{equation}
		Notice that if 
		\[
		\frac{1}{2}>\Big(\frac{3}{2} + \frac{3 + 4\sqrt{2}}{c}\Big)\nu
		\]
		then we have
		\begin{equation}\label{eq:spca_2}
			\begin{split}
				&\frac{3}{2}(1+\frac{1}{c^*}) \|\pi_{L^\perp}(\Delta\Sigma)\|+ \frac{\lambda}{c^*}\|x^*-\bar x\|\\
				\leq &  \frac{3}{2}(1 + \frac{2}{c})\|\pi_{L^\perp}(\Delta\Sigma)\|+ \frac{2\beta}{c}\|x^*-\bar x\|\\
				\leq & \frac{3}{2}(1 + \frac{2}{c})\nu\beta +  \frac{4\sqrt{2}}{c}\nu\beta \\
				= &\Big(\frac{3}{2} + \frac{3 + 4\sqrt{2}}{c}\Big)\nu\beta
			\end{split}
		\end{equation}
		which implies that \eqref{eq:spca_1} is satisfied. Thus, if \(\nu < \min\big\{ \tfrac{1}{2}(\frac{3}{2} + \frac{3 + 4\sqrt{2}}{c})^{-1},\tfrac{c^2}{4}\big\}\),
		then \eqref{sdp_spca} is exact.\qedhere
		
	\end{proof}
	

	\subsection{Wigner and Wishart spiked models}\label{subsection:spiked_model}
	
	We now consider the sparse PCA problem under two statistical models.
	In both cases, the goal is to recover an unknown sparse signal $x\in\mathbb{\RR}^{n}$, $\|x\|_0\leq k$, called the spike, given a random estimator $\Sigma$ of $x x^T$.
	The random matrix $\Sigma$ is as follows:
	\begin{itemize}
		\item Spiked Wigner: $\Sigma = \beta xx^T + \dfrac{1}{\sqrt{n}}W$, where $W$ is a random symmetric matrix from the Gaussian orthogonal ensemble (GOE), i.e.,
		it has i.i.d.\ $\mathcal{N}(0,2)$ diagonal entries,
		and i.i.d.\ $\mathcal{N}(0,1)$ off-diagonal entries.
		\item Spiked Wishart: $\Sigma = \frac{1}{N}XX^T$, where $X$ is an $n\times N$ matrix whose columns are i.i.d.\ $\mathcal{N}(0, I_n +  \beta xx^T)$.
	\end{itemize}
	
	We apply our results to the spiked Wigner model.
	
	\begin{lem}[\cite{vershynin2018high}]\label{lem:GOE}
		If $W$ is drawn from $\mathrm{GOE}(n)$, then
		\[
		\mathbb{E}\|W\| \leq 2\sqrt{n}
		\]
		Furthermore, 
		\[
		\mathbb{P}\big\{\tfrac{1}{\sqrt{n}}\|W\|\geq 2 + t\big\} \leq 2\exp(-cnt^2).
		\]
		where $c$ is a constant.
	\end{lem}
	
	\begin{thm}
		For the spiked Wigner model, if \(\beta \geq \nu^{-1}(2 + t)\) our SDP \eqref{sdp_spca} is exact with probability at least  \(1 - 2e^{-cnt^2}\) with a constant \(c\).
	\end{thm}
	
	\begin{proof}
		Applying \Cref{lem:GOE}, it follows that the assumption in \Cref{thm:stable_sdp} is satisfied with probability at least  \(1 - 2e^{-cnt^2}\) with a constant \(c\).
	\end{proof}
	
	We next apply our results to the spiked Wishart model.
	
	\begin{lem}[\cite{koltchinskii2017normal}]\label{lem:wishart}
		Let \(\bar \Sigma = I_n + \beta xx^T\), let $x_i\sim \mathcal{N}(0,\bar \Sigma), i=1,\dots,N$ and $\Sigma = \frac{1}{N}\sum_{i=1}^Nx_ix_i^T$, then there exists a constant \(c\) such that 
		\begin{equation}
			\|\Sigma - \bar\Sigma\| \leq c\|\bar\Sigma\| \max\Big\{\sqrt{\tfrac{n+\beta}{N(1+\beta)}},\tfrac{n+\beta}{N(1+\beta)},\sqrt{\tfrac{t}{N}},\tfrac{t}{N}\Big\}
		\end{equation}
		with probability at least \(1 - e^{-t}\).
	\end{lem}
	\begin{thm}\label{cor:wishart}
		For the spiked Wishart model, suppose that
		\[
		N \geq c(\beta+1)\max\{\tfrac{n+\beta}{1+\beta},t\} \max\{\tfrac{1}{\nu\beta},\tfrac{1}{\nu^2\beta^2}\}
		\]
		where \(c\) is a constant, then our SDP \eqref{sdp_spca} is exact with probability at least \(1 - e^{-t}\).
	\end{thm}
	\begin{proof}
		Applying \Cref{lem:wishart}, it follows that the assumption in \Cref{thm:stable_sdp} is satisfied with probability at least  \(1 - e^{-t}\).
	\end{proof}
	
	\subsection{Application to Restricted Isometry Property}\label{subsection:spca_ric}
	A matrix $A\in\RR^{m\times n}$ satisfies the \emph{restricted isometry property} (RIP) of order $k$ with constants $\delta^-,\delta^+\geq0$ if 
	\begin{equation}\label{eq:rip}
		(1-\delta^-)\|x\|_2^2 \leq \|Ax\|^2_2 \leq (1+\delta^+)\|x\|_2^2
	\end{equation}
	for all $x\in\RR^n$ with $\|x\|_0\leq k$;
	see, e.g.,~\cite{candes2008restricted}.
	The best RIP constants $\delta^-_+,\delta^+_+$ satisfy:
	\begin{align*}
		1 - \delta^-_* &=
		\min\{ \|Ax\|^2_2 : \|x\|_2 = 1, \|x\|_0 \leq k \},
		\\
		1 + \delta^+_* &=
		\max\{ \|Ax\|^2_2 : \|x\|_2 = 1, \|x\|_0 \leq k \}
	\end{align*}
	We can use \eqref{sdp_spca} to estimate $\delta^-_*,\delta^+_*$.
	The following corollary provides guarantees on the quality of the SDP estimates.
	
	\begin{cor}
		Let \(1+\bar \delta^+\) be the optimal value of  \eqref{sdp_spca} with  \(\Sigma = A^TA\),
		and let \( - (1 - \bar \delta^-)\) be the optimal value of  \eqref{sdp_spca} with  \(\Sigma = -A^TA\).
		Let \(\delta^+_*, \delta^-_*\) be the best valid RIP constants.
		Then
		\begin{align*}
			&1+\delta^+_*
			\;\leq\;
			{1+\bar\delta^+}
			\;\leq\;
			\min\{q,r\} (1+\delta^+_*)
			\\
			&1-\delta^-_*
			\;\geq\;
			{1-\bar \delta^-}
			\;\geq\;
			q \big(1-\delta^-_* \big) - (q-1) \|A\|_{op}^2
		\end{align*}
		where \(q = \min\{k,n\}\), and \(r = \rank(A)\).
	\end{cor}
	
	\begin{proof}
		The bound on $\delta^+_*$ follows from \Cref{thm:opt_ratio_spca}
		and the bound on $\delta^-_*$ follows from \cref{cor:lower_bound_spca}.
	\end{proof}
	
	\section{Sparse Ridge Regression}\label{section:ridge}
	
	For the specific case of \eqref{sparse_lr},
	our primal SDP becomes
	\begin{equation}\label{sdp_slr}
		\tag{\small{Q-sridge}}
		\min_{\X \in \mathbb{S}^{n+1}} \quad C_\alpha \bullet \X
		\quad\text{ s.t. }\quad
		X \in \Szero{n}{k},\quad
		\X := \begin{pmatrix}
			1 & x^T\\
			x & X
		\end{pmatrix}\succeq 0
	\end{equation}
	where $C_\alpha = \frac{1}{m}\begin{pmatrix}
		y & -A
	\end{pmatrix}^T\begin{pmatrix}
		y & -A
	\end{pmatrix} + \alpha \begin{pmatrix}
		0 & 0\\
		0 & I_n
	\end{pmatrix}.$
	Its dual SDP problem is
	\begin{equation}\label{sdp_slr_dual}
		\tag{\small{D-sridge}}
		\max_{\rho\in\RR, Z\in \mathbb{S}^{n}}\quad  \rho
		\quad\st\quad
		Q := C_\alpha
		-\begin{pmatrix}
			\rho & 0\\
			0 & k\diag Z \!-\! Z
		\end{pmatrix} \succeq 0,\quad
		Z \succeq 0
	\end{equation}
	In this section, we investigate the theoretical guarantees of these SDPs.
	First, we provide an upper bound on the relaxation gap between the optimal values of our SDP and \eqref{sparse_lr}. Second, for overdetermined sparse linear regression, we provide sufficient conditions ensuring that our method is exact.
	Lastly, we apply our result to Gaussian noise model, and show our SDP remains exact with high probability.
	
	\subsection{Relaxation gap} 
	We first define the coherence of the design matrix \(A\), a quantity that will appear in our upper bound on the relaxation gap.
	\begin{dfn}
		Let $A = (a_1~a_2~\dots~a_n) \in \RR^{m\times n}$ be a matrix with $\ell_2$-normalized columns, i.e., $\|a_i\|=1,\forall i=1,\dots,n$. The \textit{coherence} of $A$ is defined as 
		\[
		\mu(A) := \max_{i\neq j} |\langle a_i, a_j \rangle|
		\]
	\end{dfn}
	
	Write \(\bar A\) for the column-normalized version of \(A\), and let \(L:=\max_j \|A_{:j}\|_2\) be the largest column \(\ell_2\)-norm of \(A\). Let $x^*_{\alpha}$ be the optimal sparse solution of \eqref{sparse_lr}, and let $r^*_{\alpha} := \frac{1}{m}\|y - Ax^*_{\alpha}\|^2$ be the corresponding mean square error (MSE). For sparse linear regression, i.e., $\alpha=0$, denote $x^*_0$ by the optimal sparse solution and $r^*_0$ by the associated MSE. Then we obtain the following result.
	
	\begin{thm}
		Let $\rho^*_{\alpha}$ be the optimal objective value of \eqref{sparse_lr}, and $\bar\rho^*_{\alpha}$ be the optimal value of \eqref{sdp_slr}, then we have
		\begin{equation}
			(1-\tau\eta_\alpha)\frac{\|y\|^2}{m} + \tau \eta_\alpha r^*_{0} \leq\bar\rho^*_{\alpha} \leq \rho^*_{\alpha}
		\end{equation}
		where $\tau = 1 + (k-1)\mu(\bar A)$ and $\eta_\alpha = \frac{L^2}{\lambda_{\min}(A^TA)+m\alpha}$. Furthermore, if $\alpha$ is chosen as follows,
		\[
		\bar \alpha=\frac{\tau}{m} L^2 - \frac{1}{m}\lambda_{\min}(A^TA)
		\]
		which implies $\tau\eta_{\bar\alpha}=1$, then we have
		\begin{equation}
			0\leq \rho^*_{\bar\alpha} - \bar\rho^*_{\bar\alpha} \leq \tau \cdot \frac{L^2}{m}\|x^*_0\|^2 \leq k \cdot \frac{L^2}{m}\|x^*_0\|^2
		\end{equation}
	\end{thm}
	
	\begin{proof}
		Suppose that $\rho \leq \|y\|^2$, then from Schur complement we have that for the dual problem $Q \succeq 0$ is equivalent to 
		\begin{equation}\label{eq:slr_schur}
			(\frac{1}{m}\|y\|^2 - \rho) (\frac{1}{m}A^TA+ \alpha I_n) \succeq k\diag(Z) - Z  + \frac{1}{m^2}A^Tyy^T A 
		\end{equation}
		From \Cref{lem:rank_one_case}, there exists $Z\succeq 0$ such that
		\begin{equation}\label{eq:slr_lem}
			\frac{1}{m^2}\|A^Ty\|^2_{(2,k)} I_n \succeq k\diag(Z) - Z  + \frac{1}{m^2}A^Tyy^T A 
		\end{equation}
		where $ \|\cdot \|_{(2,k)}$ is the $k$-support norm, i.e., largest $\ell_2$ norm with respect to all possible supports of cardinality $k$. 
        
		Define  $P^{(k)}_A(y):=\max_{T\subseteq[n]}\{\|A_T(A_T^TA_T)^{\dagger}A_T^T y\|: |T| = k\}$, where $A_T$ denotes the submatrix of $A$ consisting of the columns with indices in $T$, and  $(A_T^TA_T)^{\dagger}$ denotes the pseudo-inverse of $A_T^TA_T$. 
        Let \(S\) be the support such that \(\|A^Ty\|^2_{(2,k)} = \|A_S^Ty\|^2\). Denote \(v_S:=\bar A^T_Sy\),  then we have that
		\begin{equation}\label{eq:slr_schur_2}
			\begin{split}
				\tfrac{1}{L^2}\|A_S^Ty\|^2 \leq  \|v_S\|^2 \leq& \lambda_{\max}(\bar A_S^T\bar A_S) (v_S^T (\bar A_S^T\bar A_S)^{\dagger}v_S) \\
				\leq & (1+(k-1)\mu(\bar A)) \|P_{\bar A}^{(k)}(y)\|^2\\
				= & \tau  \|P_{A}^{(k)}(y)\|^2
			\end{split}
		\end{equation}
		Combining \eqref{eq:slr_schur} \eqref{eq:slr_lem} and \eqref{eq:slr_schur_2}, $\rho$ is a feasible solution if it satisfies
		\begin{equation}
			(\frac{1}{m}\|y\|^2 - \rho) \big(\frac{1}{m}\lambda_{\min}(A^TA) + \alpha \big)\geq \frac{\tau L^2}{m^2} \|P_A^{(k)}(y)\|^2 
		\end{equation}
		which is equivalent to
		\begin{equation}
			\begin{split}
				\rho \leq \tfrac{1}{m}\|y\|^2 - \tfrac{\tau \eta_\alpha}{m} \|P_A^{(k)}(y)\|^2 
				&= \tfrac{1}{m}\|y\|^2 - \tau \eta_\alpha \Big(\tfrac{\|y\|^2}{m} - r^*_0\Big) \\
				&= (1-\tau\eta_\alpha)\tfrac{\|y\|^2}{m} + \tau\eta_\alpha r^*_{0}
			\end{split}
		\end{equation}
		It implies that $\bar \rho^* \geq (1-\tau\eta_\alpha)\tfrac{\|y\|^2}{m} + \tau\eta_\alpha r^*_{0}$. Suppose $\tau \eta_{\bar\alpha} = 1$, that is,
		\[
		\bar \alpha=\frac{\tau}{m} L^2 - \frac{1}{m}\lambda_{\min}(A^TA)
		\]
		then we have
		\begin{equation}
			\bar\rho^*_{\alpha}\geq \Big(r^*_0  + \bar \alpha \|x^*_0\|^2\Big) - \bar \alpha \|x^*_0\|^2 \geq \rho^*_{\alpha} - \bar \alpha \|x^*_0\|^2
		\end{equation}
		which implies that
		\begin{equation}\label{eq:slr_upper_bound}
			0\leq \rho^*_{\alpha} - \bar\rho^*_{\alpha} \leq \Big(\frac{\tau L^2}{m} - \frac{\lambda_{\min}(A^TA)}{m}\Big) \|x^*_0\|^2 \leq k\cdot\frac{L^2}{m}\|x^*_0\|^2
			\qedhere
		\end{equation} 
	\end{proof}
	
	When the regularization parameter is set to zero (\(\alpha=0\)) and the sample size exceeds the number of features (\(m>n\)), that is, in the overdetermined sparse regression regime, we establish the following upper bound on the relaxation gap.
	
	\begin{cor}
		Let $\bar r^*_0$ be the optimal value of \eqref{sdp_slr} when $\alpha=0$, then we have that
		\begin{equation}
			(1-\tau\eta_0) \frac{\|y\|^2}{m} + \tau\eta_0 r^*_0  \leq \bar r^*_0 \leq r^*_0
		\end{equation}
		Suppose $\eta_0=\tau=1$, that is, $\frac{L^2}{\lambda_{\min}(A^TA)}=1$ and $\mu(\bar A) = 0$, then $\rho^*_0 = \bar\rho^*_0$.
	\end{cor}
	
	\subsection{Exact recovery}
	We next present a setting in which our stability result applies. Before proceeding to the result, we start with the following homogenized problem for sparse linear regression.
	\begin{equation}\label{homogenious_slr}
		\tag{P-slr}
		\begin{aligned}
			\min_{x=(x_0,x_1\dots,x_n)\in \RR^{n+1}} \quad  x^T H x - 2h^Tx + c
			\quad\st\quad
			x_0^2 = 1, \|x\|_0 \leq k,
		\end{aligned}
	\end{equation}
	where \(H = (\vzero; A)^T(\vzero; A)\), \(h =(0; A^Ty)\), and \(c = \|y\|^2\). The above problem can be formulated as the following SDP.
	\begin{equation}\label{sdp_homogenious_slr}
		\tag{Q-slr}
		\begin{aligned}
			\min_{X \in \SS^{n+1}} \quad   C \bullet X
			\quad\st\quad
			X_{11} = 1, X \in \Szero{n+1}{k+1},
		\end{aligned}
	\end{equation}
	where \(C = \left(\begin{smallmatrix}
		c & -h^T\\
		-h & A^TA
	\end{smallmatrix}\right)
	\). And the dual problem of \eqref{sdp_homogenious_slr} is
	\begin{equation}\label{sdp_homogenious_slr_dual}
		\tag{D-slr}
		\begin{aligned}
			\max_{\rho\in\RR,Z\in\SS^{n+1}} \quad   \rho 
			\quad\st\quad
			C - \rho E_{11} - \big((k+1)\diag(Z)-Z\big) \succeq 0
		\end{aligned}
	\end{equation}
	where \(E_{11} = \left(\begin{smallmatrix}
		1 & 0\\
		0 & 0
	\end{smallmatrix}\right) \).
	
	\begin{thm}\label{thm:stable_LR}
		Suppose that in noiseless case the underlying parameter of the sparse linear regression model \eqref{homogenious_slr} is $\bar x$ with $\|\bar x\|_0 = k$, i.e., \(\bar y= A \bar x\), and $\sigma_{\min}(A)>0$ in \eqref{sdp_homogenious_slr}. Let \(S:=\supp(\bar x)\) and \(c:=\min\{|\bar x_i|: i \in S\}\). Denote \(\kappa\) as the condition number of \(A\), let \(q:=(1+\tfrac{1}{2}\|\bar x\|)(1+\tfrac{4(2-\sqrt{2})}{c}\|\bar x\|)\), and \(p:=\kappa(q + \tfrac{2\kappa}{c})\). In noisy setting, random noise \(\varepsilon \in \RR^n\) appears in the model,
		\begin{equation}\label{eq:noisy_slr}
			\tilde y = A\bar x + \varepsilon
		\end{equation}
		we have \eqref{sdp_homogenious_slr} is exact if \(\|\varepsilon\| < \eta \sigma_{\min}(A)\) where  \(\eta =\min\big\{\tfrac{\sqrt{p^2 + 4q}-p}{2q}, \tfrac{c}{2}, (3-2\sqrt{2})\|\bar x\| \big\}\).
	\end{thm}
	\begin{proof}
		In noiseless case, there exists dual multipliers \(\bar \rho = 0, \bar Z = \vzero\) such that \(\bar Q:=\bar C - \bar \rho E_{11} - \big((k+1)\diag(\bar Z)-\bar Z\big)\) satisfies 
		\[
		\bar Q \bullet \left(\begin{smallmatrix}
			1 & \bar\beta^T\\
			\bar\beta & (\bar\beta\bar\beta^T)
		\end{smallmatrix}\right)=0,\quad\bar Q \succeq 0 \quad\text{and}\quad \corank(\bar Q)=1
		\]
		Then it is enough to show that ACQ holds at \(\bar x = (1;\bar \beta) \). For \eqref{homogenious_slr}, the only constraint is $g(x) = (x_0)^2-1 = 0$. ACQ holds at $x = \bar x$, for $\rank(\nabla g(\bar x))  = \codim_{\bar x} (\vX_{\supp(\bar x)}) = 1$. Furthermore, we have \(\sigma_{\min}(\nabla g(\bar x))=2\) and \(\|\cA\|=1\), where \(\cA(\rho) = \rho E_{11}\). Let \(x^*\) be the optimal solution of \eqref{sdp_homogenious_slr} with \(\tilde C = \left(\begin{smallmatrix}
			\|\tilde y\|^2 & -\tilde y^T A\\
			-A^T\tilde y & A^TA
		\end{smallmatrix}\right)\). Let \(c^*:=\min\{|x^*_i|: i\in\supp(x^*)\}\). Notice that
		\begin{equation}
			\|A(x^* - \bar x)\| = \|Ax^* - \bar y\| \leq \|\tilde y - \bar y\| = \|\varepsilon\|
		\end{equation}
		which implies that \(\|x^* - \bar x\| \leq \tfrac{1}{\sigma_{\min}(A_S)}\|\varepsilon\|\).
		
		Furthermore, we have
		\begin{equation}
			\begin{split}
				\|\tilde C - \bar C\| &= \Big\|\left(\begin{smallmatrix}
					\|\tilde y\|^2 - \|\bar y\|^2 & -(\tilde y - \bar y)^T A\\
					-A^T(\tilde y - \bar y) & 0
				\end{smallmatrix}\right)\Big\|\\
				&\leq |\|\tilde y\|^2 - \|\bar y\|^2| + 2\|A^T(\tilde y - \bar y)\|\\
				&\leq \|\varepsilon\|^2 + 2\|A\|_{op}\|\varepsilon\|
			\end{split}
		\end{equation}
		
		Thus, if \(\|\varepsilon\|\leq \min\{\tfrac{c}{2},(3-2\sqrt{2})\|\bar x\|\}\sigma_{\min}(A_S)\), we have \(\|x-x^*\|\leq \frac{c}{2}\) which implies that  \(c^* \geq \tfrac{c}{2}\) and consequently \(\supp(x^*)=\supp(\bar x)\). It also implies 
		\[
		2(\sqrt{2}-1)\|\bar x\|\leq\|x^*\|\leq 2(2- \sqrt{2})\|\bar x\|
		\]
		Then it is straightforward to show that
		\begin{equation}
			\begin{split}
				\|x^*-\bar x\|^2 &\leq (3-2\sqrt{2})^2\|\bar x\|^2\\
				&\leq \|x^*\|^2 + \|\bar x\|^2 -\sqrt{2}\|x^*\|\|\bar x\|
			\end{split}
		\end{equation}
		which implies that \(\langle\tfrac{x*}{\|x^*\|},\tfrac{\bar x}{\|\bar x\|}\rangle \geq \tfrac{\sqrt{2}}{2}\).
		In order to show \eqref{eq:exact_region}, it is enough to show that
		\begin{equation}
			\begin{split}
				\nu_2(\bar C) =\sigma^2_{\min}(A) >& 2q\|\varepsilon\|^2 + (q+\tfrac{2\kappa}{c})\|A\|_{op}\|\varepsilon\|
			\end{split}
		\end{equation}
		which holds if \(\|\varepsilon\| < \tfrac{\sqrt{p^2 + 4q}-p}{2q} \sigma_{\min}(A)\). Thus, combining all results, we have that if \(\eta =\min\big\{\tfrac{\sqrt{p^2 + 4q}-p}{2q}, \tfrac{c}{2}, (3-2\sqrt{2})\|\bar x\| \big\}\), then \eqref{sdp_homogenious_slr} is exact.
	\end{proof}
	
	\subsection{Gaussian noise model}
	
	We now consider the linear regression model with Gaussian random noise,
	\[
	y = A \bar x + \epsilon, \;\text{where}\; \epsilon \sim N(0, \delta^2 I_m)
	\]
	
	\begin{thm}
		For linear regression model with Gaussian random noise, if \( \eta \sigma_{\min}(A) \geq \delta\sqrt{m} + t\), then \eqref{sdp_homogenious_slr} is exact with probability at least \(1 - e^{-t^2/(2\delta^2)}\)
	\end{thm}
	\begin{proof}
		It is true that 
		\[
		\mathbb{P}\{\|\varepsilon\| \geq \delta\sqrt{m} + t\} \leq  e^{-t^2/(2\delta^2)}
		\]
		which implies that
		\[
		\mathbb{P}\{\|\varepsilon\| \geq \eta \sigma_{\min(A)}+t\} \leq  e^{-t^2/(2\delta^2)}
		\]
		Applying the above inequality to \Cref{thm:stable_LR}, we have \eqref{sdp_homogenious_slr} is exact with probability at least \(1-  e^{-t^2/(2\delta^2)}\).
	\end{proof}
	
	\section{Experiments}\label{section:experiments}
	
	In this section we test the computational performance of our \ourset relaxation \eqref{sparse_sdp}, and its variant \eqref{sparse_sdp_plus}. We consider four applications: (i) sparse PCA, (ii) sparse ridge regression, (iii) bounding the restricted isometry constant,  and (iv) sparse canonical correlation analysis. We compare our methods against heuristic methods and alternative SDP relaxations. 
	Our experiments were ran in \texttt{Julia} using a MacBook Pro.
	The SDPs are solved with the commercial solver \texttt{COPT}.

	\subsection{Sparse PCA}
	We consider two types of datasets:
	\begin{enumerate}
		\item Synthetic: \(\Sigma\) is generated from the spiked Wigner and spiked Wishart models.
		\item Real: \(\Sigma\) is the empirical covariance matrix from four datasets: \emph{Pitprops} (\(n=13\)), \emph{Eisen-1} (\(n=79\)), \emph{Eisen-2} (\(n=118\)), and \emph{Colon} (\(n=500\)).
	\end{enumerate}
	
	We compare our methods \eqref{sdp_spca} and its variant under \eqref{sparse_sdp_plus} against the following benchmarks:
	\begin{itemize}
		\item SDP-\(\ell_1\): Special case of \eqref{eq:sdpK} with $K = \Sone{n}{k}$; see \cite{d2004direct}.
		\item SDP-LX: Special case of \eqref{eq:sdpK} with \(K=\Sz{n}{k}\); see \cite{li2025exact}.
		
		\item SDP-BCP: 
		Let \(\Sbs{n}{k}\) be defined as follows,
		\begin{equation}
			\begin{split}
				\Sbs{n}{k} = \{X: &\; \exists z \in \RR^n \st 
				\|X_{i,:}\|^2\leq X_{ii}z_i,\;0\leq z_i\leq \tr(X), i\in[n]\\
				&X_{ij}\leq M_{ij}z_i,\;\forall i,j\in[n],\quad  \sum_{i=1}^n z_i = k \tr(X),X\in \Sone{n}{k}\}
			\end{split}
		\end{equation}
		where \(M_{ij} = 1\) if \(i=j\), \(M_{ij} = \tfrac{1}{2}\) if \(i\neq j\). Then their SDP is  a special case of  \eqref{eq:sdpK} with \(K=\Sbs{n}{k}\); see \cite{bertsimas2022solving}.
		\item TPow: truncated power method (TPower\cite{yuan2013truncated})
		\item TPCA: naive truncated PCA, namely, truncating the PCA solution to the \(k\) largest entries.
	\end{itemize}
	For the SDPs, we take the optimal objective value of the relaxation as an upper bound, and we obtain a corresponding lower bound by truncating its optimal solution via TPCA and evaluating the resulting objective.
	For the heuristic methods TPower and TPCA, we report only the objective value attained by their solutions as a lower bound; upper bounds are not available.
	
	\subsubsection{Spiked Models with synthetic data}
	
	Consider the spiked Wigner and Wishart models, as discussed in \Cref{subsection:spiked_model}.
	We generate the spike vector \(\bar x\in \RR^n\) with \(\|\bar x\|_0=20\), where its nonzero entries are i.i.d. standard Gaussian. Then we normalize \(\bar x\) such that \(\|\bar x\|=1\).
	We evaluate \(100\) independent random instances of dimension \(n=50\).
	For each instance, we compute the true optimal solution and, for every method, form the ratio of the true optimal objective value to that method’s lower bound. We then plot the empirical CDF of this ratio across all instances. By construction, the ratio is at least one; a value of \(1\) indicates exact recovery, and lower ratios (closer to \(1\)) reflect better performance. For the two spiked models, we consider sparsity \(k \in \{3,4,5,6\}\). As shown in \Cref{fig:wigner_spike_model} and \Cref{fig:wishart_spike_model}, SDP-S1 is one of the top-performing method.

	\begin{figure}[ht]
		\centering 
		\vspace{-0.35cm} 
		\subfigtopskip=2pt 
		\subfigbottomskip=2pt
		\subfigcapskip=-5pt 
		\subfigure[\(k=3\)]{
			\includegraphics[width=0.45\linewidth]{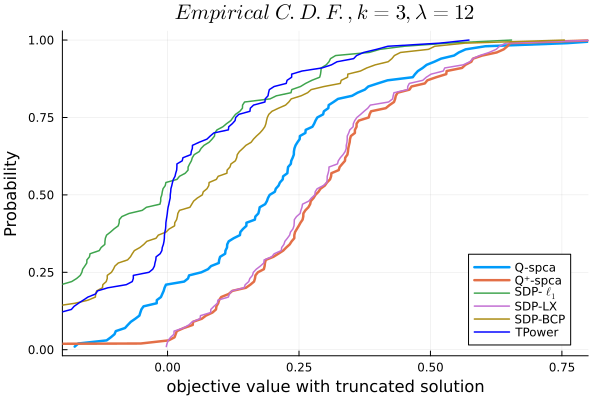}}
		\subfigure[\(k=4\)]{
			\includegraphics[width=0.45\linewidth]{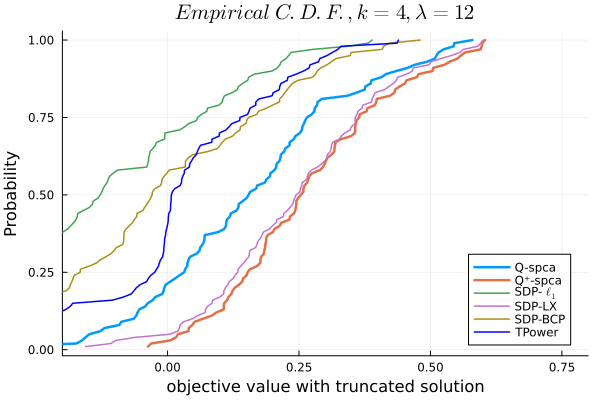}}
		\subfigure[\(k=5\)]{
			\includegraphics[width=0.45\linewidth]{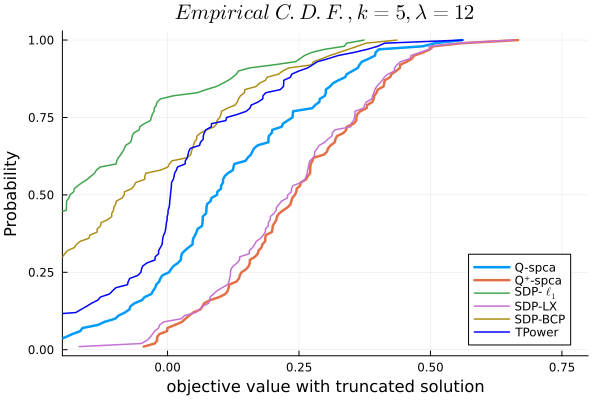}}
		\subfigure[\(k=6\)]{
			\includegraphics[width=0.45\linewidth]{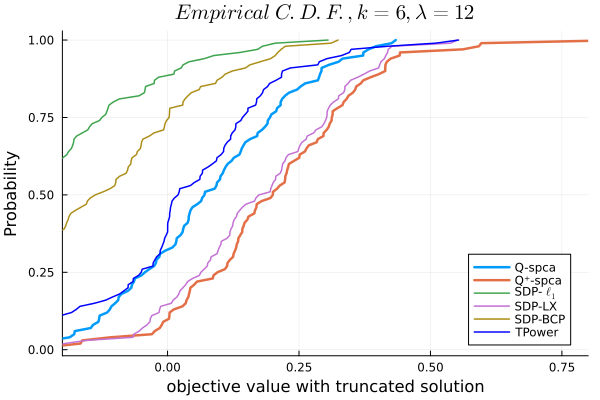}}
		\caption{Wigner spiked model: $n=50,\lambda=12$}
		\label{fig:wigner_spike_model}
	\end{figure}
	
	\begin{figure}[ht]
		\centering 
		\vspace{-0.35cm} 
		\subfigtopskip=2pt 
		\subfigbottomskip=2pt
		\subfigcapskip=-5pt 
		\subfigure[\(k=3\)]{
			\includegraphics[width=0.45\linewidth]{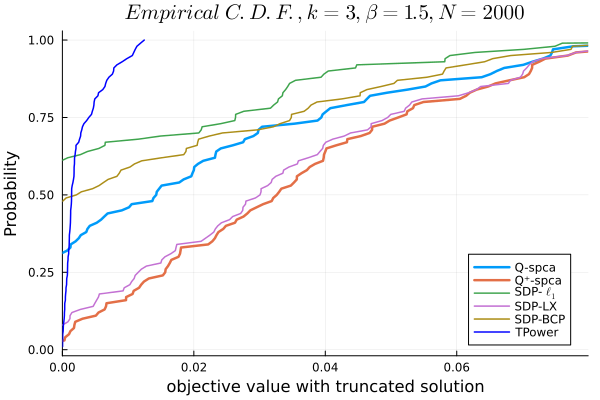}}
		\subfigure[\(k=4\)]{
			\includegraphics[width=0.45\linewidth]{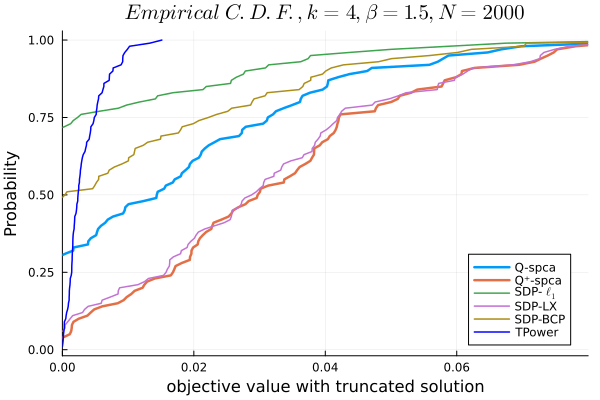}}
		
		\subfigure[\(k=5\)]{
			\includegraphics[width=0.45\linewidth]{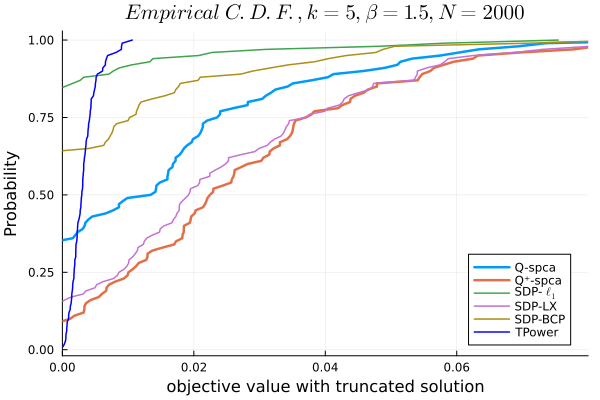}}
		\subfigure[\(k=6\)]{
			\includegraphics[width=0.45\linewidth]{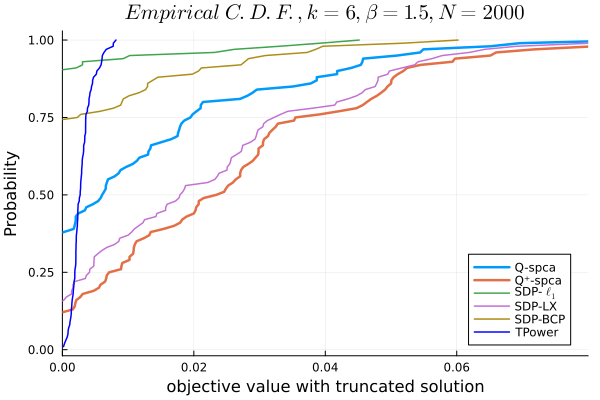}}
		
		\caption{Wishart spiked model: $n = 50, N=2000, \beta=1.5$}
		\label{fig:wishart_spike_model}
	\end{figure}
	
	\subsubsection{Real datasets}
	For each dataset, we report the upper and lower bounds achieved by each method as the sparsity level \(k\) varies; note that TPower and TPCA only provide lower bounds. For \emph{Pitprops} (\(n=13\)), SDP-S obtains the true solution when \(k=5\), and its variant SDP-S1 attains exact solutions throughout \(k\in \{4,5,6,7\}\). On the remaining datasets, \emph{Eisen-1} (\(n=79\)), \emph{Eisen-2} (\(n=118\)), and \emph{Colon} (\(n=500\))\footnote{Owing to solver limitations, we report only the instances solved successfully.}, SDP-S attains the global optimum\footnote{Upper and lower bound coincide.} across all tested values of \(k\).
	Results are presented in \Cref{fig:spca_pitprops}, \Cref{fig:spca_eisen_1}, \Cref{fig:spca_eisen_2}, and \Cref{fig:spca_colon}.
	
	\begin{figure}[htb!]
		\centering
		\includegraphics[width=0.9\linewidth]{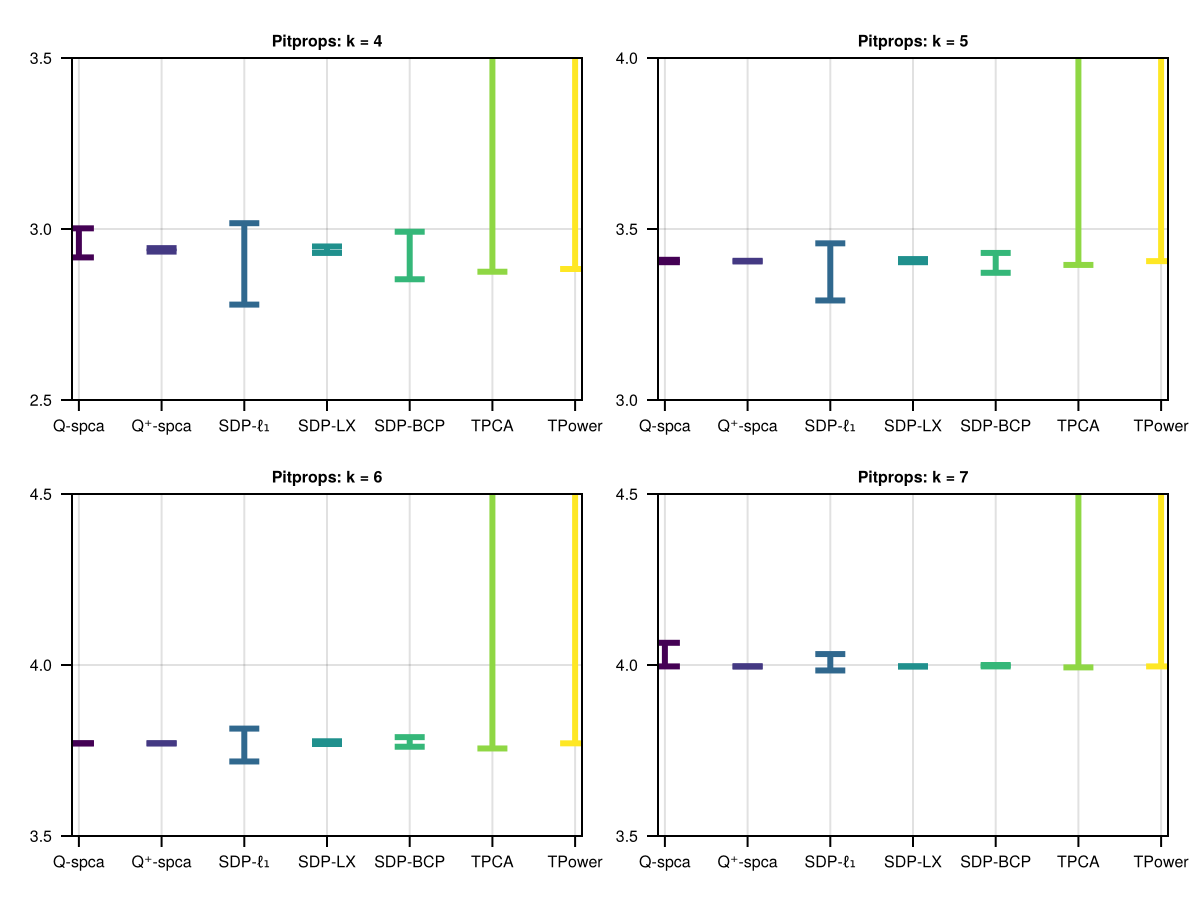}
		\caption{Comparison of upper and lower bound with Pitprops dataset ($n = 13$).}
		\label{fig:spca_pitprops}
	\end{figure}
	
	\begin{figure}[ht]
		\centering
		\includegraphics[width=0.9\linewidth]{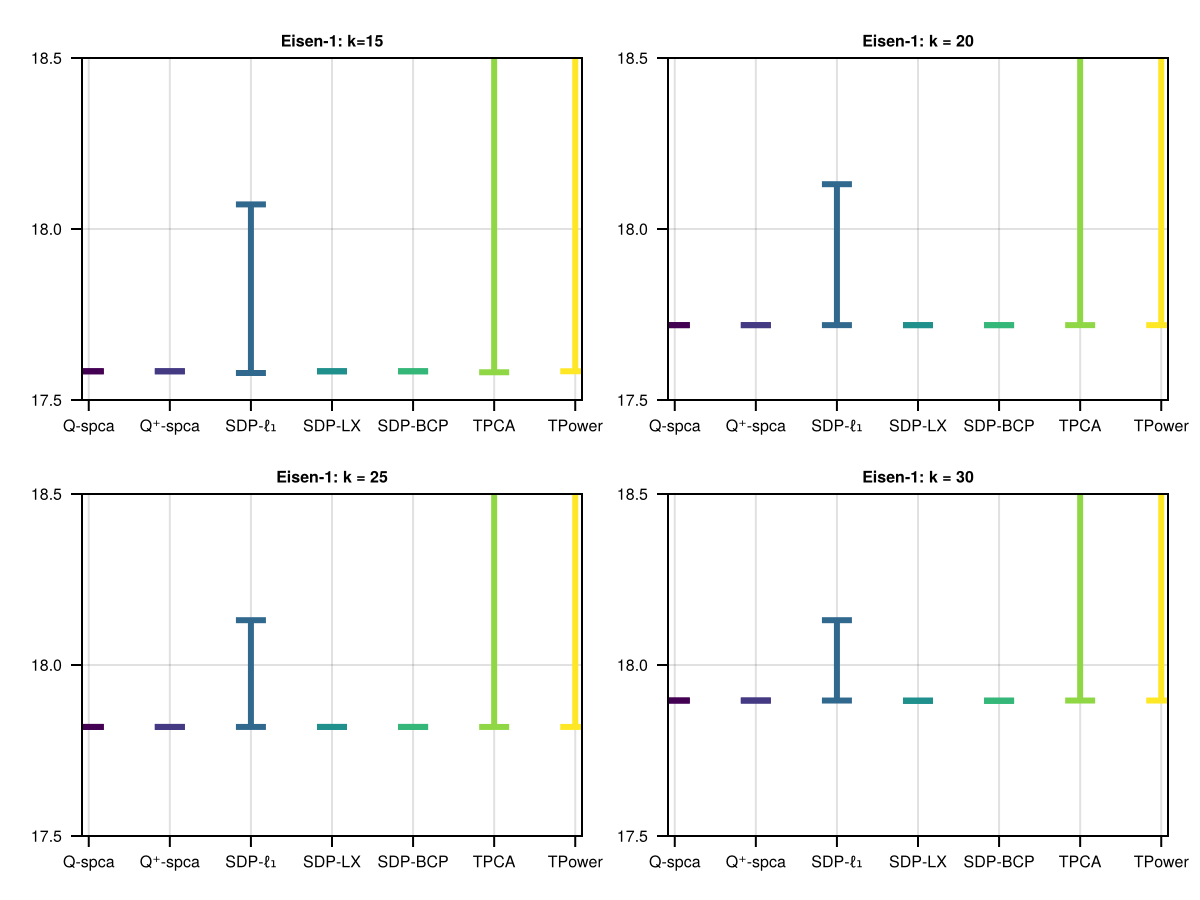}
		\caption{Comparison of upper and lower bound with Eisen-1 dataset ($n = 79$).}
		\label{fig:spca_eisen_1}
	\end{figure}
	
	\begin{figure}[ht]
		\centering
		\includegraphics[width=0.9\linewidth]{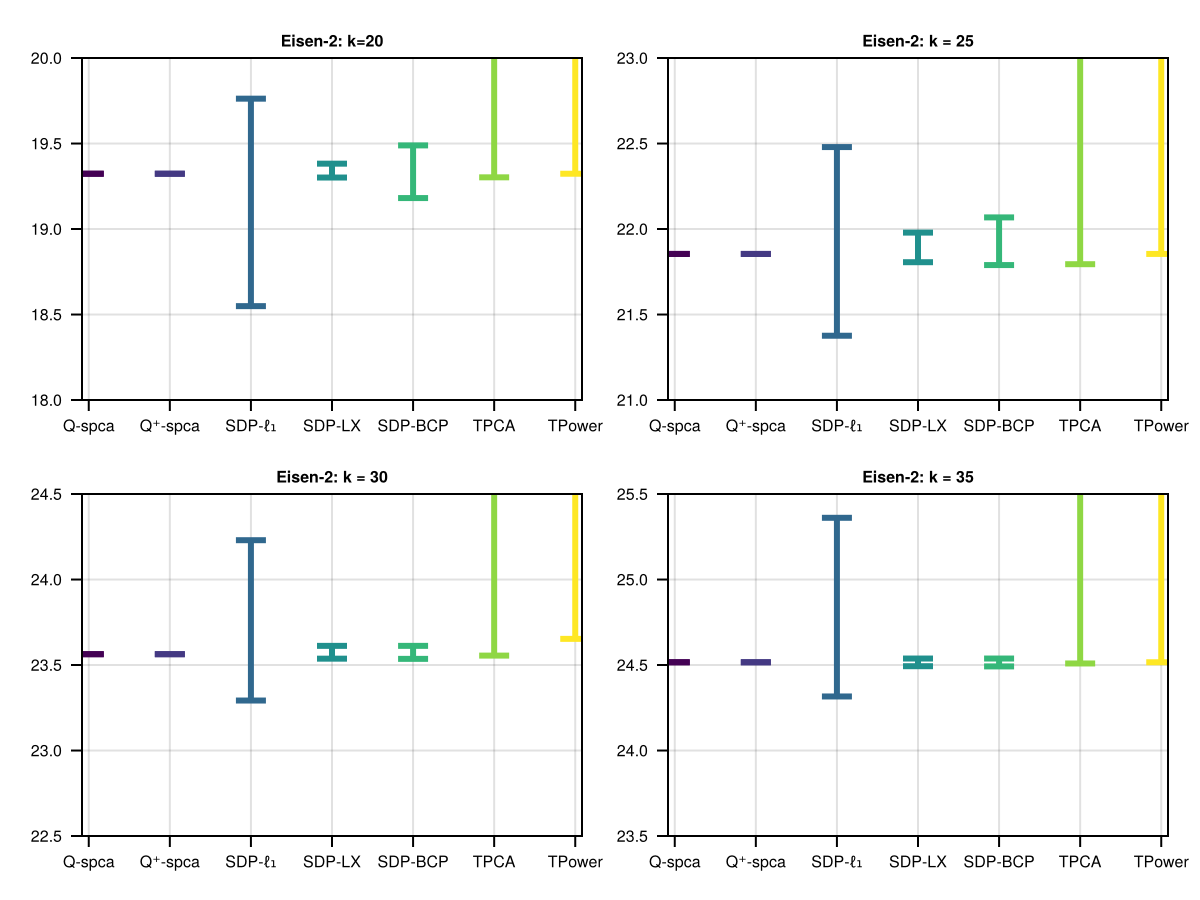}
		\caption{Comparison of upper and lower bound with Eisen-2 dataset ($n = 118$).}
		\label{fig:spca_eisen_2}
	\end{figure}
	
	\begin{figure}[ht]
		\centering
		\includegraphics[width=0.9\linewidth]{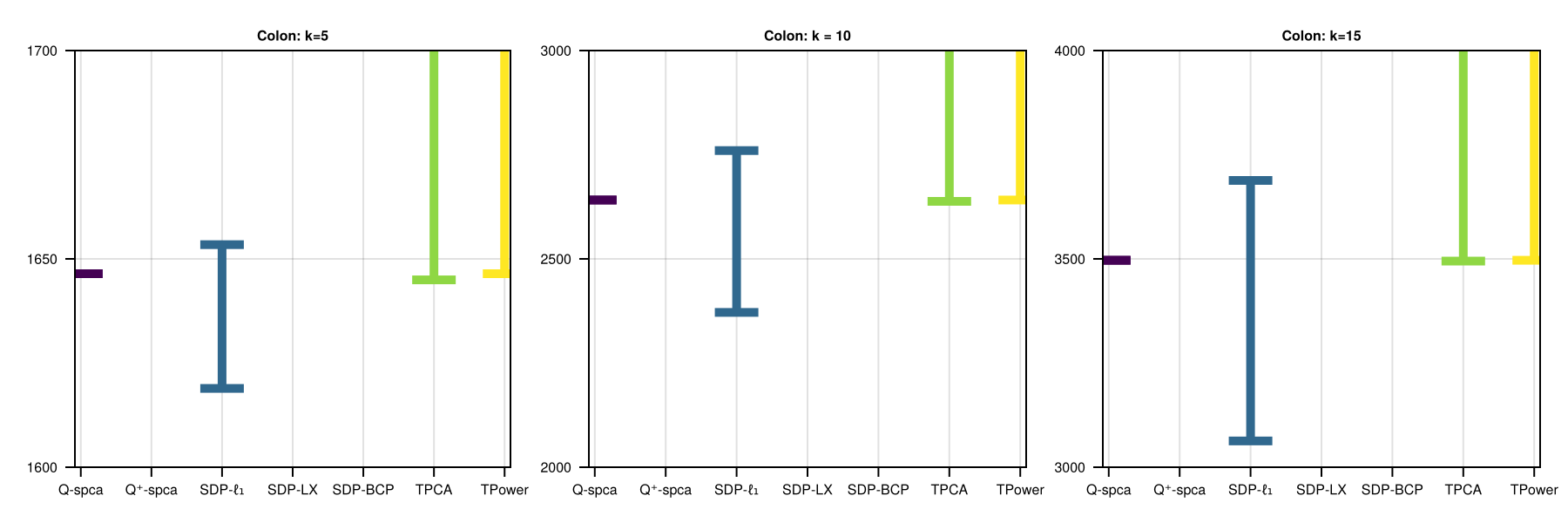}
		\caption{Comparison of upper and lower bound with Colon dataset ($n = 500$).}
		\label{fig:spca_colon}
	\end{figure}

	\subsection{Sparse Ridge Regression}
	
	We consider a synthetic dataset
	for the linear model \(y=A\bar x+\varepsilon\).
	The matrix \(A \in \RR^{m \times n}\), ground-truth vector \(\bar x\in\RR^{n}\), and noise vector \(\varepsilon\) have i.i.d. standard Gaussian entries.
	We set the size \((m,n) = (30,50)\), sparsity \(k=10\) and the noise level \(\sigma=3\) (i.e., \(\varepsilon\sim\mathcal{N}(0,\sigma^{2}I_m)\)).
	
	We compare our methods \eqref{sdp_slr} and its variant under \eqref{sparse_sdp_plus} against the following benchmarks:
	\begin{itemize}
		\item SDP relaxations based on \(\Sone{n}{k}\);
		\item  Greedy algorithm \cite{xie2020scalable};
		\item Iterative hard thresholding (IHT), and hard thresholding pursuit (HTP) \cite{yuan2018gradient}
	\end{itemize}
	
	For each regularization constant \(\lambda \in \{0.1, 0.3, 0.5\}\), we generate \(100\) independent random instances and compare the empirical CDF of the attained objective values from each method.
	For SDP-based methods, when the relaxation is not exact, we truncate the first row of the solution to obtain a \(k\)-sparse solution and report the corresponding objective value.
	
	In the plots, smaller objective value indicates better performance.  \eqref{sdp_slr} and its variant under \eqref{sparse_sdp_plus} are the best two methods. The results are shown in \Cref{fig:sparse_ridge_cdf}.
	
	\begin{figure}[ht!]
		\centering 
		\vspace{-0.35cm} 
		\subfigtopskip=2pt 
		\subfigbottomskip=2pt
		\subfigcapskip=-5pt 
		\subfigure[$\lambda = 0.1$]{
			\label{fig_slr_cdf_1}
			\includegraphics[width=0.3\linewidth]{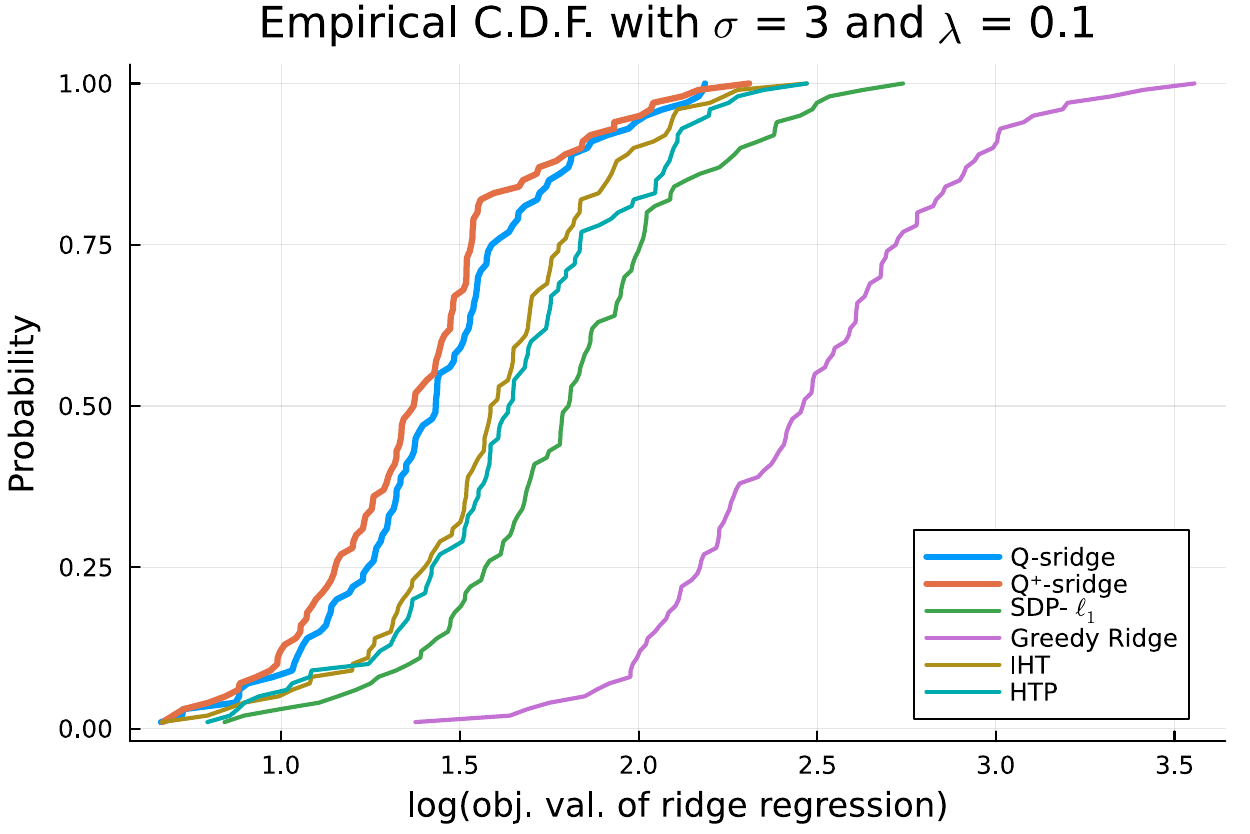}}
		\subfigure[$\lambda = 0.3$]{
			\label{fig_slr_cdf_2}
			\includegraphics[width=0.3\linewidth]{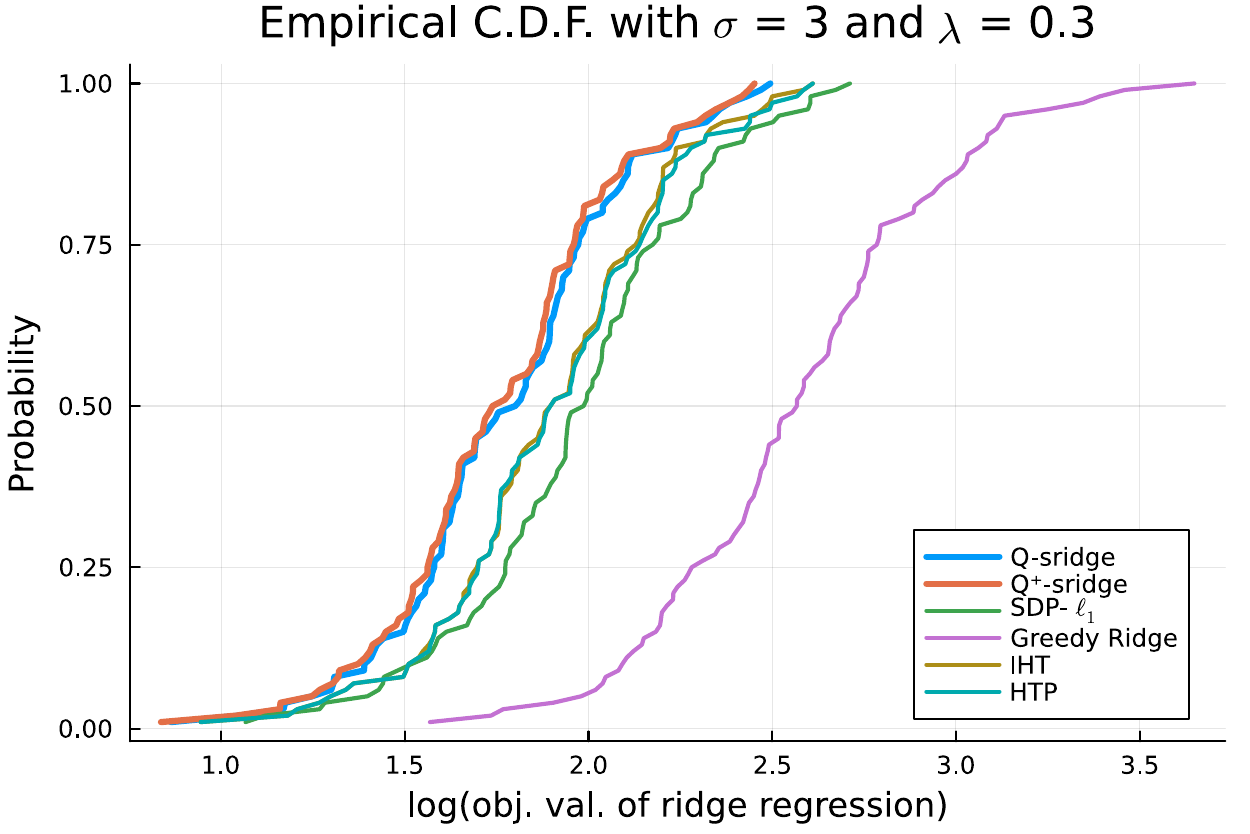}}
		\subfigure[$\lambda = 0.5$]{
			\label{fig_slr_cdf_3}
			\includegraphics[width=0.3\linewidth]{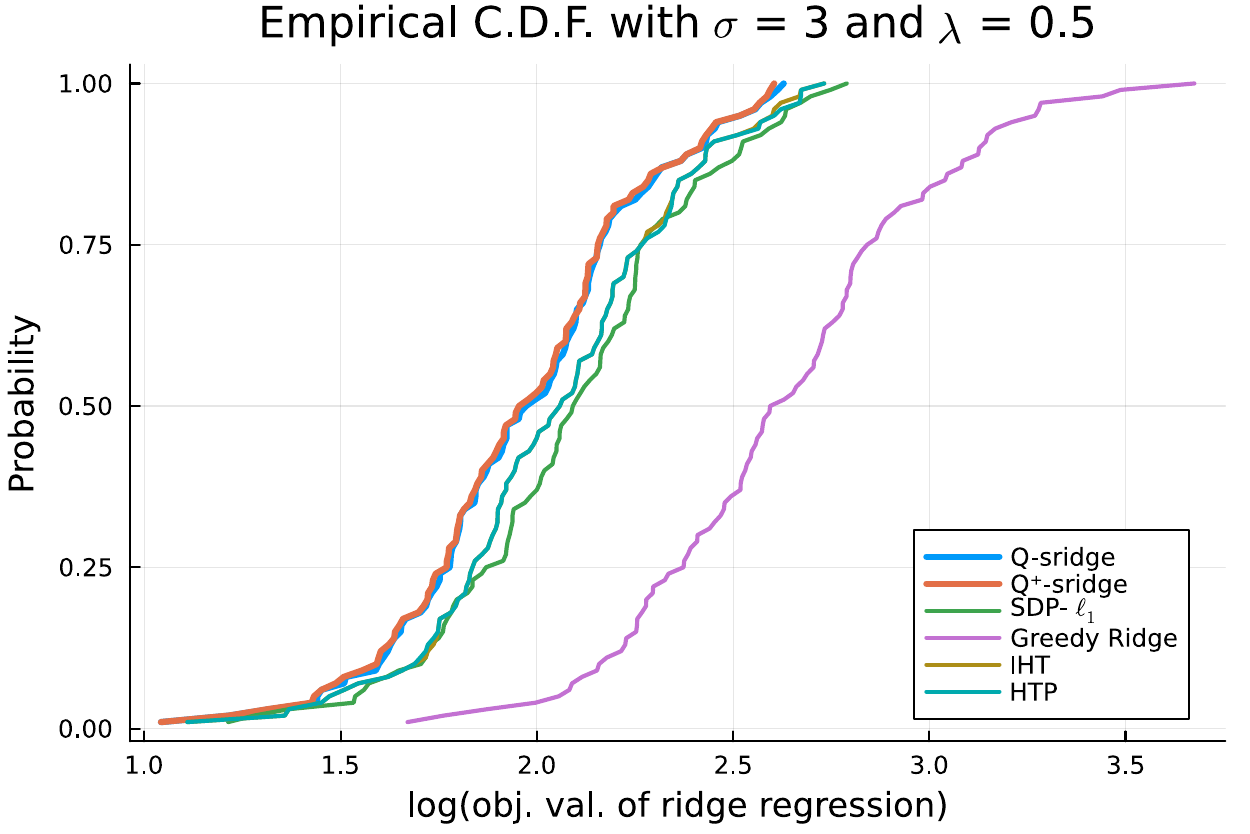}}
		
		\caption{Sparse ridge regression: empirical CDF of objective value.}
		\label{fig:sparse_ridge_cdf}
	\end{figure}

	\subsection{Restricted Isometry Property}\label{section:experiments_ric}
	As discussed in \Cref{subsection:spca_ric}, we consider the problem of determining the optimal RIP constants $\delta^-,\delta^+$. 
	We estimate the RIP constants using \eqref{sdp_spca} with $C = \pm A^T A$.
	
	We consider the synthetic dataset from \cite{baraniuk2008simple}.
	The sensing matrix \(A\in\RR^{m\times n}\) has i.i.d.\ entries, drawn from one of three distributions:
	\text{(i) Gaussian:} 
	\begin{equation}\label{eq:rip_distribution_normal}
		a_{ij}\sim \mathcal{N}(0,1/m),    
	\end{equation}
	\text{(ii) Symmetric Bernoulli (scaled):}
	\small{
    \begin{equation}\label{eq:rip_distribution_bern2}
		a_{ij}=\begin{cases}
			+1/\sqrt{m} & \text{with probability } \tfrac{1}{2},\\[2pt]
			-1/\sqrt{m} & \text{with probability } \tfrac{1}{2},
		\end{cases}
	\end{equation}
	\text{(iii) Sparsified symmetric Bernoulli (three-point):}
	}
    \small{
    \begin{equation}\label{eq:rip_distribution_bern3}
		a_{ij}=\begin{cases}
			+1/\sqrt{m} & \text{with probability } \tfrac{1}{6},\\[2pt]
			0            & \text{with probability } \tfrac{2}{3},\\[2pt]
			-1/\sqrt{m} & \text{with probability } \tfrac{1}{6}.
		\end{cases}
	\end{equation}
	}
    
	We compare the performance of Q-spca, \(\text{Q}^+\)-spca, and SDP-\(\ell_1\).
	We fix \((n,m)=(80,30)\) and \(k=10\), and, for each of the three entry distributions, generate \(100\) independent random instances.
	For every instance and each method, we compute an upper bound on the restricted isometry constant (RIC) and plot the empirical CDF across the trials.
	As shown in \Cref{fig:rip}, \(\text{Q}^+\)-spca yields tighter bounds than those based on  Q-spca or SDP-\(\ell_1\).
	
	\begin{figure}[ht!]
		\centering 
		\vspace{-0.35cm} 
		\subfigtopskip=2pt 
		\subfigbottomskip=2pt
		\subfigcapskip=-5pt 
		\subfigure[Normal: \eqref{eq:rip_distribution_normal}]{
			\label{normal}
			\includegraphics[width=0.3\linewidth]{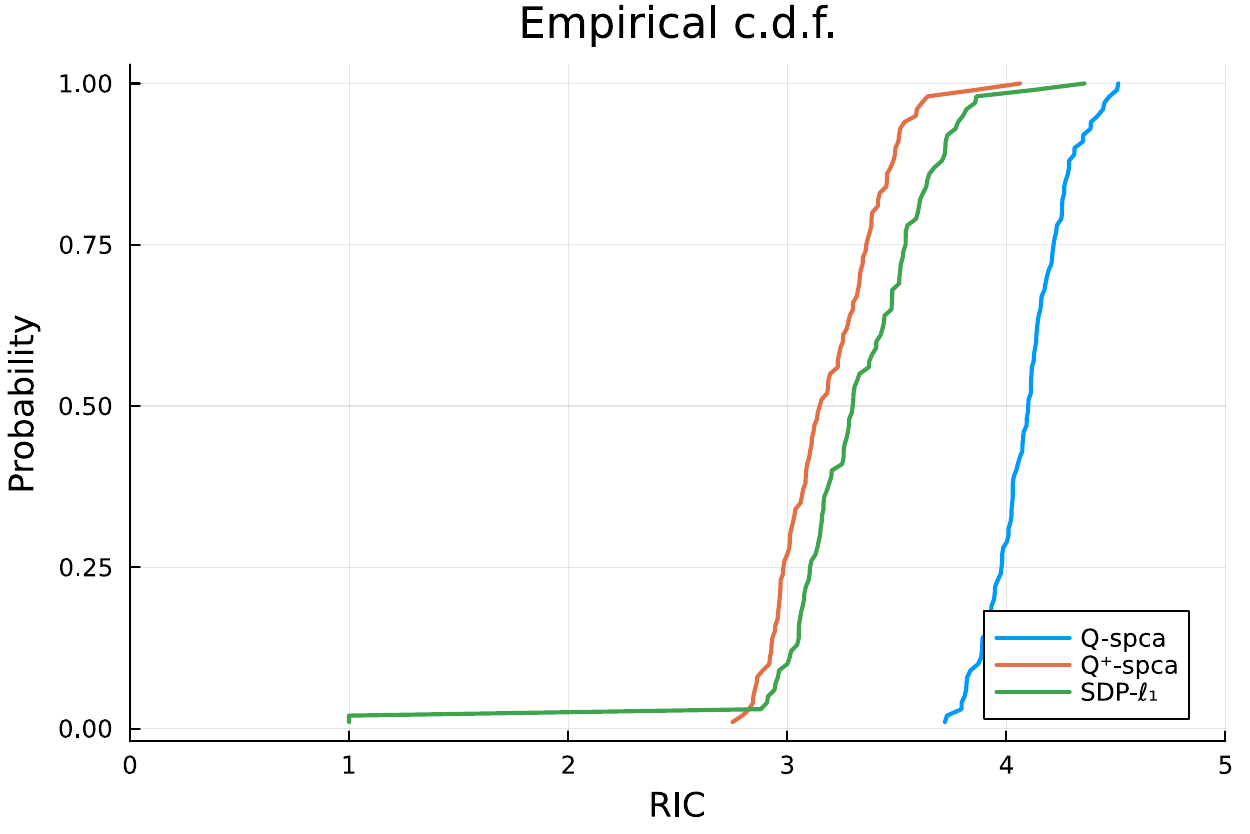}}
		\quad 
		\subfigure[Bernoulli: \eqref{eq:rip_distribution_bern2}]{
			\label{bern_2}
			\includegraphics[width=0.3\linewidth]{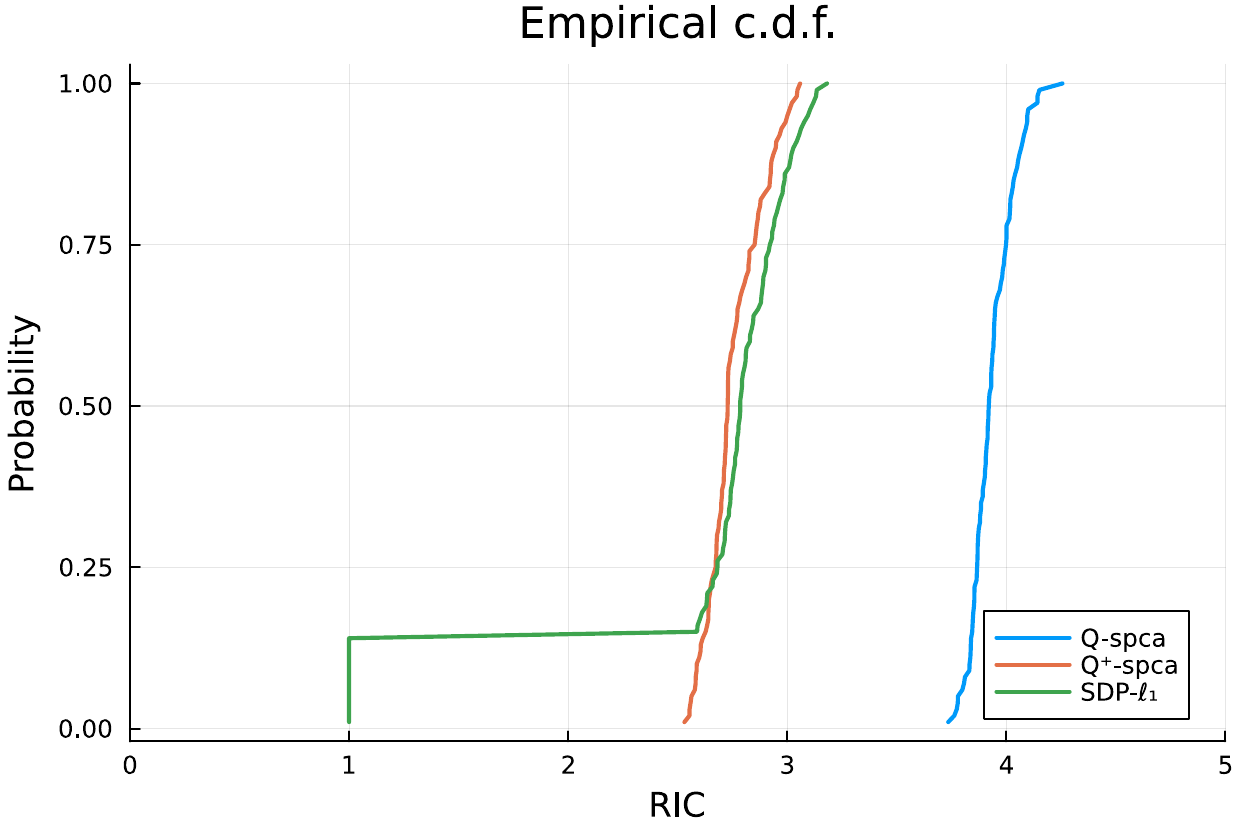}}
		\subfigure[Bernoulli: \eqref{eq:rip_distribution_bern3}]{
			\label{bern_3}
			\includegraphics[width=0.3\linewidth]{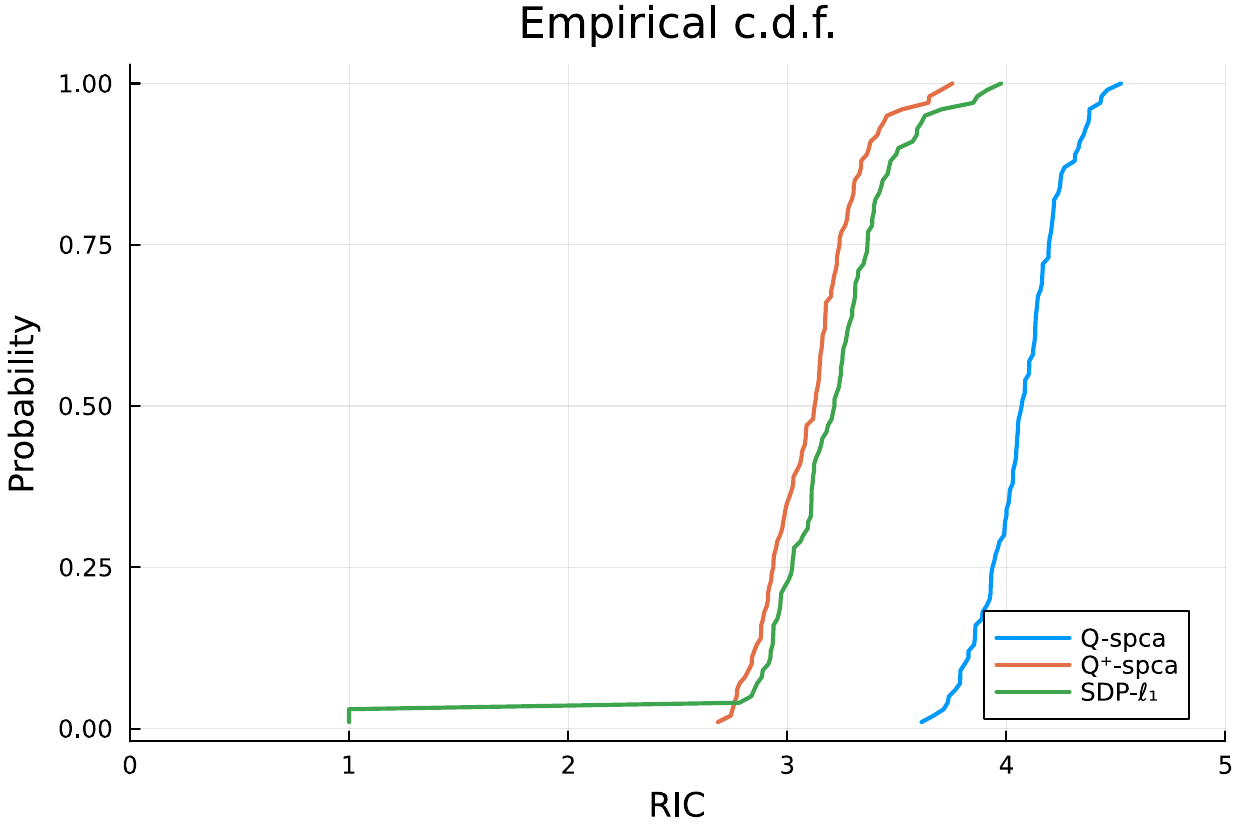}}
		
		\caption{Empirical CDF of restricted isometry constant $m=80, n = 30, k=10$}
		\label{fig:rip}
	\end{figure}
	
	\subsection{Sparse Canonical Correlation Analysis}
	In order to extract the interpretable correlation of two groups of data, sparse canonical correlation analysis problem\cite{chu2013sparse,hardoon2011sparse,li2024sparse} is considered as follow,
	\begin{equation}
		\tag{\small{s-cca}}\label{sparse_cca}
		\begin{aligned}
			\max_{u,v} \quad  u^T\Sigma_{xy}v
			\quad\text{ s.t. }\quad u^T\Sigma_{xx}u\leq 1, v^T\Sigma_{yy}v\leq 1,
			\|u\|_0 \leq k_1, \|v\|_0 \leq k_2
		\end{aligned}
	\end{equation}
	
	For the above problem, \eqref{sparse_sdp} can be formulated as
	\begin{equation}\label{sdp_scca}
		\tag{\small{Q-scca}}
		\begin{aligned}
			\max_{X\in \mathbb{S}^n} \quad & C \bullet X\\
			\quad\text{ s.t. }\quad
			&X_{11} \bullet \Sigma_{xx} \leq 1,\quad
			k_1\diag(X_{11}) - X_{11} \succeq 0\\
			&X_{22}\bullet \Sigma_{yy} \leq 1,\quad
			k_2\diag(X_{22}) - X_{22} \succeq 0\\
			&X := \begin{pmatrix}
				X_{11} & X_{12}\\
				X_{21} & X_{22}
			\end{pmatrix}\succeq 0
		\end{aligned}
	\end{equation}
	where $C = \begin{pmatrix}
		0 & \frac{1}{2}\Sigma_{xy}\\
		\frac{1}{2}\Sigma_{xy}^T&0
	\end{pmatrix}.$
	
	We randomly generate covariance matrix $\Sigma = \begin{pmatrix}
		\Sigma_{xx} &\Sigma_{xy}\\
		\Sigma_{xy}^T & \Sigma_{yy}
	\end{pmatrix}$ with the following two structure,
	
	(i) Spiked covariance models:
	\begin{align*}
		\Sigma_{xx} = I_{n_1} + XX^T,\quad
		\Sigma_{yy} = I_{n_2} + YY^T,\quad
		\Sigma_{xy} = \rho \Sigma_{xx}uv^T\Sigma_{yy}
	\end{align*}
	where entries of $X\in\RR^{n_1\times n_1}$, $Y\in\RR^{n_2\times n_2}$ are i.i.d. standard Gaussian random variables.
	
	(ii) Toeplitz matrices:
	\begin{align*}
		\Sigma_{xx} &= (\sigma_{ij}),~\text{where}~\sigma_{ij}=r_1^{|i-j|}, \forall i,j\in[n_1]\\
		\Sigma_{yy} &= (\sigma'_{ij}),~\text{where}~\sigma'_{ij}=r_2^{|i-j|}, \forall i,j\in[n_2]\\
		\Sigma_{xy} &= \rho \Sigma_{xx}uv^T\Sigma_{yy}
	\end{align*}
	where $r_1,r_2$ are constants.
	
	In both of those two models, $\rho$ is a random variable with uniform distribution from 0 to 1. We randomly generate two vector \(\bar u\in\RR^{n_1},\bar v\in\RR^{n_2}\) whose entries are i.i.d. standard Gaussian. Then we truncate and normalize $\bar u, \bar v$ to generate $u,v$ such that $\|u\|_0=k_1$, $\|v\|_0=k_2$, $u^T \Sigma_{xx} u =1$ and $v^T \Sigma_{yy} v=1$. We randomly sample $m = 3000$ vectors from $\mathcal{N}(0,\Sigma)$ to generate empirical covariance matrix. 
	
	We set \((n_1,k_1;n_2,k_2) = (25,5;20,5)\), and \(r_1=r_2=0.7\). We generate \(100\) independent random instances and compare Q-scca, \(\text{Q}^+\)-scca, and SDP-\(\ell_1\).
	For each instance and method, we take the relaxation’s optimal objective value as an upper bound and the objective value of the rounded sparse solution as a lower bound, and plot the empirical CDF of (i) the ratio of upper to lower bound; (ii) objective value of the rounded sparse solution across trials. As shown in \Cref{fig:scca_ratio},  both of Q-scca, \(\text{Q}^+\)-scca outperform SDP-\(\ell_1\).
	
	\begin{figure}[ht]
		\centering 
		\vspace{-0.35cm} 
		\subfigtopskip=2pt 
		\subfigbottomskip=2pt
		\subfigcapskip=-5pt 
		\subfigure[Spiked covariance model: empirical CDF of ratio of upper to lower bound.]{
			\label{fig_scca_ratio_spike}
			\includegraphics[width=0.4\linewidth]{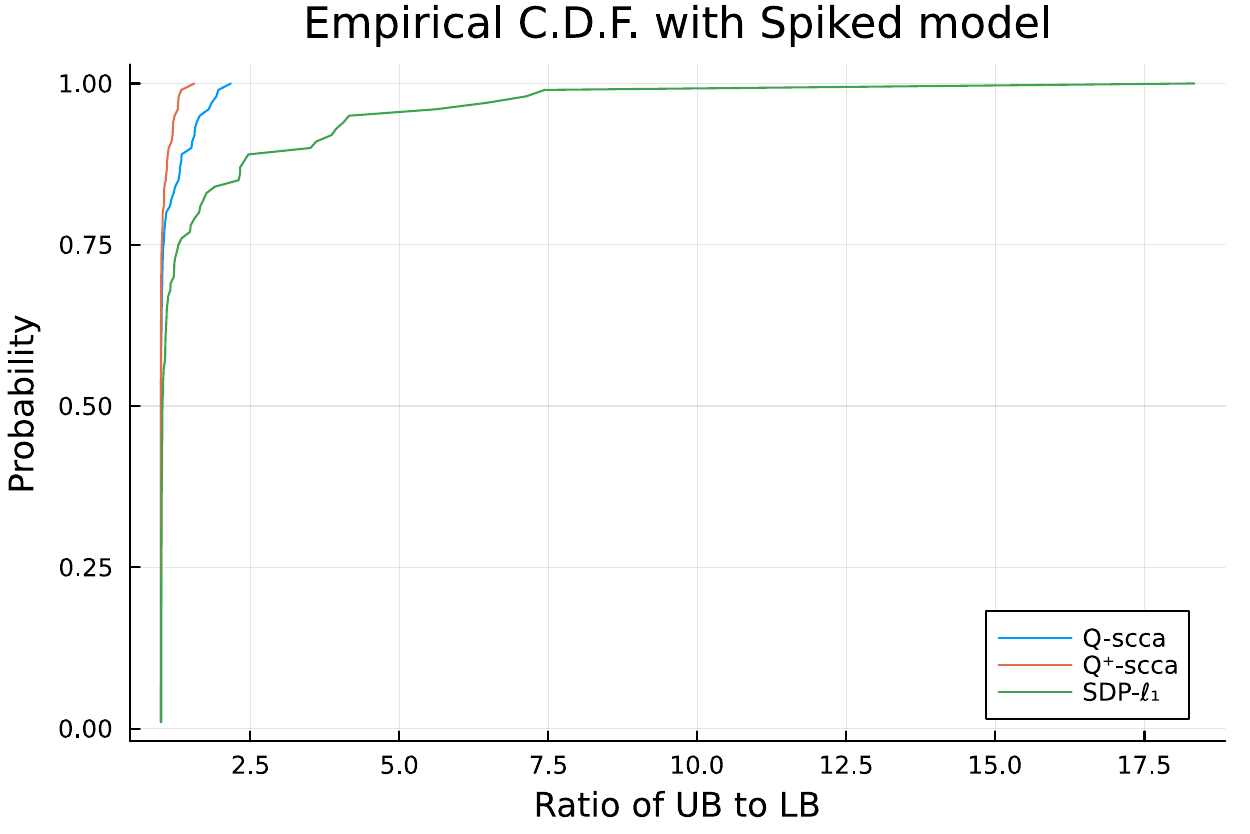}}
		\quad 
		\subfigure[Spiked covariance model: empirical CDF of objective value of sparse solution.]{
			\label{fig_scca_ratio_2}
			\includegraphics[width=0.4\linewidth]{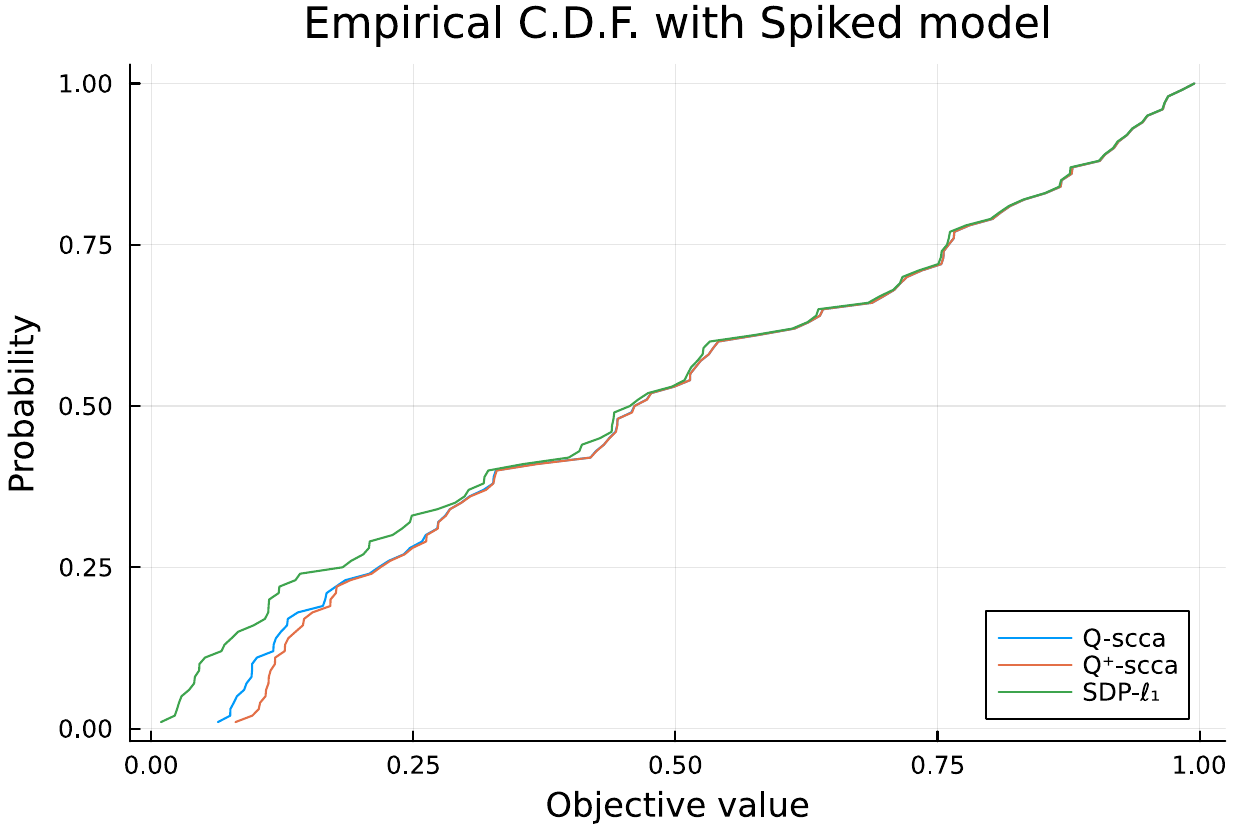}}
		
		\subfigure[Toeplitz covariance model: empirical CDF of ratio of upper to lower bound.]{
			\label{fig_scca_ratio_3}
			\includegraphics[width=0.4\linewidth]{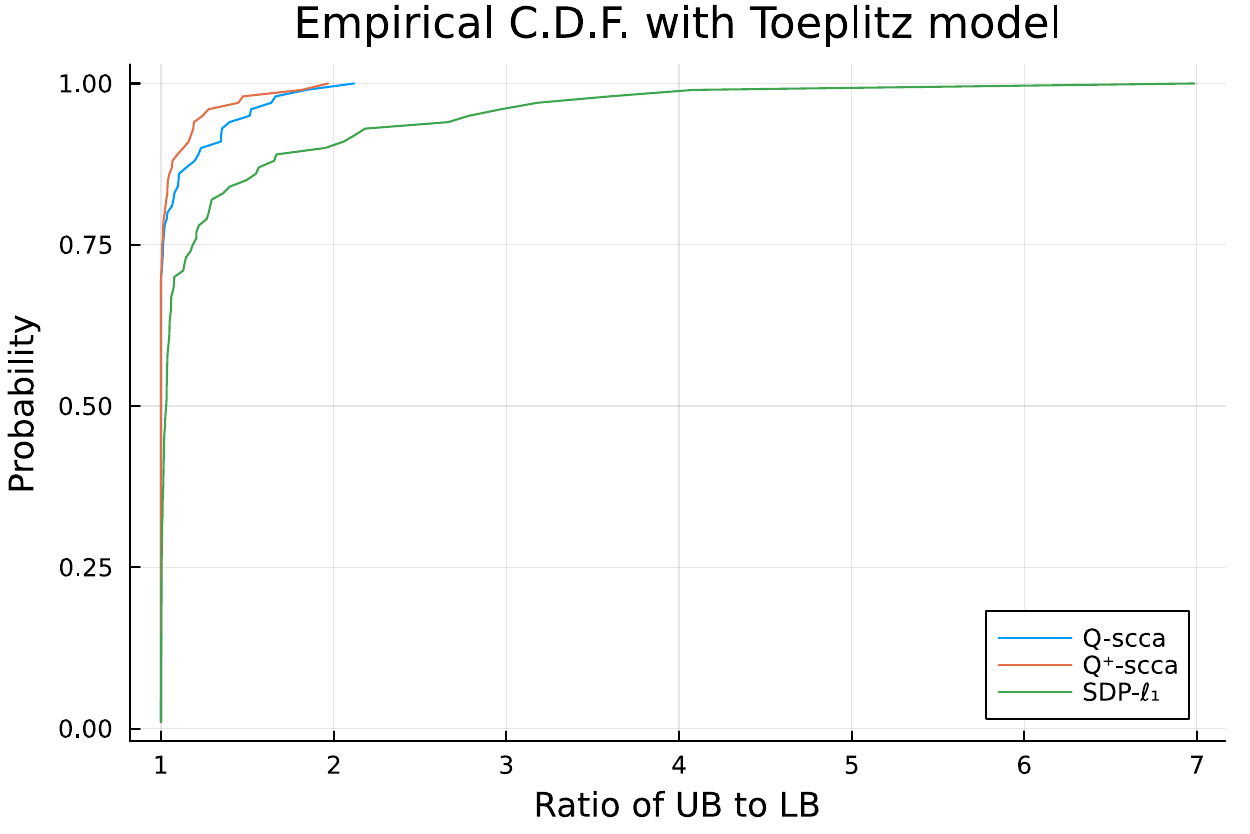}}
		\quad
		\subfigure[Toeplitz covariance model: empirical CDF of objective value of sparse solution.]{
			\label{fig_scca_ratio_4}
			\includegraphics[width=0.4\linewidth]{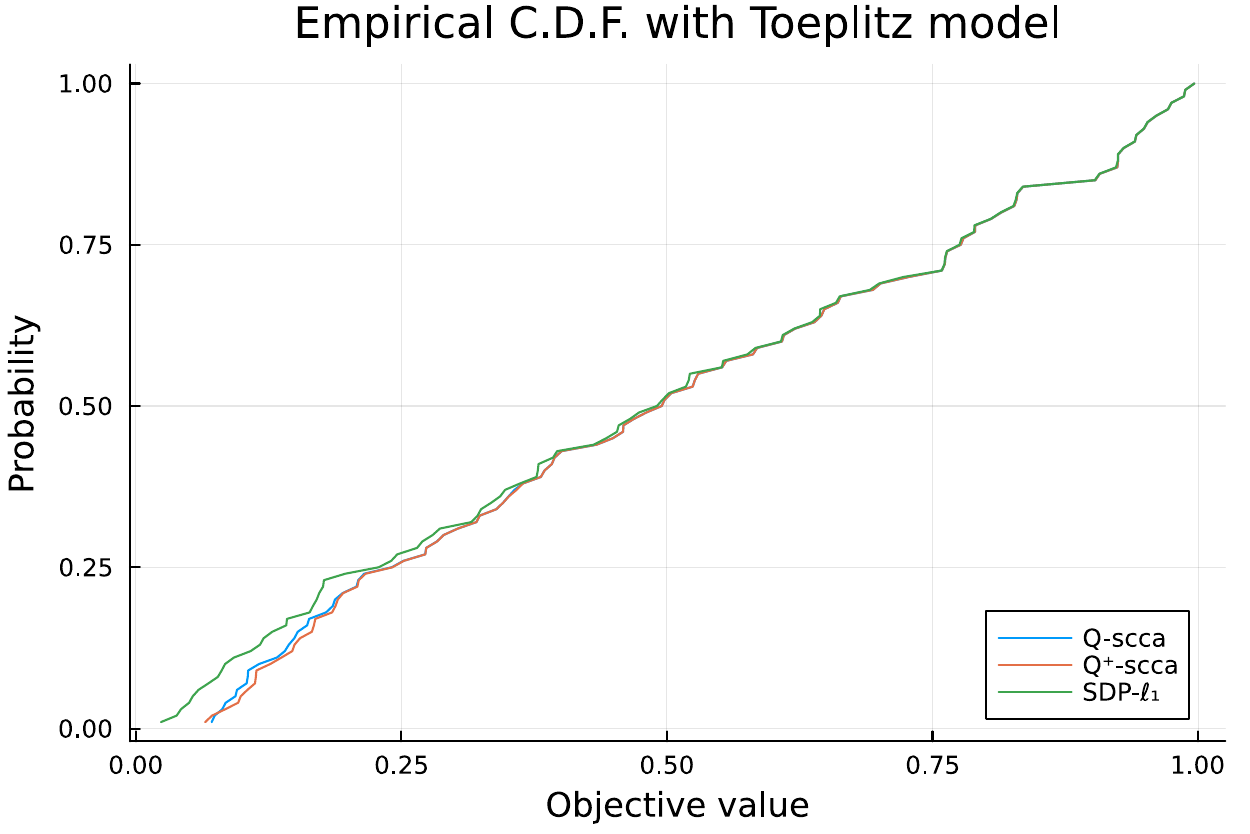}}
		
		\caption{SCCA: empirical CDF.}
		\label{fig:scca_ratio}
	\end{figure}

    \section*{Declarations}
    
    \subsection*{Funding}
    This work was partially supported by the Office of Naval Research, N00014-23-1-2631. 
    
    \subsection*{Competing interests}
    The authors have no competing interests to declare that are relevant to the content of this article.
	
	\bibliographystyle{abbrv}
	\bibliography{ref.bib}

@article{d2004direct,
  title={A direct formulation for sparse PCA using semidefinite programming},
  author={d'Aspremont, Alexandre and Ghaoui, Laurent and Jordan, Michael and Lanckriet, Gert},
  journal={Advances in neural information processing systems},
  volume={17},
  year={2004}
}

@article{journee2010generalized,
  title={Generalized power method for sparse principal component analysis.},
  author={Journ{\'e}e, Michel and Nesterov, Yurii and Richt{\'a}rik, Peter and Sepulchre, Rodolphe},
  journal={Journal of Machine Learning Research},
  volume={11},
  number={2},
  year={2010}
}

@article{luss2013conditional,
  title={Conditional Gradient Algorithms for Rank-One Matrix Approximations with a Sparsity Constraint},
  author={Luss, Ronny and Teboulle, Marc},
  journal={SIAM Review},
  volume={55},
  number={1},
  pages={65},
  year={2013},
  publisher={Society for Industrial and Applied Mathematics}
}

@article{hein2010inverse,
  title={An inverse power method for nonlinear eigenproblems with applications in 1-spectral clustering and sparse PCA},
  author={Hein, Matthias and B{\"u}hler, Thomas},
  journal={Advances in neural information processing systems},
  volume={23},
  year={2010}
}

@article{yuan2013truncated,
  title={Truncated Power Method for Sparse Eigenvalue Problems.},
  author={Yuan, Xiao-Tong and Zhang, Tong},
  journal={Journal of Machine Learning Research},
  volume={14},
  number={4},
  year={2013}
}

@article{moghaddam2005spectral,
  title={Spectral bounds for sparse PCA: Exact and greedy algorithms},
  author={Moghaddam, Baback and Weiss, Yair and Avidan, Shai},
  journal={Advances in neural information processing systems},
  volume={18},
  year={2005}
}

@article{berk2019certifiably,
  title={Certifiably optimal sparse principal component analysis},
  author={Berk, Lauren and Bertsimas, Dimitris},
  journal={Mathematical Programming Computation},
  volume={11},
  pages={381--420},
  year={2019},
  publisher={Springer}
}

@article{gally2018framework,
  title={A framework for solving mixed-integer semidefinite programs},
  author={Gally, Tristan and Pfetsch, Marc E and Ulbrich, Stefan},
  journal={Optimization Methods and Software},
  volume={33},
  number={3},
  pages={594--632},
  year={2018},
  publisher={Taylor \& Francis}
}

@article{li2025exact,
  title={Exact and approximation algorithms for sparse principal component analysis},
  author={Li, Yongchun and Xie, Weijun},
  journal={INFORMS Journal on Computing},
  volume={37},
  number={3},
  pages={582--602},
  year={2025},
  publisher={INFORMS}
}

@article{bertsimas2022solving,
  title={Solving large-scale sparse PCA to certifiable (near) optimality},
  author={Bertsimas, Dimitris and Cory-Wright, Ryan and Pauphilet, Jean},
  journal={Journal of Machine Learning Research},
  volume={23},
  number={13},
  pages={1--35},
  year={2022}
}

@article{d2008optimal,
  title={Optimal Solutions for Sparse Principal Component Analysis.},
  author={d'Aspremont, Alexandre and Bach, Francis and El Ghaoui, Laurent},
  journal={Journal of Machine Learning Research},
  volume={9},
  number={7},
  year={2008}
}

@article{d2014approximation,
  title={Approximation bounds for sparse principal component analysis},
  author={d’Aspremont, Alexandre and Bach, Francis and Ghaoui, Laurent El},
  journal={Mathematical Programming},
  volume={148},
  pages={89--110},
  year={2014},
  publisher={Springer}
}

@article{dey2022using,
  title={Using {$\ell$}1-Relaxation and Integer Programming to Obtain Dual Bounds for Sparse PCA.},
  author={Dey, Santanu S and Mazumder, Rahul and Wang, Guanyi},
  journal={Operations Research},
  volume={70},
  number={3},
  year={2022}
}

@article{hardoon2011sparse,
  title={Sparse canonical correlation analysis},
  author={Hardoon, David R and Shawe-Taylor, John},
  journal={Machine Learning},
  volume={83},
  pages={331--353},
  year={2011},
  publisher={Springer}
}

@article{chu2013sparse,
  title={Sparse canonical correlation analysis: New formulation and algorithm},
  author={Chu, Delin and Liao, Li-Zhi and Ng, Michael K and Zhang, Xiaowei},
  journal={IEEE transactions on pattern analysis and machine intelligence},
  volume={35},
  number={12},
  pages={3050--3065},
  year={2013},
  publisher={IEEE}
}

@book{vershynin2018high,
  title={High-dimensional probability: An introduction with applications in data science},
  author={Vershynin, Roman},
  volume={47},
  year={2018},
  publisher={Cambridge university press}
}

@article{baraniuk2008simple,
  title={A simple proof of the restricted isometry property for random matrices},
  author={Baraniuk, Richard and Davenport, Mark and DeVore, Ronald and Wakin, Michael},
  journal={Constructive approximation},
  volume={28},
  pages={253--263},
  year={2008},
  publisher={Springer}
}

@article{cifuentes2022local,
  title={On the local stability of semidefinite relaxations},
  author={Cifuentes, Diego and Agarwal, Sameer and Parrilo, Pablo A and Thomas, Rekha R},
  journal={Mathematical Programming},
  pages={1--35},
  publisher={Springer}
}

@article{tropp2007signal,
  title={Signal recovery from random measurements via orthogonal matching pursuit},
  author={Tropp, Joel A and Gilbert, Anna C},
  journal={IEEE Transactions on information theory},
  volume={53},
  number={12},
  pages={4655--4666},
  year={2007},
  publisher={IEEE}
}

@article{bahmani2013greedy,
  title={Greedy sparsity-constrained optimization},
  author={Bahmani, Sohail and Raj, Bhiksha and Boufounos, Petros T},
  journal={The Journal of Machine Learning Research},
  volume={14},
  number={1},
  pages={807--841},
  year={2013},
  publisher={JMLR. org}
}

@article{tropp2004greed,
  title={Greed is good: Algorithmic results for sparse approximation},
  author={Tropp, Joel A},
  journal={IEEE Transactions on Information theory},
  volume={50},
  number={10},
  pages={2231--2242},
  year={2004},
  publisher={IEEE}
}

@article{xie2020scalable,
  title={Scalable algorithms for the sparse ridge regression},
  author={Xie, Weijun and Deng, Xinwei},
  journal={SIAM Journal on Optimization},
  volume={30},
  number={4},
  pages={3359--3386},
  year={2020},
  publisher={SIAM}
}

@article{yuan2018gradient,
  title={Gradient hard thresholding pursuit},
  author={Yuan, Xiao-Tong and Li, Ping and Zhang, Tong},
  journal={Journal of Machine Learning Research},
  volume={18},
  number={166},
  pages={1--43},
  year={2018}
}

@article{bertsimas2009algorithm,
  title={Algorithm for cardinality-constrained quadratic optimization},
  author={Bertsimas, Dimitris and Shioda, Romy},
  journal={Computational Optimization and Applications},
  volume={43},
  number={1},
  pages={1--22},
  year={2009},
  publisher={Springer}
}

@article{zheng2014improving,
  title={Improving the performance of MIQP solvers for quadratic programs with cardinality and minimum threshold constraints: A semidefinite program approach},
  author={Zheng, Xiaojin and Sun, Xiaoling and Li, Duan},
  journal={INFORMS Journal on Computing},
  volume={26},
  number={4},
  pages={690--703},
  year={2014},
  publisher={INFORMS}
}

@article{bertsimas2020sparse,
  title={Sparse high-dimensional regression},
  author={Bertsimas, Dimitris and Van Parys, Bart},
  journal={The Annals of Statistics},
  volume={48},
  number={1},
  pages={300--323},
  year={2020},
  publisher={JSTOR}
}

@article{bertsimas2021sparse,
  title={Sparse convex regression},
  author={Bertsimas, Dimitris and Mundru, Nishanth},
  journal={INFORMS Journal on Computing},
  volume={33},
  number={1},
  pages={262--279},
  year={2021},
  publisher={INFORMS}
}

@article{atamturk2025rank,
  title={Rank-one convexification for sparse regression},
  author={Atamturk, Alper and Gomez, Andres},
  journal={Journal of Machine Learning Research},
  volume={26},
  number={35},
  pages={1--50},
  year={2025}
}

@inproceedings{chan2007direct,
  title={Direct convex relaxations of sparse SVM},
  author={Chan, Antoni B and Vasconcelos, Nuno and Lanckriet, Gert RG},
  booktitle={Proceedings of the 24th international conference on Machine learning},
  pages={145--153},
  year={2007}
}

@article{bienstock1996computational,
  title={Computational study of a family of mixed-integer quadratic programming problems},
  author={Bienstock, Daniel},
  journal={Mathematical programming},
  volume={74},
  number={2},
  pages={121--140},
  year={1996},
  publisher={Springer}
}

@article{natarajan1995sparse,
  title={Sparse approximate solutions to linear systems},
  author={Natarajan, Balas Kausik},
  journal={SIAM journal on computing},
  volume={24},
  number={2},
  pages={227--234},
  year={1995},
  publisher={SIAM}
}

@article{sturm2003cones,
  title={On cones of nonnegative quadratic functions},
  author={Sturm, Jos F and Zhang, Shuzhong},
  journal={Mathematics of Operations research},
  volume={28},
  number={2},
  pages={246--267},
  year={2003},
  publisher={INFORMS}
}

@article{nemirovski1999maximization,
  title={On maximization of quadratic form over intersection of ellipsoids with common center},
  author={Nemirovski, Arkadi and Roos, Cornelis and Terlaky, Tam{\'a}s},
  journal={Mathematical programming},
  volume={86},
  number={3},
  pages={463--473},
  year={1999},
  publisher={Springer}
}

@article{pong2014generalized,
  title={The generalized trust region subproblem},
  author={Pong, Ting Kei and Wolkowicz, Henry},
  journal={Computational optimization and applications},
  volume={58},
  number={2},
  pages={273--322},
  year={2014},
  publisher={Springer}
}

@book{ben2001lectures,
  title={Lectures on modern convex optimization: analysis, algorithms, and engineering applications},
  author={Ben-Tal, Aharon and Nemirovski, Arkadi},
  year={2001},
  publisher={SIAM}
}

@article{bao2011semidefinite,
  title={Semidefinite relaxations for quadratically constrained quadratic programming: A review and comparisons},
  author={Bao, Xiaowei and Sahinidis, Nikolaos V and Tawarmalani, Mohit},
  journal={Mathematical programming},
  volume={129},
  number={1},
  pages={129--157},
  year={2011},
  publisher={Springer}
}

@article{pardalos1991quadratic,
  title={Quadratic programming with one negative eigenvalue is NP-hard},
  author={Pardalos, Panos M and Vavasis, Stephen A},
  journal={Journal of Global optimization},
  volume={1},
  number={1},
  pages={15--22},
  year={1991},
  publisher={Springer}
}

@inproceedings{aholt2012qcqp,
  title={A QCQP approach to triangulation},
  author={Aholt, Chris and Agarwal, Sameer and Thomas, Rekha},
  booktitle={European Conference on Computer Vision},
  pages={654--667},
  year={2012},
  organization={Springer}
}

@article{luo2010semidefinite,
  title={Semidefinite relaxation of quadratic optimization problems},
  author={Luo, Zhi-Quan and Ma, Wing-Kin and So, Anthony Man-Cho and Ye, Yinyu and Zhang, Shuzhong},
  journal={IEEE Signal Processing Magazine},
  volume={27},
  number={3},
  pages={20--34},
  year={2010},
  publisher={IEEE}
}

@article{cifuentes2021convex,
  title={A convex relaxation to compute the nearest structured rank deficient matrix},
  author={Cifuentes, Diego},
  journal={SIAM Journal on Matrix Analysis and Applications},
  volume={42},
  number={2},
  pages={708--729},
  year={2021},
  publisher={SIAM}
}

@article{haemers2022spectral,
  title={Spectral symmetry in conference matrices},
  author={Haemers, Willem H and Parsaei Majd, Leila},
  journal={Designs, Codes and Cryptography},
  volume={90},
  number={9},
  pages={1983--1990},
  year={2022},
  publisher={Springer}
}

@article{koltchinskii2017normal,
  title={NORMAL APPROXIMATION AND CONCENTRATION OF SPECTRAL PROJECTORS OF SAMPLE COVARIANCE},
  author={KOLTCHINSKII, VLADIMIR and LOUNICI, KARIM},
  journal={The Annals of Statistics},
  volume={45},
  number={1},
  pages={121--157},
  year={2017}
}

@article{candes2008restricted,
  title={The restricted isometry property and its implications for compressed sensing},
  author={Candes, Emmanuel J},
  journal={Comptes rendus. Mathematique},
  volume={346},
  number={9-10},
  pages={589--592},
  year={2008}
}

@article{yu2015useful,
  title={A useful variant of the Davis--Kahan theorem for statisticians},
  author={Yu, Yi and Wang, Tengyao and Samworth, Richard J},
  journal={Biometrika},
  volume={102},
  number={2},
  pages={315--323},
  year={2015},
  publisher={Oxford University Press}
}

@book{bonnans2013perturbation,
  title={Perturbation analysis of optimization problems},
  author={Bonnans, J Fr{\'e}d{\'e}ric and Shapiro, Alexander},
  year={2013},
  publisher={Springer Science \& Business Media}
}

@article{blekherman2022sparse,
  title={Sparse PSD approximation of the PSD cone},
  author={Blekherman, Grigoriy and Dey, Santanu S and Molinaro, Marco and Sun, Shengding},
  journal={Mathematical Programming},
  volume={191},
  number={2},
  pages={981--1004},
  year={2022},
  publisher={Springer}
}

@article{li2024sparse,
  title={On sparse canonical correlation analysis},
  author={Li, Yongchun and Dey, Santanu S and Xie, Weijun},
  journal={Advances in Neural Information Processing Systems},
  volume={37},
  pages={10707--10734},
  year={2024}
}

@article{blekherman2022hyperbolic,
  title={Hyperbolic relaxation of k-locally positive semidefinite matrices},
  author={Blekherman, Grigoriy and Dey, Santanu S and Shu, Kevin and Sun, Shengding},
  journal={SIAM Journal on Optimization},
  volume={32},
  number={2},
  pages={470--490},
  year={2022},
  publisher={SIAM}
}
	
	\begin{appendices}
		\section{Additional proofs}\label{Appendix_proof}
		
		\subsection{Proof of \Cref{thm:special_case_n3_k2}}\label{append:thm_special_case_k2}
		We start with the special case \(n=3,k=2\).
		\begin{proof}
			(\(n=3,k=2\)). Suppose $V \in \conv(\Szero{3}{2})$
			\begin{equation}
				\begin{split}
					V &=  \begin{pmatrix}
						a & b & c\\
						b & d & h\\
						c & h & f
					\end{pmatrix}\\
					& = 
					\begin{pmatrix}
						x & b & 0\\
						b & y & 0\\
						0 & 0 & 0    
					\end{pmatrix} + 
					\begin{pmatrix}
						a - x & 0 & c\\
						0 & 0 & 0\\
						c & 0 & z
					\end{pmatrix} + 
					\begin{pmatrix}
						0 & 0 & 0\\
						0 & d-y & h\\
						0 & h & f-z    
					\end{pmatrix}
				\end{split}
			\end{equation}
			with boundary
			\begin{equation}\label{boundary1}
				b^2 = xy,\quad c^2 =  (a-x)z,\quad h^2 =  (d-y)(f-z)
			\end{equation}
			where $0\leq x\leq a$, $0\leq y\leq d$, $0\leq z\leq f$.
			Then we change variables,
			\[
			x = a\sin^2\alpha,\quad y = d\sin^2\beta,\quad z = f\sin^2\gamma
			\]
			which implies that \eqref{boundary1} is equivalent to 
			\begin{equation}\label{boundary2}
				b^2 = ad\sin^2\alpha\sin^2\beta,\quad c^2 = af\cos^2\alpha\sin^2\gamma,\quad h^2 = df\cos^2\beta \cos^2\gamma
			\end{equation}
			where $\alpha,\beta,\gamma \in \RR$. In order to derive the analytical expression for \eqref{boundary2}, we consider the critical points where the determinant of Jacobian matrix is 0. The expression for \eqref{boundary2} is as follows,
			\begin{equation}\label{convex_hull_bd}
				\sin^2\alpha\cos^2\beta\cos^2\gamma = \cos^2\alpha\sin^2\beta\sin^2\gamma
			\end{equation}
			The boundary for $V\in S^d_{n,k}$ is 
			\begin{equation}\label{diag_bd}
				\begin{split}
					&\det(2\diag(V) - V) = adf - ah^2 - dc^2 - fb^2 - 2bch = 0,~~\text{when}~~ bch>0 \\
					&\det(V) =  adf - ah^2 - dc^2 - fb^2 + 2bch = 0,~~\text{when}~~ bch<0 
				\end{split}
			\end{equation}
			Then we are going to show the boundary of $\conv(Q_{3,2})$ \eqref{convex_hull_bd} and $\Szero{3}{2}$\eqref{diag_bd} are equivalent.
			
			Let $A = \sin^2\alpha\cos^2\beta\cos^2\gamma$ and $B = \cos^2\alpha\sin^2\beta\sin^2\gamma$. Then from \eqref{boundary2}, we have 
			\begin{equation}\label{eq:1}
				\begin{split}
					ah^2 + dc^2 + fb^2 &= adf\big(\sin^2\alpha\sin^2\beta + \cos^2\beta\sin^2\gamma + \cos^2 \alpha \cos^2\gamma\big)\\
					&= adf(1-2\cos^2\alpha\sin^2\beta\sin^2\gamma)\\
					& = adf(1-2B)
				\end{split}
			\end{equation}
			And
			\begin{equation}
				(bch)^2 = (adf)^2 AB =  (adfB)^2 
			\end{equation}
			which means 
			\begin{equation}\label{eq:2}
				bch = 
				\begin{cases}
					adfB, &~~\text{if}~~bch>0\\
					-adfB, &~~\text{if}~~ bch<0
				\end{cases}
			\end{equation}
			Thus, from \eqref{eq:1} and \eqref{eq:2}, it is straightforward that \eqref{diag_bd} is equivalent to $\eqref{convex_hull_bd}$.
		\end{proof}
		
		Then we move the case when \(n>k=2\). We introduce the following definition and lemma which characterizes \(\conv(Q_{n,2})\).
		\begin{dfn}
			A symmetric matrix \(X \in \SS^n\) is diagonally dominant if 
			\[
			x_{ii} \geq \sum_{j\neq i}|x_{ij}|,\quad\forall i\in [n].
			\]
			A symmetric matrix \(X\in \SS^n\) is scaled diagonal dominant if there exists a positive diagonal matrix \(D\in\RR^{n\times n}\) such that \(DXD\) is diagonally dominant, i.e., there exists a positive vector $d\in\RR^n_+$ such that
			\[
			x_{ii}d_i \geq  \sum_{j\neq i}d_j|x_{ij}|,\quad\forall i\in [n].
			\]
		\end{dfn}
		
		\begin{lem}\label{lem:sdd_factorization}
			A symmetric matrix \(X\in \SS^n\) is scaled diagonal dominant if and only if it has a factorization, \(X = VV^T\) where each column of $V$ has at most 2 nonzero entries.
		\end{lem}
		
		\begin{proof} (\(n>k=2\))
			(i): Consider \(X\in \mathcal{D}_n\cap \Szero{n}{2}\) which can be decomposed as \(X  = \mu I_n + M\) where $M$ consists of off-diagonal entries of \(X\).  Notice that \(X\in\Szero{n}{2}\) implies that \(\mu I_n \pm M \succeq 0\), i.e., \(\rho(M) \leq \mu\).
			
			From \Cref{lem:sdd_factorization}, \(X\in \mathcal{D}_n\cap \conv(Q_{n,2})\) is equivalent to that \(X\) is scaled diagonal dominant, that is, there exists a positive vector \(d\in\RR^n_+\) such that the following inequality holds entrywise,
			\[
			|M|d \leq \mu d. 
			\]
			From Perron-Frobenius Theorem, the above inequality is equivalent to
			\[
			\rho(|M|) \leq \mu.
			\]
			We claim that when $n\geq 4$, there exists \(M\in\SS^n\) such that \(\rho(M) < \rho(|M|)\). Let $\mu:= \rho(M)$, then we have \(X: = \mu I_n + M \in \Szero{n}{k}\) but \(X \notin \conv(Q_{n,2})\), which implies that \(\conv(Q_{n,2}) \subset \Szero{n}{2}\).
			
			In order to finish the proof of (i), it remains to prove the above claim. When \(n=4\), we construct 
			\[
			M_4 = \begin{pmatrix}
				0 & 1 & 1 & -1\\
				1 & 0 & -1 & 1\\
				1 & -1 & 0 & -1\\
				-1 & 1 & -1 & 0\\
			\end{pmatrix}
			\]
			and we have \(\sqrt{5} = \rho(M_4)<\rho(|M_4|) = 3\). For \(n>4\), let \(M_n\) be constructed as follows,
			\[
			M_n = \begin{pmatrix}
				M_4&0\\
				0&0
			\end{pmatrix}
			\]
			which satisfies \(\rho(M_n)<\rho(|M_n|)\).\qedhere
			
		\end{proof}
		
		\subsection{Proof of \Cref{thm:Szero_notsubset_Sone}}\label{append:thm_Szero_notsubset_Sone}
		\begin{dfn}
			A matrix $C\in \RR^{n\times n}$ is a conference matrix if it has 0 on the diagonal and \(\pm 1\) off the diagonal, and 
			\[
			C^TC = (n-1)I_n.
			\]
		\end{dfn}
		\begin{lem}[\cite{haemers2022spectral}]\label{lem:conference_matrix}
			When \(n=q+1\) where \(q\) is prime power such that \(q\equiv1(\bmod 4)\), there exists a symmetric conference matrix \(C \in \RR^{n\times n}\), and its spectral radius is \(\sqrt{n-1}\).
		\end{lem}
		
		\begin{proof}[Proof of \Cref{thm:Szero_notsubset_Sone}]
			When \(n=q+1\) where \(q\) is prime power such that \(q\equiv1(\bmod 4)\), from \Cref{lem:conference_matrix}, let \(X = \mu I_n + C\), where $C$ is a symmetric conference matrix.  By letting $\mu: = \sqrt{n-1}$, we get \(X \in \Szero{n}{k}\).  Since \(k < \sqrt{n-1}+1\), it follows that
        \[
        n\sqrt{n-1} + n(n-1) = \|X\|_1 > k \tr(X) = kn\sqrt{n-1},
        \]
       Thus, it is true that \(X \notin \Sone{n}{k}\). 
		\end{proof}

        \subsection{Proof of  \Cref{thm:extreme_ray_special}}\label{append:thm_extreme_ray_special}

	\begin{proof}[Proof of \Cref{thm:extreme_ray_special}]
        Assume first that \(X\in\Szero{n}{k}\), and let us show that it is an extreme ray.
        Suppose there exists \(X_1,X_2\in \Szero{n}{k}\) \st \(X = X_1 + X_2\). Let \(M := X_1 - u_1u_1^T =  u_2u_2^T - X_2\). Then we have
		\[
		X_1 = u_1u_1^T + M,\quad X_2 = u_2u_2^T - M
		\]
		Since \(X_1,X_2\in \Szero{n}{k}\), we have \(X_1,X_2\succeq 0\) which implies that \(M =-q_1 u_1u_1^T + q_2u_2u_2^T\) where \(0\leq q_1,q_2\leq 1\). Then we have
		\[
		X_1 = (1-q_1) u_1u_1^T + q_2u_2u_2^T,\quad X_2 = q_1u_1u_1^T + (1-q_2) u_2u_2^T 
		\]
		If \(1>q_1 + q_2\), let \(v:= u_1\), then we have \(v^T(k\diag(X_1)-X_1)v<0\), that is, \(X_1\notin \Szero{n}{k}\); If \(1<q_1 + q_2\), let \(v:= u_2\), then we have \(v^T(k\diag(X_2)-X_2)v<0\), that is, \(X_2\notin \Szero{n}{k}\); 
		If \(1 = q_1 + q_2\), then \(X_1,X_2\) are multiples of \(X\).   Thus, \(X\) is an extreme ray in \Szero{n}{k}. 
		
		It remains to prove that \(X\in\Szero{n}{k}\), which is equivalent to 
		\begin{equation}\label{eq:extreme_ray_opt}
			\min_{v\in\RR^n, \|v\|=1} v^T(k\diag(X)-X)v \geq 0
		\end{equation}
        The solution of \eqref{eq:extreme_ray_opt} can be expressed as \(v = \lambda_1x + \lambda_2y\) with \(\|x\|=\|y\|=1\), where \(x\in \langle u_1,u_2\rangle\) and \(y\in \langle u_1,u_2\rangle^\perp\).
		We begin with proving \(x^T(k\diag(X)-X)x=0\). Suppose \(x=\mu_1u_1 + \mu_2u_2\) where \(\mu_1^2 + \mu_2^2=1\) and \(\langle u_1,u_2\rangle=0\), then we have
        \begin{equation}
            \begin{split}
                & x^T(k\diag(u_1u_1^T)-u_1u_1^T)x \\ =
                & \mu_1^2\Big(k\sum_{i=1}^n (u_1)_i^4 - 1\Big) + \mu_2^2 \Big(k\sum_{i=1}^n (u_1)_i^2(u_2)_i^2 - \langle u_1,u_2\rangle^2\Big)\\
                & + 2\mu_1\mu_2\Big(k\sum_{i=1}^n (u_1)^2_i(u_2)_i - \langle u_1,u_2\rangle\Big)
            \end{split}
        \end{equation}
        and it is also true that
         \begin{equation}
            \begin{split}
                & x^T(k\diag(u_2u_2^T)-u_2u_2^T)x \\ =
                & \mu_2^2\Big(k\sum_{i=1}^n (u_2)_i^4 - 1\Big) 
                 + \mu_1^2\Big(k\sum_{i=1}^n (u_1)_i^2(u_2)_i^2 - \langle u_1,u_2\rangle^2\Big)\\
                 & + 2\mu_1\mu_2\Big(k\sum_{i=1}^n (u_1)_i(u_2)_i^3 - \langle u_1,u_2\rangle\Big)
            \end{split}
        \end{equation}
        From \eqref{eq:extreme_ray_special}, it is true that 
        \(\sum_{i=1}^n(u_1)_i(u_2)_i\big((u_1)^2_i+(u_2)^2_i\big) = 0\), which implies that
        \begin{equation}
            \begin{split}
                x^T(k\diag(X)-X)x &= k\sum_{i=1}^n\big(\mu_1^2(u_1)_i^4+(u_1)_i^2(u_2)_i^2+\mu_2^2(u_2)_i^4\big)-1\\
                &= k\sum_{i=1}^n\big(\tfrac{1}{2}(u_1)_i^4+(u_1)_i^2(u_2)_i^2+\tfrac{1}{2}(u_2)_i^4\big)-1\\
                &= \frac{k}{2}\sum_{i=1}^n\big((u_1)_i^2+(u_2)_i^2\big)^2-1 = 0
            \end{split}
        \end{equation}
        where the second equality follows from \(\sum_{i=1}^n(u_1)_i^4 = \sum_{i=1}^n(u_2)_i^4\) in \eqref{eq:extreme_ray_special}. Then we will show that \(x^T(k\diag(X)-X)y=0\). Observe that
        \begin{equation}
            \begin{split}
                 x^T(k\diag(X)-X)y 
                &= k\sum_{i=1}^n \big((u_1)_i^2 + (u_2)_i^2\big) x_iy_i \\
                &= k \big((u_1)_i^2 + (u_2)_i^2\big) \sum_{i=1}^n x_iy_i = 0
            \end{split}
        \end{equation}
        where the first equality follows from \(\langle u_1,y\rangle=\langle u_2,y\rangle=0\), and the second equality holds because \((u_1)_i^2 + (u_2)_i^2 = (u_1)_j^2 + (u_2)_j^2, \forall i,j\in [n]\).
        It follows that
        \begin{equation}
            \begin{split}
                 v^T (k\diag(X)-X)v =  k \lambda_2^2 \big(y^T\diag(X)y\big)\geq0 
            \end{split}
        \end{equation}
        Thus, the minimum of \eqref{eq:extreme_ray_opt} is obtained when \(\lambda_2=0\), that is, \(v\in \langle u_1,u_2\rangle\), which implies that \(X \in \Szero{n}{k}\).
	\end{proof}
	
		\subsection{Proof of \Cref{thm:extreme_ray_of_ourset_dual}}\label{append:thm_extreme_ray_dual}
		\begin{proof}
			Let \(X = xx^T\) where \(\|x\|_0\leq k\). Without loss of generality, let \(S := [k]\) be the support of \(x\) and let \(S_j := [k-1] \cup \{j\}, \forall j \in [n] \setminus [k]\). We will show that \(X\) is an extreme ray of \((\conv(Q_{n,k}))^*\).
			
			If there exists \(X_1,X_2 \in (\conv(Q_{n,k}))^*\) and \(\lambda_1,\lambda_2 > 0\) such that \(X = \lambda_1X_1 + \lambda_2X_2\), then we have 
			\begin{align}
				X_{S} &= \lambda_1(X_1)_{S} + \lambda_2(X_2)_{S}\\
				X_{S_j} &= \lambda_1(X_1)_{S_j} + \lambda_2(X_2)_{S_j}, \forall j \in  [n] \setminus [k]
			\end{align}
			Since for \(i=1,2\), \((X_i)_{S_j}\succeq 0, \forall j \in  [n] \setminus [k]\), we have
			\begin{equation}\label{eq:extreme_ray_1}
				0 = X_{jj} = (X_1)_{jj} = (X_2)_{jj} , \forall j  \in  [n] \setminus [k].
			\end{equation}
			Then \eqref{eq:extreme_ray_1} implies that \((X_1)_{ij} = (X_2)_{ij} = 0, \forall (i,j) \notin [k] \times [k]\). Notice that \(\rank(X_{S})=1\) and \((X_1)_{S},(X_2)_{S}\) are positive semidefinite. Thus, we have \(X_1,X_2\) are multiples of \(X\), and \(X\) is an extreme ray of \((\conv(Q_{n,k}))^*\).
		\end{proof}

        \subsection{Proof of \Cref{lem:primal_convergence}}\label{append:lem_primal_convergence}
        We begin with the following lemma, see \cite[Prop.4.4]{bonnans2013perturbation}.

        \begin{lem}\label{lem:continuity}
            Let \(F:S\times\Theta\rightarrow\RR\) be a continuous function, where \(S\in\RR^N\) is a compact set. Then the function \(f:\Theta \rightarrow\RR\) such that \(\theta \mapsto \min_{x\in S} F(x,\theta)\) is continous.
        \end{lem}
        
        \begin{proof}[Proof of \Cref{lem:primal_convergence}]
        Without loss of generality, we assume that \(\bar Q = \bar C\), that is, \(\bar\lambda = \vzero\). For \(\{A_i,b_i\}_{i\in[m]}\) at least one \(b_i \neq 0\), we assume that \(b_1=1\). Let \(S:=\supp(\bar x)\). Decompose \(x_S = x' + y\) where \(x'\in \langle\bar x_S\rangle\) and \( y \in \langle\bar x_S\rangle^\perp\). Let \(u:= (A_1)_Sy\), then we have
        \begin{equation}
        \begin{split}
             x_S^T (A_1)_S x_S &= (x')^T(A_1)_S x' + y^T(A_1)_Sy + 2 y^T (A_1)_S x' \\
                       &= \|x'\|^2 + u^T y + 2 u^T x' = 1
        \end{split}
        \end{equation}
        which implies that \( \|x' + u\|^2 = 1 + \|u\|^2 - u^Ty \). It follows that
        \begin{equation}
            \begin{split}
                \|x'\| \leq \|u\| + \sqrt{1 + \|u\|^2 - u^Ty}\leq 1 + \alpha\|y\|
            \end{split}
        \end{equation}
        where \(\alpha:= \|(A_1)_S\| + \|(A_1)_S\big((A_1)_S-I_k\big)\|^{\tfrac{1}{2}}\). 
        
        As \(C\) is sufficiently close to \(\bar C\), we claim that the \(x_S\) belongs a compact set \(V:=\{x_S=x'+y\in\RR^k: x'\in \langle \bar x_S\rangle, y\in  \langle \bar x_S\rangle^\perp, \|y\|\leq 1, \|x'\|\leq 1 + \alpha\|y\|\}\). Suppose \(c\) is the second smallest eigenvalue of \(\bar C_S\) and \(\|C_S - \bar C_S\| \leq \min\{\tfrac{1}{(1+\alpha)^2},\tfrac{1}{\|\bar x\|^2}\}\tfrac{c}{4}\). If \(y \notin V\), it follows that
        \begin{equation}
        \begin{split}
            x_S^T C_S x&\geq  x_S^T \bar C_S x_S - |x_S^T \bar C_S x_S - x_S^T C_S x_S|\\
                   \geq & c\|y\|^2 - \|C_S - \bar C_S\|\|x_S\|^2\\
                   \geq & \big(c - \|C_S - \bar C_S\|\big)\|y\|^2 - \|C_S - \bar C_S\|(1+\alpha \|y\|)^2\\
                   \geq & c/2
        \end{split}
        \end{equation}
        Such a point \(x\) cannot be optimal, because we have
        \begin{equation}
            \begin{split}
                \bar x_S^T C_S \bar x_S \leq \|C_S - \bar C_S\| \|\bar x_S\|^2 \leq c/4.
            \end{split}
        \end{equation}
        We prove that \(x \in V\) which is a compact set. From \Cref{lem:continuity}, it is true that \(x_SC_Sx_S\rightarrow \bar x_S\bar C_S\bar x_S = 0\), as \(C_S\rightarrow \bar C_S\). Then we obtain
        \[
        c\|y\|^2 = x_S^T \bar C_S x_S \leq  x_S^T  C_S x_S  + \|C_S - \bar C_S\|\|x_S\|^2 \rightarrow 0, \;\text{as}\; C_S\rightarrow \bar C_S
        \]
        Thus, we have \(x\rightarrow \bar x\) as \(C \rightarrow \bar C\), since \(x_i = \bar x_i = 0, \forall i  \in [n]\setminus S\).
    \end{proof}

		\subsection{Proof of \Cref{lem:rank_one_case}}\label{append:lem_rank_one_case}
		\begin{proof}
			Suppose $|\sigma_1|\geq|\sigma_2|>\cdots\geq |\sigma_n|$. Let \(\sigma_{(k)} =(\sigma_1,\dots,\sigma_k)\), and \(\sigma_{(n-k)} =(\sigma_{k+1},\dots,\sigma_n)\). Let $\bar u(\tau) = (u(\tau);\tilde u(\tau))=\sqrt{\frac{w^{(k)}}{k}}(\frac{\tau}{\sigma_1},\cdots,\frac{\tau}{\sigma_k};\frac{\sigma_{k+1}}{\tau},\cdots,\frac{\sigma_{n}}{\tau})$ where $w^{(k)} = \sum_{i=1}^k\sigma_i^2$,. Then we are going to show that there exits $\tau^* \in (|\sigma_{k+1}|,|\sigma_{k}|)$ such that 
			\[
			\lambda_{max} \Big(k\diag(\bar u(\tau^*)\bar u(\tau^*)^T) - \bar u(\tau^*)\bar u(\tau^*)^T + \Sigma\Big) = w^{(k)}.
			\]
			Let $V = k\diag(\bar u(\tau^*)\bar u(\tau^*)^T) - \bar u(\tau^*)\bar u(\tau^*)^T + \Sigma$. It follows that $s = (\sigma_1,\cdots,\sigma_k,0,\cdots,0)^T$ satisfies
			\[
			Vs = w^{(k)}s, \; \text{and} \;(V - ss^T) s = 0.
			\]
			Thus, it remains to show that 
			\[
			\lambda_{max}(V - ss^T) \leq w^{(k)}
			\]
			Consider
			\[
			V  - ss^T = \begin{pmatrix}
				W_{11}&W_{12}\\
				W_{21}&W_{22}
			\end{pmatrix}
			\]
			and 
			\[
			w^{(k)} I_n - (V  - ss^T) = \begin{pmatrix}
				w^{(k)} I_k- W_{11}&-W_{12}\\
				-W_{21}& w^{(k)} I_{n-k} - W_{22}
			\end{pmatrix}
			\]
			From the structure of \(\bar u(\tau)\), we have
			\[
			w^{(k)} I_{n-k}- W_{22} \succ 0
			\]
			In order to show that
			\[
			w^{(k)} I_n - (V  - ss^T)\succeq 0  
			\]
			it is equivalent to show
			\begin{equation}\label{eq:lem_4_1}
			    (w^{(k)} I_{n-k}- W_{22}) - W_{21}\Big(w^{(k)} I_k - W_{11}\Big)^{-1}W_{12}^T\succeq 0
			\end{equation}
			Since $w^{(k)} I_k - W_{11}$ is a diagonal matrix plus a rank-1 matrix, we can get its inverse analytically, that is, 
			\begin{equation}
				\begin{split}
					(w^{(k)}I_k - W_{11})^{-1} &= \Big(D+u(\tau)u(\tau)^T\Big)^{-1}\\
					&= D^{-1} - \frac{D^{-1}u(\tau)u(\tau)^TD^{-1}}{1 + u(\tau)^TD^{-1}u(\tau)}
				\end{split}
			\end{equation}
			where $D=w^{(k)}I_k - k\diag(u(\tau)u(\tau)^T)$. Then we have
			\begin{equation}
				\begin{split}
					&W_{21}\Big(w^{(k)} I_k - W_{11}\Big)^{-1}W_{12}^T \\
					&=\Big( \hat\sigma^TD^{-1}\hat\sigma- \frac{\gamma^2}{1 + u(\tau)^TD^{-1}u(\tau)}\Big)\sigma_{(n-k)}\sigma_{(n-k)}^T
				\end{split}
			\end{equation}
			where $\hat \sigma_i = \sigma_i - \frac{w^{(k)}}{k}\frac{1}{\sigma_i}, i=1,\dots,k$, and  
			\[
			\gamma = \hat\sigma^TD^{-1}u(\tau) = \sqrt{\frac{1}{w^{(k)}k}}\sum_{i=1}^k\frac{\big(\tau - \frac{w^{(k)}}{k}\frac{\tau}{\sigma_i^2}\big)}{1 - \frac{\tau^2}{\sigma_i^2}}.
			\]
			  We claim that  
			as $\tau \rightarrow |\sigma_k|$, 
			\begin{equation}\label{eq:psd_condition}
				\begin{split}
					&(w^{(k)} I_{n-k}- W_{22}) - W_{21}\Big(w^{(k)} I_k - W_{11}\Big)^{-1}W_{12}^T\\
					&= \tilde D + \eta\sigma_{(n-k)}\sigma_{(n-k)}^T \succeq 0
				\end{split}
			\end{equation}
			where $\eta = \frac{w^{(k)}}{k\tau}-1-\big( \hat\sigma^TD^{-1}\hat\sigma- \frac{\gamma^2}{1 + u(\tau)^TD^{-1}u(\tau)}\big)$, and
			\[
			\tilde D = w^{(k)}\begin{pmatrix}
				1 - \frac{\sigma_{k+1}^2}{\tau^2} & & &\\
				&1 - \frac{\sigma_{k+2}^2}{\tau^2}& & \\
				&&\ddots&\\
				&&&1 - \frac{\sigma_{n}^2}{\tau^2}
			\end{pmatrix}\succ 0
			\]
			In order to prove the claim, it is enough to show that $\eta\geq0$ as $\tau \rightarrow |\sigma_k|$. Observe that
			\begin{equation}
				\begin{split}
					\hat\sigma D^{-1}\hat\sigma = \frac{1}{w^{(k)}} \sum_{i=1}^k\frac{(\sigma_i - \frac{w^{(k)}}{k}\frac{1}{\sigma_i})^2}{1 - \frac{\tau^2}{\sigma_i^2}}
				\end{split}
			\end{equation}
			Without loss of generality, we assume $|\sigma_k|=1$ and then it follows that
			\begin{equation}
				\begin{split}
					&\lim\limits_{\tau\rightarrow1} \Big(\hat\sigma D^{-1}\hat\sigma - \frac{\gamma^2}{1 + u(\tau)^TD^{-1}u(\tau)}\Big)\\
					&=\frac{1}{w^{(k)}}\sum_{i=1}^{k-1} \frac{\sigma_i^2 + \frac{1}{\sigma_i^2} - 2}{1-\frac{1}{\sigma_i^2}} + \frac{k}{w^{(k)}}\big(1 - \frac{w^{(k)}}{k}\big)^2\\
					&=w^{(k)} - k + \frac{k}{w^{(k)}}\big(1 - \frac{w^{(k)}}{k}\big)^2\\
					& =  \frac{w^{(k)}}{k}-1
				\end{split}
			\end{equation}
			Thus, we have $\eta\geq 0$ as $\tau\rightarrow 1$ which implies that \eqref{eq:psd_condition} holds. 
		\end{proof}
		
		\subsection{Dual cones}\label{append_dual_cones} In this section, we will derive the dual cones of \(\Szero{n}{k}, \Sone{n}{k}\), and \(\conv(Q_{n,k})\). Let \(\mathcal{L}: \SS^n \mapsto \SS^n, \mathcal{L}(Z) = k\diag(Z)-Z\). Notice that \(\mathcal{L} = \mathcal{L}^*\), we have
		\[
		(\Szero{n}{k})^* = (\SS^n_+ \cap \mathcal{L}^{-1}(\SS^n_+))^* = \overline{\SS^n_+ + \mathcal{L}^*(\SS^n_+)} = \overline{\SS^n_+ + \mathcal{L}(\SS^n_+)}
		\]
		which is equivalent to
		\[
		(\Szero{n}{k})^* = \{Y + k\diag(Z) - Z: Y \succeq 0, Z \succeq 0\}.
		\]
		
		There is an equivalent formulation for \Sone{n}{k}:
		\[
		\Sone{n}{k} = \SS^n_+\cap \mathcal{C}_1 
		\]
		where \(\mathcal{C}_1 := \{X\in\SS^n: (kI_n - Z) \bullet X \geq 0, \|Z\|_\infty \leq 1, Z\in \SS^n\}\). Notice that
		\[
		(\mathcal{C}_1 )^* = \{ kI_n - Z: \|Z\|_\infty \leq 1, Z\in \SS^n\}
		\]
		Then we have
		\[
		(\Sone{n}{k})^* = \overline{\SS^n_+ + (\mathcal{C}_1 )^*}
		\]
		which is equivalent to
		\[
		(\Sone{n}{k})^* = \{Y + \alpha kI_n - Z: \alpha\geq 0, Y \succeq0, Z\in\SS^n, \|Z\|_\infty \leq \alpha\}.
		\]
		Then we discuss \(\conv(Q_{n,k})^*\). From the definition of dual cone, we have
		\begin{equation}
			\begin{split}
				\conv(Q_{n,k})^* &= \{Y\in\SS^n: Y \bullet X \geq 0, \forall X\in \conv(Q_{n,k})\}\\
				&=\{Y\in\SS^n: Y_{S}\succeq 0, |S|\leq k \}
			\end{split}
		\end{equation}
		
		\subsection{Proof of \Cref{lem:Sz_rank_one}}\label{append:lem_Sz_rank_one}
		\begin{proof}
			Suppose that \(\|x\|_0=p\), and \(\supp(x)=[p]\). From the definition of \Sz{n}{k}, we have
			\begin{equation}\label{eq:extreme_ray_Sint}
				\|X_{i,:}\|^2 = x_i^2 \big(\sum_{j=1}^p x_i^2\big) = X_{ii} \tr(X) \leq X_{ii}z_i,\;\forall i\in [n]
			\end{equation}
			Since \(X_{ii}\geq 0, z_i\geq 0,\;\forall i\in [n]\), \eqref{eq:extreme_ray_Sint} forces that \(z_i = \tr(X),\;\forall i\in [p]\), and \(z_i = 0,\;\forall i\in [n]\setminus[p]\), which implies that 
			\[
			\sum_{i=1}^nz_i = p\tr(X)
			\]
			If \(p>k\), it would violate the constraint on the summation of \(z\) in \Sint{n}{k}. Thus, if \(X = xx^T\in \Sint{n}{k}\), then we have \(\|x\|_0 \leq k\). 
		\end{proof}
		
		\subsection{Proof of \Cref{thm:Sz_perturbation}}\label{append:thm_Sz_perturbation}
		We start with the following cone, 
		\small{\begin{equation}
				\begin{split}
					\Sint{n}{k} = \{X: \;\exists z \in \RR^n \st
					&\|X_{i,:}\|^2\leq X_{ii}z_i,\;0\leq z_i\leq \tr(X), i\in[n]\\
					&\sum_{i=1}^n z_i \leq k \tr(X),X\succeq 0\}
				\end{split}
		\end{equation}}
        
        \begin{lem}\label{lem:Sint}
		Let \(2k<n\), and let $x \in \RR^n$ with $\|x\|=1$, \(\|x\|_0=k\) and not all its non-zero coordinates identical.
        Then
		\begin{enumerate}
			\item 
            \(x x^T + \epsilon \, w w^T \notin \Szero{n}{k}\)
            for any $w \!\in\! \RR^n$, $\|w\|=1$ ,$\|w\|_0\!=\!n$ and any $\epsilon \!>\! 0$.
			\item
            Suppose that \(w\) with $\|w\|=1$, \(\|w\|_0=n\), satisfies \(\sum_{i\in\supp(x)}(1 - |\tfrac{w_i}{x_i}|)^2 > n - k\), then we have
            \(x x^T + \epsilon \, w w^T \in \Sint{n}{k}\)
            $\|w\|_0 \!=\! n$ for small enough $\epsilon \!>\! 0$.
		\end{enumerate}
	\end{lem}
		\begin{proof}
			Throughout the proof, we assume \( \supp(x) = [k]\). We begin with the first part, i.e., \(U := x x^T + \epsilon \, w w^T \notin \Szero{n}{2}\). Let \(X = xx^T\) and \(W = ww^T\). Define \(\cD (U): \SS^n\mapsto\SS^n, U\mapsto k\diag(U)-U\). It is enough to show that \(\cD(U)=\cD(X) + \epsilon\cD(W)\) has at least one negative eigenvalue.  Let \(y = (\vzero,w_{k+1}^{-1},\dots,w_{n}^{-1})\), then we have
			\begin{equation}
				y^T \cD(U) y = \epsilon y^T \cD(W) y = \epsilon (2k-n)(n-k)<0
			\end{equation}
			which implies that \(\cD(W)\) has at least one negative eigenvalue, that is, \(U\notin \Szero{n}{k}\).
			
			For the second part, for any \(i \in [k]\) we observe that 
			\begin{equation}
				\begin{split}
					\sum_{j=1}^n (x_ix_j+ \epsilon w_iw_j)^2
					= & x^2_i  + \epsilon^2w^2_i  + 2\epsilon x_iw_i \langle x,w\rangle\\
					< & x^2_i  + \epsilon^2w^2_i  + 2\epsilon x_iw_i\\
					= & \Big(x_i + \epsilon w_i \Big)^2
                    \leq  (\tr(X)-\delta_i)X_{ii}
				\end{split}
			\end{equation}
            where \(\delta_i \in [0, \epsilon\tfrac{(x_i - w_i)^2}{x_i^2 + \epsilon w_i^2}]\). We also have that
			\begin{equation}
				\begin{split}
					\sum_{j=1}^n (\epsilon w_iw_j)^2
					= & \epsilon^2 \cdot w^2_i 
					\leq \delta'_i \cdot \epsilon w^2_i
				\end{split}
			\end{equation}
            where \(\delta'_i \geq \epsilon\). Since we have 
            \[
            \sum_{i\in\supp(x)}\big(1-|\tfrac{w_i}{x_i}|\big)^2>n-k,
            \]
            there exists \(\{\delta_i\}_{i\in[k]}\) and \(\{\delta'_i\}_{i\in[n]\setminus[k]}\) such that
			\begin{equation}
				\sum_{i\in [k]} \delta_i = \sum_{i\in [n]\setminus [k]} \delta'_i
			\end{equation}
			for any \(\epsilon\) satisfies
            \[
            \sum_{i\in\supp(x)} \frac{(|x_i| - |w_i|)^2}{x_i^2 + \epsilon w_i^2} \geq n - k
            \]
            Thus, it is true that \(X \in \Sint{n}{k}\) 
		\end{proof}
		
		Then we proceed to the proof of \Cref{thm:Sz_perturbation}.
		\begin{proof}[Proof of \Cref{thm:Sz_perturbation}]
			Throughout the proof, we assume \( \supp(x_1) = [k]\). From \Cref{lem:Sint}, \eqref{eq:compare_ourset} is true. We start with the proof of \eqref{eq:compare_Sz}. For any \(i \in [k]\) we observe that
			\begin{equation}\label{eq:compare_Sz_1}
				\begin{split}
					\Big(\sum_{j=1}^n|X_{ij}|\Big)^2 &=  \Big(\sum_{j=1}^n|x_ix_j + \epsilon w_iw_j|\Big)^2\\
					&\leq \Big(\sum_{j=1}^n|x_ix_j| + |\epsilon w_iw_j|\Big)^2\\
					&=\Big(|x_i|\|x\|_1+\epsilon|w_i|\|w\|_1\Big)^2 \\
					&\leq  \big(\|x\|_1^2 + \epsilon\|w\|_1^2\big) \big(x_i^2 + \epsilon w_i^2 \big)\\
				\end{split}
			\end{equation}
			Since we have \(\|x\|_1 \leq \sqrt{k}\|x\|\),  then \eqref{eq:compare_Sz_1} implies 
			\begin{equation}
				\begin{split}
					\Big(\sum_{j=1}^n|X_{ij}|\Big)^2 &\leq  k \big(1 + \epsilon - \delta_i\big) \big(x_i^2 + \epsilon w_i^2 \big)\\
					&= k (\tr(X)-\delta_i) X_{ii}
				\end{split}
			\end{equation}
             where \(\delta_i \in [0, \epsilon\tfrac{(|x_i| - |w_i|)^2}{x_i^2 + \epsilon w_i^2}]\). Then for any \(i \in [n]\setminus [k]\), there exists \(\delta'_i>0\) such that 
			\begin{equation}
				\begin{split}
					\Big(\sum_{j=1}^n|X_{ij}|\Big)^2 & =   \epsilon^2\|w\|_1^2\cdot(x_2)_i^2 \leq \delta'_i \epsilon \cdot k (w)_i^2 = \delta'_i \epsilon \cdot k X_{ii}
				\end{split}
			\end{equation}
            where \(\delta_i'\geq \epsilon\). Since we have 
            \[
            \sum_{i\in\supp(x)}\big(1-|\tfrac{w_i}{x_i}|\big)^2>n-k,
            \]
            there exists \(\{\delta_i\}_{i\in[k]}\) and \(\{\delta'_i\}_{i\in[n]\setminus[k]}\) such that
			\begin{equation}
				\sum_{i\in [k]} \delta_i = \sum_{i\in [n]\setminus [k]} \delta'_i
			\end{equation}
			for any \(\epsilon\) satisfies
            \[
            \sum_{i\in\supp(x)} \frac{(|x_i| - |w_i|)^2}{x_i^2 + \epsilon w_i^2} \geq n - k
            \]
            Thus, it is true that \(X \in \Sz{n}{k}\) 
		\end{proof}
		
	\end{appendices}

\end{document}